\documentclass{amsart}
\usepackage{amsmath}
\usepackage{amssymb}
\usepackage[all]{xy}
\usepackage{graphicx}
\usepackage{mathrsfs}

\newcommand{\C}{\mathbb{C}}
\newcommand{\F}{\mathbb{F}}
\newcommand{\Fq}{\F_q}
\newcommand{\Fqb}{\overline{\mathbb{F}}_q}
\newcommand{\Qlb}{\overline{\mathbb{Q}}_\ell}
\newcommand{\fK}{\mathfrak{K}}
\newcommand{\fO}{\mathfrak{O}}
\newcommand{\Z}{\mathbb{Z}}

\newcommand{\cO}{\mathcal{O}}
\newcommand{\Fr}{\mathsf{Fr}}
\newcommand{\Gm}{\mathbb{G}_m}
\newcommand{\bA}{\mathbb{A}}

\newcommand{\cS}{\mathscr{S}}
\newcommand{\cT}{\mathscr{T}}
\DeclareMathOperator{\Spec}{Spec}
\newcommand{\pt}{\mathrm{pt}}

\newcommand{\scA}{\mathscr{A}}
\newcommand{\scD}{\mathscr{D}}

\newcommand{\scM}{\mathscr{M}}
\newcommand{\Irr}{\mathrm{Irr}}
\newcommand{\Db}{D^{\mathrm{b}}}
\newcommand{\Kb}{K^{\mathrm{b}}}

\newcommand{\Proj}{\mathsf{Proj}}
\newcommand{\Inj}{\mathsf{Inj}}
\newcommand{\Pure}{\mathsf{Pure}}
\newcommand{\Tilt}{\mathsf{Tilt}}
\newcommand{\Kos}{\mathsf{Kos}}
\newcommand{\Orl}{\mathsf{Orl}}
\newcommand{\op}{\mathrm{op}}

\newcommand{\pse}{\Im}
\newcommand{\pg}{\varpi}
\newcommand{\incl}{\iota}
\newcommand{\itr}{\varrho}
\newcommand{\iinf}{\upsilon}
\newcommand{\degr}{\zeta}

\newcommand{\Gv}{{\check G}}

\newcommand{\Iv}{{\check I}}
\newcommand{\Jv}{{\check J}}
\newcommand{\GvO}{{\Gv(\fO)}}

\newcommand{\Fl}{\mathsf{Fl}}
\newcommand{\Gr}{\mathsf{Gr}}
\newcommand{\tcN}{{\widetilde{\mathcal{N}}}}
\newcommand{\cN}{{\mathcal{N}}}

\newcommand{\Weil}{\mathrm{Weil}}
\newcommand{\mix}{\mathrm{mix}}

\newcommand{\misc}{\mathrm{misc}}
\newcommand{\ddel}{\cD^{\mathrm{Del}}}
\newcommand{\dw}{\cD^\Weil}
\newcommand{\dm}{\cD^\mix}

\newcommand{\dmisc}{\cD^\misc}

\newcommand{\dc}{\cD}
\newcommand{\dsw}{\cD_\cS^\Weil}
\newcommand{\dsm}{\cD_\cS^\mix}
\newcommand{\tD}{{}^{\mathrm{t}}\cD}
\newcommand{\dsmisc}{\cD_\cS^\misc}
\newcommand{\ds}{\cD_\cS}

\newcommand{\Pervm}{\mathsf{P}^\mix}
\newcommand{\Pervw}{\mathsf{P}^\Weil}
\newcommand{\Pervs}{\mathsf{P}_\cS}
\newcommand{\Pervsw}{\mathsf{P}_\cS^\Weil}
\newcommand{\Pervsm}{\mathsf{P}_\cS^\mix}
\newcommand{\Purew}{\mathsf{Pure}^\Weil}
\newcommand{\exs}{\varkappa}

\DeclareMathOperator{\real}{\mathsf{real}}
\newcommand{\IC}{\mathrm{IC}}
\newcommand{\ICm}{\mathrm{IC}^\mix}
\newcommand{\Dm}{\Delta^\mix}
\newcommand{\Nm}{\nabla^\mix}
\newcommand{\Tm}{\mathrm{T}^\mix}
\newcommand{\Tw}{\mathrm{T}}
\newcommand{\uQlb}{{\underline{\overline{\mathbb{Q}}}_\ell}}

\newcommand{\pH}{{}^p\! H}

\newcommand{\Coh}{\mathsf{Coh}}
\newcommand{\Cohgm}{\Coh^{G \times \Gm}}
\newcommand{\Mod}{\text{-}\mathsf{Mod}}
\newcommand{\free}{{\mathrm{free}}}
\newcommand{\Dfr}{\Db_\free}

\newcommand{\cF}{\mathcal{F}}
\newcommand{\cG}{\mathcal{G}}
\newcommand{\cK}{\mathcal{K}}
\newcommand{\cL}{\mathcal{L}}

\newcommand{\cA}{\mathcal{A}}

\newcommand{\simto}{\overset{\sim}{\longrightarrow}}
\newcommand{\ovto}[1]{\overset{#1}{\longrightarrow}}
\newcommand{\xto}[1]{\xrightarrow{#1}}

\newcommand{\ttimes}{\mathrel{\widetilde{\times}}}
\newcommand{\tboxtimes}{\mathrel{\widetilde{\boxtimes}}}
\newcommand{\cRHom}{R\mathcal{H}\!\mathit{om}}
\newcommand{\la}{\langle}
\newcommand{\ra}{\rangle}

\newcommand{\tlimind}{\mathop{2\text{-}\mathrm{lim\ ind}}\limits}
\DeclareMathOperator{\im}{im}
\newcommand{\End}{\mathrm{End}}
\newcommand{\Lotimes}{\mathchoice%
  {\overset{\scriptscriptstyle L}{\otimes}}%
  {\otimes^{\scriptscriptstyle L}}{\otimes^L}{\otimes^L}}

\newcommand{\D}{\mathbb{D}}

\newcommand{\cD}{\mathcal{D}}

\newcommand{\cW}{\mathcal{W}}

\newcommand{\id}{\mathrm{id}}

\newcommand{\scB}{\mathscr{B}}
\newcommand{\scC}{\mathscr{C}}
\newcommand{\scE}{\mathscr{E}}

\newcommand{\cC}{\mathcal{C}}

\DeclareMathOperator{\Hom}{Hom}
\DeclareMathOperator{\Ext}{Ext}
\DeclareMathOperator{\Ind}{Ind}

\DeclareMathOperator{\uHom}{\underline{Hom}}
\DeclareMathOperator{\uEnd}{\underline{End}}
\DeclareMathOperator{\uExt}{\underline{Ext}}
\DeclareMathOperator{\uRHom}{\mathit{R}\underline{Hom}}

\DeclareMathOperator{\wt}{wt}
\DeclareMathOperator{\gr}{gr}
\DeclareMathOperator{\cok}{cok}
\DeclareMathOperator{\rad}{rad}

\DeclareMathOperator{\supp}{supp}

\def\eq#1{\text{\rm $#1$-eq}}

\newtheorem{thm*}{Theorem}
\numberwithin{equation}{section}
\newtheorem{thm}{Theorem}[section]
\newtheorem{lem}[thm]{Lemma}
\newtheorem{prop}[thm]{Proposition}
\newtheorem{cor}[thm]{Corollary}
\theoremstyle{definition}
\newtheorem{defn}[thm]{Definition}

\theoremstyle{remark}
\newtheorem{rmk}[thm]{Remark}

\title{Koszul duality and semisimplicity of Frobenius}

\author{Pramod N. Achar}
\address{Department of Mathematics\\
   Louisiana State University\\
   Baton Rouge, LA 70803\\
   USA}
\email{pramod@math.lsu.edu}
\author{Simon Riche}
\address{Clermont Universit{\'e}, Universit{\'e} Blaise Pascal, Laboratoire de  
Math{\'e}ma\-tiques, BP 10448, F-63000 Clermont-Ferrand. \newline
\indent CNRS, UMR 6620, Laboratoire de Math{\'e}matiques, F-63177 Aubi{\`e}re.
}
\email{simon.riche@math.univ-bpclermont.fr}

\begin{document}

\begin{abstract}
A fundamental result of Beilinson--Ginzburg--Soergel states that on flag varieties and related spaces, a certain modified version of the category of $\ell$-adic perverse sheaves exhibits a phenomenon known as \emph{Koszul duality}.  The modification essentially consists of discarding objects whose stalks carry a nonsemisimple action of Frobenius.  In this paper, we prove that a number of common sheaf functors (various pull-backs and push-forwards) induce corresponding functors on the modified category or its triangulated analogue.  In particular, we show that these functors preserve semisimplicity of the Frobenius action.
\end{abstract}

\maketitle

\section{Introduction}
\label{sect:intro}

Let $X$ be a variety over a finite field $\F_q$.  In Deligne's work on the Weil conjectures~\cite{del1,del2}, a central role is played by the category of ``mixed constructible complexes of $\Qlb$-sheaves'' on $X$, denoted $\dw(X)$ in the present paper.  (Henceforth, we will avoid calling this category ``mixed,'' as that conflicts with the terminology of~\cite{bgs}.) In order to belong to $\dw(X)$, a complex $\cF$ must have the property that the eigenvalues of the Frobenius action on stalks of $\cF$ at $\F_{q^n}$-points of $X$ are of a certain form.  One of the main results of Deligne's work states that this constraint on eigenvalues of Frobenius is preserved by all the usual sheaf operations~\cite[\S6.1]{del2}.

However, since the work of Beilinson--Ginzburg--Soergel~\cite{bgs}, it has been known that $\dw(X)$ and its abelian subcategory $\Pervw(X)$ of perverse sheaves are ``too large'' for certain applications in representation theory.  For instance, when $X$ is the flag variety of a reductive algebraic group $G$, the category $\Pervsw(X)$ of perverse sheaves smooth along the stratification $\cS$ by Bruhat cells is very close to being a \emph{Koszul category} (see Section~\ref{ss:koszul-cat}).  To achieve Koszulity, one must replace it by the full subcategory $\Pervsm(X)$ consisting of objects on which the Frobenius action is semisimple and has integral eigenvalues.  A similar phenomenon occurs at the level of the derived category in work of Arkhipov--Bezrukavnikov--Ginzburg~\cite{abg}; see the remarks at the end of the introduction.

For a variety $X$ with a fixed stratification $\cS$, we may pose two general questions:
\begin{enumerate}
\item[(Q1)] Is there a triangulated category $\dsm(X) \subset \dsw(X)$ analogous to $\Pervsm(X)$ on whose objects the Frobenius action is semisimple and has integral eigenvalues?
\item[(Q2)] Following Deligne, do the usual sheaf operations preserve these stronger conditions on the action of Frobenius?
\end{enumerate}
These questions (along with (Q3) below) are closely related to the ``standard conjectures on algebraic cycles'' and to the Tate conjecture; see~\cite[\S2.9]{tate} or~\cite[Proposition~1.15]{milne:mot}.  The aim of this paper is to supply positive answers in certain very special cases.  In fact, the flag variety is the archetype for the cases we are able to treat; the Koszul duality phenomenon is an essential ingredient in our proofs.

Let us explain what form the answers to the questions above might take, starting with~(Q1).  It is fairly easy to write down (see Section~\ref{ss:miscible}) a condition on objects that generalizes the definition of $\Pervsm(X)$.   However, the resulting full subcategory, which we denote $\dsmisc(X)$ and call the \emph{miscible category}, has a severe disadvantage: \emph{it is not a triangulated category}.  The problem lies in the word ``full'': $\dsmisc(X)$ contains morphisms with no cone, so to give a satisfactory answer to (Q1), we must discard some morphisms from that category so that what remains is a triangulated category.  Equivalently, we could answer (Q1) by constructing a triangulated category $\dsm(X)$ together with a triangulated functor $\incl: \dsm(X) \to \dsw(X)$ such that the following conditions hold:
\begin{enumerate}
\item[(D1)] $\incl$ is faithful (but not full in general).  
\item[(D2)] The essential image of $\incl$ is $\dsmisc(X)$.
\end{enumerate}
One additional desideratum we might impose on $\dsm(X)$ and $\incl$ is as follows:
\begin{enumerate}
\item[(D3)] $\dsm(X)$ admits a $t$-structure whose heart can be identified with $\Pervsm(X)$, and $\incl$ is $t$-exact and induces a fully faithful functor $\incl: \Pervsm(X) \to \Pervsw(X)$.
\end{enumerate}

Turning now to~(Q2), we say that a functor $F: \dsw(X) \to \dw_\cT(Y)$ is \emph{miscible} if $F(\dsmisc(X)) \subset \dmisc_\cT(Y)$.  A positive answer to~(Q2) consists of showing that the usual sheaf operations are miscible.  However, the restricted functor $F: \dsmisc(X) \to \dmisc_\cT(Y)$ is not one that can be studied with the usual tools of homological algebra, because the categories involved are not triangulated.

In retrospect, we see that (Q2) was too coarse a question, because it was only about preserving a certain class of objects.  Instead, we really ought to ask:
\begin{enumerate}
\item[(Q3)] Do the usual sheaf operations preserve the class of morphisms in the image of $\incl: \dsm(X) \to \dsmisc(X)$?
\end{enumerate}

Let us make this more precise.  A miscible functor $F: \dsw(X) \to \dw_\cT(Y)$ is said to be \emph{genuine} if there is a functor of triangulated categories $\tilde F: \dsm(X) \to \dm_\cT(Y)$ making the diagram
\[
\xymatrix{
\dsm(X) \ar[d]_{\tilde F} \ar[r]^{\incl} & \dsmisc(X) \ar[d]^F \\
\dm_\cT(Y) \ar[r]_{\incl} & \dmisc_\cT(Y)}
\]
commute.  (Q3)~asks us to show that the usual sheaf operations are genuine.  The definition of genuineness suggests that we should go back and add one more desideratum to our list:
\begin{enumerate}
\item[(D4)] For any genuine functor $F: \dsw(X) \to \dw_\cT(Y)$, the induced functor $\tilde F: \dsm(X) \to \dm_\cT(Y)$ is unique up to isomorphism.
\end{enumerate}

\bigskip

In this paper, we consider a very special class of stratifications, called \emph{affable} stratifications.  For varieties with an affable stratification, we explain how to construct the category $\dsm(X)$ and the functor $\incl: \dsm(X) \to \dsw(X)$ satisfying desiderata (D1)--(D4), answering~(Q1).  We prove that a number of common functors (proper push-forwards, tensor products, etc.) are at least miscible, answering (Q2).  For some of these (notably, locally closed inclusions, and push-forward along a smooth proper map), we further prove that they are genuine, answering (Q3).  

The paper is divided into three parts.  Part~\ref{part:homological} introduces various notions and results in abstract homological algebra that are needed later.  In particular, Section~\ref{sect:infext} introduces a class of additive categories, called \emph{infinitesimal extensions}, that are ``almost triangulated.'' Section~\ref{sect:hot-orlov} introduces {Orlov categories}, which are a useful tool for constructing morphisms between functors of triangulated categories.  Orlov categories also turn out to be closely related to Koszul duality, of which we give a self-contained account in Section~\ref{sect:kosorlov}.

Part~\ref{part:sheaf} is the core of the paper.  It contains the definition of \emph{affable stratification}, and the definition of the category $\dsm(X)$.  (This definition relies on the fact that $\dsmisc(X)$ is an infinitesimal extension.)  The main results, which assert the miscibility or genuineness of various functors, appear in Section~\ref{sect:genuine}.  Their proofs rely heavily on the theory of Orlov categories.

Finally, Part~\ref{part:applications} gives two brief applications of these results to representation theory, both related to the work of Arkhipov--Bezrukavnikov--Ginzburg~\cite{abg} mentioned earlier.  That paper deals with the affine Grassmannian $\Gr$ for a semisimple algebraic group, stratified by orbits of an Iwahori subgroup.  Realizing that $\dsw(\Gr)$ was the wrong category for their purposes, the authors of that paper substituted the derived category $\Db\Pervsm(\Gr)$.  It turns out that in this case, the natural functor $\Db\Pervsm(\Gr) \to \dsw(\Gr)$ is faithful and induces an equivalence $\Db\Pervsm(\Gr) \simto \dsm(\Gr)$.  Using the sheaf operations on this category that are made available by the results of Part~\ref{part:sheaf}, we prove two small results about Andersen--Jantzen sheaves and about Wakimoto sheaves.

In a subsequent paper~\cite{ar}, the authors will use the theory developed here to show that a derived version of the geometric Satake equivalence coming from~\cite{abg} is compatible with restriction to a Levi subgroup.

\subsection*{Acknowledgments}

The first author is grateful to the Universit\'e Clermont-Fer\-rand II for its hospitality during a visit in June 2010, when much of the work in this paper was carried out.  This visit was supported by the CNRS and the ANR.  In addition, P.A. received support from NSA Grant No.~H98230-09-1-0024 and NSF Grant No.~DMS-1001594, and S.R. is supported by ANR Grant No.~ANR-09-JCJC-0102-01.

\part{Homological algebra}
\label{part:homological}

\section{Mixed and Koszul categories}
\label{sect:mixed}

We begin by collecting a number of definitions related to abelian and triangulated categories.  Fix a field $\Bbbk$.  In this section, and throughout Part~\ref{part:homological}, all additive categories will be $\Bbbk$-linear, and all functors between additive categories will be assumed to be additive and $\Bbbk$-linear as well.

In any additive category $\scA$, we write $\Ind(\scA)$ for the set of isomorphism classes of indecomposable objects in $\scA$, or, by an abuse of notation, for a chosen set of representative objects of those isomorphism classes.  Similarly, in an abelian category $\scM$, we write $\Irr(\scM)$ for the set of isomorphism classes of simple objects, or for a chosen set of representatives of those isomorphism classes. For any $L \in \Irr(\scM)$, the ring $\End(L)$ is a division ring over $\Bbbk$.  We say that $\scM$ is \emph{split} if
\[
\End(L) \cong \Bbbk \qquad\text{for all $L \in \Irr(\scM)$.}
\]
Finally, we say that $\scM$ is a \emph{finite-length} abelian category if it is both noetherian and artinian.  

\subsection{Mixed categories}

Let $\scM$ be a finite-length abelian category.  As in~\cite{bgs}, a \emph{mixed structure} on $\scM$ is a function $\wt: \Irr(\scM) \to \Z$ such that
\begin{equation}\label{eqn:mixed-ext-vanish}
\Ext^1(S,S') = 0
\qquad
\text{if $S, S'$ are simple objects with $\wt(S') \ge \wt(S)$.}
\end{equation}
This function is called a \emph{weight function}.  The \emph{set of weights} of an object $X$ is simply the set of values of $\wt$ evaluated on the composition factors of $X$.  An object is said to be \emph{pure} if all its simple composition factors have the same weight.  It is a consequence of~\eqref{eqn:mixed-ext-vanish} that pure objects are automatically semisimple.  Every object $X$ is endowed with a canonical \emph{weight filtration}, denoted
\[
W_\bullet X,
\]
such that $W_k X$ is the unique maximal subobject of $X$ with weights${}\le k$.

\subsection{Mixed triangulated categories}

Suppose that we have a triangulated category $\scD$ equipped with a bounded $t$-structure whose heart is $\scM$.  A \emph{mixed structure} on $\scD$ is simply a mixed structure on $\scM$ that satisfies the following stronger version of~\eqref{eqn:mixed-ext-vanish}:
\begin{equation}\label{eqn:mixed-hom-vanish}
\Hom_{\scD}^i(S,S') = 0
\quad
\text{if $S, S' \in \scM$ are simple and $\wt(S') > \wt(S) - i$.}
\end{equation}
Here, as usual, we write $\Hom^i(S,S')$ for $\Hom(S,S'[i])$.  When $i = 1$, this condition is equivalent to~\eqref{eqn:mixed-ext-vanish}, by~\cite[Remarque~3.1.17(ii)]{bbd}.  An object $M \in \scD$ is said to \emph{have weights${}\le w$} (resp. \emph{have weights${}\ge w$}, \emph{be pure of weight $w$}) if each cohomology object $H^i(M) \in \scM$ has weights${}\le w+i$ (resp. has weights${}\ge w+i$, is pure of weight $w+i$).  

In the special case where $\scD = \Db(\scM)$, condition~\eqref{eqn:mixed-ext-vanish} implies~\eqref{eqn:mixed-hom-vanish}, because any morphism in $\Hom^i(S,S')$ is a composition of morphisms in various $\Hom^1$-groups.  In other words, the bounded derived category of a mixed abelian category automatically has a mixed structure.  The following basic facts are well-known.

\begin{lem}\label{lem:mixed-tri-basic}
Let $\scM$ be the heart of a $t$-structure on $\scD$, and suppose $\scD$ has a mixed structure.
\begin{enumerate}
\item If $X, Y \in \scD$ are objects such that $X$ has weights${}\le w$ and $Y$ has weights${}>w$, then $\Hom(X,Y) = 0$.\label{it:mixed-tri-basic-homv}
\item Let $X$ be an object of $\scD$ with weights${}\ge a$ and${}\le b$.  For any $w \in \Z$, there is a distinguished triangle\label{it:mixed-tri-basic-dt}
\[
X' \to X \to X'' \to
\]
where $X'$ has weights${}\ge a$ and${}\le w$, and $X''$ has weights${}>w$ and${}\le b$.
\item Every pure object $X \in \scD$ is semisimple.  That is, if $X$ is pure of weight $w$, then $X \cong \bigoplus_i H^i(X)[-i]$, where each $H^i(X) \in \scM$ is a pure (and therefore semisimple) object of weight $w+i$.\label{it:mixed-tri-basic-pure}\qed
\end{enumerate}
\end{lem}

Note that neither the distinguished triangle in part~\eqref{it:mixed-tri-basic-dt} nor the direct-sum decomposition in part~\eqref{it:mixed-tri-basic-pure} is canonical in general.

\subsection{Tate twists; mixed and graded versions}

Suppose now that $\scM$ is a mixed abelian category endowed with an autoequivalence, denoted $M \mapsto M\la 1\ra$, such that for a simple object $S$, $\wt(S\la 1\ra) = \wt(S) + 1$.  Suppose also that we have an exact functor $\degr: \scM \to \scM'$ to another finite-length abelian category $\scM'$, together with an isomorphism $\varepsilon: \degr \circ \la 1 \ra \simto \degr$.  Assume that every simple object of $\scM'$ lies in the essential image of $\degr$.  Then $\scM$ is called a \emph{mixed version} of $\scM'$ if for all objects $M,N \in \scM$, $\degr$ induces an isomorphism
\begin{equation}\label{eqn:degrading}
\bigoplus_{n \in \Z} \Hom_{\scM}(M,N\la n\ra) \simto \Hom_{\scM'}(\degr M, \degr N).
\end{equation}

There are two natural ways to generalize this notion to the setting of triangulated categories.  Suppose that $\scD$ is a triangulated category equipped with an autoequivalence $\la 1\ra: \scD \to \scD$, a functor $\degr: \scD \to \scD'$ whose essential image generates $\scD'$ as a triangulated category, and an isomorphism $\varepsilon: \degr \circ \la 1 \ra \simto \degr$.  Then $\scD$ is called a \emph{graded version} of $\scD'$ if the isomorphism~\eqref{eqn:degrading} holds for all objects $M,N \in \scD$.   

Suppose, in addition, that $\scD$ and $\scD'$ are equipped with $t$-structures such that $\scD$ is a mixed triangulated category, and such that the functors $\la 1 \ra$ and $\degr$ are $t$-exact.  In this case, $\scD$ is said to be a \emph{mixed version} of $\scD'$.

\subsection{Koszul categories}
\label{ss:koszul-cat}

Let $\scM$ be a mixed abelian category.  $\scM$ is said to be \emph{Koszul} if the following stronger version of~\eqref{eqn:mixed-ext-vanish} and~\eqref{eqn:mixed-hom-vanish} holds:
\begin{equation}\label{eqn:koszul-ext-vanish}
\Ext^i(S,S') = 0
\qquad
\text{if $S, S'$ are simple objects with $\wt(S') \ne \wt(S) - i$.}
\end{equation}
In contrast with the setting of~\eqref{eqn:mixed-ext-vanish} and~\eqref{eqn:mixed-hom-vanish}, the $i = 1$ case of~\eqref{eqn:koszul-ext-vanish} does not imply the general condition.  On the other hand, this equation implies the following stronger version of Lemma~\ref{lem:mixed-tri-basic}\eqref{it:mixed-tri-basic-homv}.  

\begin{lem}\label{lem:koszul-tri-basic}
Let $\scM$ be a Koszul category, and let $X, Y \in \Db(\scM)$.  If $X$ has weights${}\le w$ and $Y$ has weights${}> w$, then $\Hom(X,Y) = \Hom(Y,X) = 0$.\qed
\end{lem}

\begin{cor}\label{cor:koszul-split}
Let $\scM$ be a Koszul category.  If $X \in \scM$ has no composition factors of weight $w$, then $X \cong W_{w-1}X \oplus X/W_{w-1}X$.
\end{cor}
\begin{proof}
Since $X/W_{w-1}X$ has weights${}\ge w+1$ and $(W_{w-1}X)[1]$ has weights${}\le w$, we have $\Ext^1(X/W_{w-1}X, W_{w-1}X) \cong \Hom(X/W_{w-1}X, (W_{w-1}X)[1]) = 0$ by the previous lemma, so the short exact sequence $0 \to W_{w-1}X \to X \to X/W_{w-1}X \to 0$ splits.
\end{proof}

A key feature of Koszul categories is that one can often construct a new abelian category $\scM^\natural$, called the \emph{Koszul dual} of $\scM$, such that there is a canonical equivalence of derived categories of $\scM$ and $\scM^\natural$.  ($\scM$ and $\scM^\natural$ need not be equivalent abelian categories.)  A very general form of this equivalence, in which $\scM$ and $\scM^\natural$ are both categories of finitely-generated modules over Koszul rings, is developed in~\cite{bgs}.

Assuming that $\scM$ has enough projectives, the category $\scM^\natural$ can be described as the full subcategory of $\Db(\scM)$ given by
\begin{equation}\label{eqn:koszul-duality}
\scM^\natural = \left\{ X \in \Db(\scM) \,\Bigg|\,
\begin{array}{c}
\text{for any indecomposable projective $P \in \scM$,} \\
\text{we have $\Hom(X,P[k]) = 0$ if $k < \wt (P/\rad P)$} \\
\text{and $\Hom(P[k],X) = 0$ if $k > \wt (P/\rad P)$}
\end{array}
\right\}.
\end{equation}
The following theorem is one way to formulate Koszul duality in this setting.

\begin{thm}[Koszul Duality]\label{thm:koszul-duality}
Let $\scM$ be a Koszul category with enough projectives, and assume that every object has finite projective dimension.  Then $\scM^\natural$ is the heart of a $t$-structure on $\Db(\scM)$, and the realization functor (see Section~\ref{ss:realization})
\[
\real: \Db(\scM^\natural) \to \Db(\scM)
\]
is an equivalence of categories.  Moreover, $\scM^\natural$ is itself a Koszul category, with
\[
\Irr(\scM^\natural) = \{ P[-\wt (P/\rad P)] \mid \text{$P \in \scM$ an indecomposable projective} \}
\]
and with weight function $\wt^\natural: \Irr(\scM^\natural) \to \Z$ given by
\[
\wt^\natural( P[-\wt (P/\rad P)] ) = \wt (P/\rad P).
\]
Finally, $\scM^\natural$ has enough injectives, and every object has finite injective dimension.  The indecomposable injectives are of the form $\{ L[-\wt L] \mid L \in \Irr(\scM) \}$.
\end{thm}

There is, of course, an analogous construction of a Koszul dual category $\scM \mapsto {}^\natural \scM$ for Koszul categories with enough injectives in which every object has finite injective dimension.  Starting from a category $\scM$ satisfying the hypotheses of Theorem~\ref{thm:koszul-duality}, one finds that the composition
\[
\Db({}^\natural(\scM^\natural)) \simto \Db(\scM^\natural) \simto \Db(\scM)
\]
induces an equivalence of abelian categories ${}^\natural(\scM^\natural) \simto \scM$.  In this way, passage to the Koszul dual is an involution.

The idea of Koszul duality is quite well-known.  However, the specific version stated above cannot readily be extracted from the statements in~\cite{bgs}, because that paper imposes additional assumptions on $\scM$: specifically, $\scM$ is assumed to be endowed with a Tate twist, and to have only finitely many isomorphism classes of simple objects up to Tate twist.  We will therefore give a self-contained proof in Section~\ref{sect:kosorlov}, which also contains a more general statement with weaker assumptions on $\scM$.

\subsection{Realization functors for homotopy categories}
\label{ss:realization}

Let $\scD$ be a triangulated category equipped with a $t$-structure $(\scD^{\le 0}, \scD^{\ge 0})$, and let $\scC = \scD^{\le 0} \cap \scD^{\ge 0}$ be its heart.  One can ask for a $t$-exact functor of triangulated categories
\[
\Db(\scC) \to \scD
\]
that restricts to the identity functor on $\scC$.  Such a functor is called a \emph{realization functor}.  One well-known construction of such a functor, adequate for Theorem~\ref{thm:koszul-duality} above, is given in~\cite[\S 3.1]{bbd}.  That construction assumes that $\scD$ is a full triangulated subcategory of the derived category of an abelian category.  Unfortunately, many of the triangulated categories we encounter in this paper are not of that form.

In~\cite{bei}, Beilinson has explained how to axiomatize the notion of a ``filtered derived category'' and thereby generalize the construction of~\cite[\S 3.1]{bbd} to other triangulated categories.  In particular, he treats the case where $\scD$ is the category $\dw(X)$ of mixed Weil complexes of $\Qlb$-sheaves on a variety over a finite field, cf.~Section~\ref{sect:mixedweil}.

Another case that is important in this paper is that in which $\scD$ is the bounded homotopy category $\Kb(\scA)$ of an additive category $\scA$.  In this section, we explain how to apply the formalism of~\cite{bei} to this setting.

\begin{lem}\label{lem:chain-ses}
Suppose we have three objects $X= (X^\bullet,d_X)$, $Y = (Y^\bullet, d_Y)$, and $Z = (Z^\bullet,d_Z)$ in $\Kb(\scA)$, and two chain maps $f = (f^\bullet): X \to Y$ and $g = (g^\bullet): Y \to Z$.  Assume that for each $i \in \Z$, we have an identification $Y^i \cong X^i \oplus Z^i$ such that the maps
\[
f^i: X^i \to X^i \oplus Z^i
\qquad\text{and}\qquad
g^i: X^i \oplus Z^i \to Z^i
\]
are the inclusion and projection maps, respectively for $X^i$ and $Z^i$ as direct summands of $Y^i$.  Then there is a chain map $\delta: Z \to X[1]$ such that
\[
X \ovto{f} Y \ovto{g} Z \ovto{\delta} X[1]
\]
is a distinguished triangle in $\Kb(\scA)$.
\end{lem}
\begin{proof}
Using the identification $Y^i = X^i \oplus Z^i$, we can write the differential $d_Y^i: Y^i \to Y^{i+1}$ as a matrix
\[
d_Y^i =
\begin{bmatrix} s^i & t^i \\ u^i & v^i \end{bmatrix}.
\]
Note that $d_Y^i \circ f^i = [\begin{smallmatrix} s^i \\ u^i \end{smallmatrix}]$.  On the other hand, $f^{i+1} \circ d_X^i = [\begin{smallmatrix} d_X^i \\ 0 \end{smallmatrix}]$.  We conclude that $s^i = d_X^i$ and $u^i = 0$.  Similar reasoning shows that $v^i = d_Y^i$.  Define $\delta: Z \to X[1]$ by setting $\delta^i = t^i: Z^i \to X^{i+1} = (X[1])^i$.  It follows from the fact that $d_Y^{i+1} \circ d_Y^i = 0$ that $\delta^\bullet$ is, in fact, a chain map.  Moreover, it is now evident from the formula for $d_Y$ above that $Y$ is the cocone of $\delta: Z \to X[1]$.
\end{proof}

Let $F\scA$ denote the additive category whose objects are sequences
\[
\cdots \overset{e_{-1}}{\longleftarrow} X_{-1} \overset{e_{0}}{\longleftarrow} X_0 \overset{e_{1}}{\longleftarrow} X_1 \overset{e_{2}}{\longleftarrow} \cdots
\]
of objects in $\scA$, satisfying the following conditions:
\begin{enumerate}
\item Each $e_i: X_i \to X_{i-1}$ is an inclusion of a direct summand of $X_{i-1}$.
\item There are integers $a \le b$ such that:
\begin{enumerate}
\item $X_i = X_a$ and $e_i = \id$ for all $i \le a$.
\item $X_i = 0$ for all $i > b$.
\end{enumerate}
\end{enumerate}
If $X = (X_\bullet, e^X_\bullet)$ and $Y = (Y_\bullet, e^Y_\bullet)$ are two objects of $F\scA$, a morphism $f: X \to Y$ is simply a collection of maps $(f_i: X_i \to Y_i)_{i \in \Z}$ such that $f_{i-1} \circ e^X_i = e^Y_i \circ f_i$ for all $i$.  Intuitively, we may think of $F\scA$ as the category of ``objects in $\scA$ equipped with finite decreasing filtrations.''

Let $s: F\scA \to F\scA$ be the functor which sends an object $X = (X_\bullet, e^X_\bullet)$ to
\[
s(X) = (s(X)_\bullet, e^{s(X)}_\bullet)
\qquad\text{where}\qquad
s(X)_i = X_{i-1}
\quad\text{and}\quad
e^{s(X)}_i = e^X_{i-1},
\]
and likewise for morphisms.  Note that we have a canonical morphism
\[
\alpha: X \to s(X)
\qquad\text{given by}\qquad
\alpha_i = e_i: X_i \to s(X)_i.
\]
For any $n \in \Z$, we can form the following full additive subcategories of $F\scA$:
\begin{align*}
F\scA({\le n}) &= \{ X = (X_\bullet,p_\bullet) \mid \text{$X_i = 0$ for $i > n$} \}, \\
F\scA({\ge n}) &= \{ X = (X_\bullet,p_\bullet) \mid \text{$X_i = X_n$ and $e_i = \id$ for $i \le n$} \}.
\end{align*}
Lastly, consider the functor $j: \scA \to F\scA$ that sends an object $X$ to the sequence given by
\[
j(X)_i = 
\begin{cases}
X & \text{if $i \le 0$,} \\
0 & \text{if $i > 0$,}
\end{cases}
\qquad
e^{j(X)}_i = 
\begin{cases}
\id & \text{if $i \le 0$,} \\
0 & \text{if $i > 0$,}
\end{cases}
\]
and which sends a morphism $f: X \to Y$ in $\scA$ to the sequence $(f_i)$ with  $f_i = f$ for $i \le 0$ and $f_i = 0$ for $i > 0$.  

\begin{lem}\label{lem:addfilt}
\begin{enumerate}
\item Given objects $X \in F\scA({\ge 1})$ and $Y \in F\scA({\le 0})$, we have $\Hom(X,Y) = 0$.  Moreover, $\alpha$ induces isomorphisms\label{it:ntfilt-hom}
\begin{equation}\label{eqn:ntfilt-hom}
\Hom(Y,X) \cong \Hom(Y, s^{-1}X) \cong \Hom(sY,X).
\end{equation}
\item Every object $X \in F\scA$ admits a direct-sum decomposition $X \cong A \oplus B$ with $A \in F\scA({\ge 1})$ and $B \in F\scA({\le 0})$.\label{it:ntfilt-dt}
\item The functor $j$ induces an equivalence of additive categories $j: \scA \simto F\scA({\le 0}) \cap F\scA({\ge 0})$.\label{it:ntfilt-heart}
\end{enumerate}
\end{lem}
\begin{proof}
\eqref{it:ntfilt-hom}~If $X = (X_\bullet, e^X_\bullet) \in F\scA({\ge 1})$ and $Y = (Y_\bullet, e^Y_\bullet) \in F\scA({\le 0})$, then for any map $f =(f_\bullet): X \to Y$, we clearly have $f_i = 0$ for $i \ge 1$.  On the other hand, for $i \le 0$, we have
\[
f_i \circ e^X_{i+1} \circ e^X_{i+2} \circ \cdots \circ e^X_1 = e^Y_{i+1} \circ e^Y_{i+2} \circ \cdots \circ e^Y_1 \circ f_1 = 0.
\]
Since $e^X_{i+1} = \cdots = e^X_1 = \id_{X_1}$, it follows that $f_i = 0$ for all $i$, so $\Hom(X,Y) = 0$.  Next, for any morphism $g: Y \to s^{-1}X$ or $g: Y \to X$, we have $g_i = 0$ for $i \ge 1$.  Thus, the natural map $\phi: \Hom(Y,s^{-1}X) \to \Hom(Y,X)$ induced by $\alpha_{s^{-1}X}$ can be described by
\[
\phi(g)_i :
\begin{cases}
e^X_{i+1} \circ g_i & \text{if $i \le 0$,} \\
0 & \text{if $i \ge 1$.}
\end{cases}
\]
But $X \in F(\scA)({\ge 1})$ means that $e^X_{i+1}$ is the identity map for $i \le 0$, and it follows that $\phi$ is a bijection.  The same reasoning shows that $\Hom(Y,X) \cong \Hom(sY,X)$.

\eqref{it:ntfilt-dt}~Let $N \le 0$ be such that $X = (X_\bullet, e^X_\bullet) \in F\scA({\ge N})$.  We will construct the terms of $A$ and $B$ by downward induction as follows.  For $i \ge 1$, let $A_i = X_i$ and let $B_i = 0$.  Next, for $N \le i \le 0$, if $A_{i+1}$ and $B_{i+1}$ are already defined, then the map $e_{i+1}: X_{i+1} \to X_i$ lets us regard $A_{i+1}$ and $B_{i+1}$ as direct summands of $X_i$.  Let $Y_i$ be a complementary direct summand in $X_i$ to $X_{i+1}$, and then set $A_i = A_{i+1}$ and $B_i = B_{i+1} \oplus Y_i$.  With respect to the identifications $X_{i+1} = A_{i+1} \oplus B_{i+1}$ and $X_i = A_i \oplus B_i$, $e_{i+1}$ has the form
\begin{equation}\label{eqn:ntfilt-oplus}
e^X_{i+1} = \begin{bmatrix}
\id & 0 \\ 0 & \bar e_{i+1}
\end{bmatrix}
\end{equation}
for some map $\bar e_{i+1}: B_{i+1} \to B_i$.  Finally, for $i < N$, we set $A_i = A_N$ and $B_i = B_N$.  Let us put
\[
e_i^A =
\begin{cases}
e^X_i & \text{if $i > 1$}, \\
\id_{A_1} & \text{if $i \le 1$},
\end{cases}
\qquad
e_i^B =
\begin{cases}
0 & \text{if $i > 1$,} \\
\bar e_i & \text{if $N < i \le 1$,} \\
\id & \text{if $i \le N$.}
\end{cases}
\]
Then the object $A = (A_\bullet, e^A_\bullet)$ belongs to $F\scA({\ge 1})$, $B = (B_\bullet, e^B_\bullet)$ lies in $F\scA({\le 0})$.  It follows from~\eqref{eqn:ntfilt-oplus} that $X \cong A \oplus B$.

\eqref{it:ntfilt-heart}~It is clear that $j$ is faithful and essentially surjective.  Moreover, it is easy to see that any morphism $f = (f_\bullet): X \to Y$ between two objects $X, Y \in F\scA({\le 0}) \cap F\scA({\ge 0})$ is determined by $f_0$.  Thus, $j$ is full, and hence an equivalence of categories.
\end{proof}

We now consider the bounded homotopy category $\Kb(F\scA)$ of $F\scA$.  The functors $s$ and $j$ extend in an obvious way to functors of triangulated categories
\[
s: \Kb(F\scA) \to \Kb(F\scA), \qquad
j: \Kb(\scA) \to \Kb(F\scA),
\]
and $\alpha$ extends to a morphism of functors $\alpha: \id_{\Kb(F\scA)} \to s$.  We also define $\Kb(F\scA)({\le n})$ (resp.~$\Kb(F\scA)({\ge n})$) to be the full subcategory of $\Kb(F\scA)$ consisting of objects isomorphic to a complex $X = (X^\bullet, d_X)$ with $X^i \in F\scA({\le n})$ (resp.~$X^i \in F\scA({\ge n})$) for all $i$.

\begin{lem}\label{lem:trifilt}
With the above notation, we have the following properties.
\begin{enumerate}
\item $s^n (\Kb(F\scA)({\le 0})) = \Kb(F\scA)({\le n});s^n(\Kb(F\scA)({\ge 0})) = \Kb(F\scA)({\ge n})$.\label{it:filt-shift}
\item $\Kb(F\scA)({\ge 1}) \subset \Kb(F\scA)({\ge 0})$, $\Kb(F\scA)({\le 1}) \supset \Kb(F\scA)({\le 0})$, and $\bigcup_{n \in \Z} \Kb(F\scA)({\le n}) = \bigcup_{n \in \Z} \Kb(F\scA)({\ge n}) = \Kb(F\scA)$.\label{it:filt-nondeg}
\item For any object $X \in \Kb(F\scA)$, we have $\alpha_X = s(\alpha_{s^{-1}(X)})$.\label{it:filt-alpha}
\item For $X \in \Kb(F\scA)({\ge 1})$ and $Y \in \Kb(F\scA)({\le 0})$, we have $\Hom(X,Y) = 0$.  Moreover, $\alpha$ induces isomorphisms\label{it:filt-hom}
\[
\Hom(Y,X) \cong \Hom(Y, s^{-1}X) \cong \Hom(sY,X).
\]
\item For any object $X \in \Kb(F\scA)$, there is a distinguished triangle $A \to X \to B \to$ with $A \in \Kb(F\scA)({\ge 1})$ and $B \in \Kb(F\scA)({\le 0})$.\label{it:filt-dt}
\item Every object of $\Kb(F\scA)({\le 0}) \cap \Kb(F\scA)({\ge 0})$ is isomorphic to a chain complex $X = (X^\bullet, d_X)$ with $X^i \in F\scA({\le 0}) \cap F\scA({\ge 0})$ for all $i$.\label{it:filt-supp}
\item The functor $j$ gives rise to an equivalence of triangulated categories $j: \Kb(\scA) \simto \Kb(F\scA)({\le 0}) \cap \Kb(F\scA)({\ge 0})$.\label{it:filt-heart}
\end{enumerate}
\end{lem}
\begin{proof}
Parts~\eqref{it:filt-shift}--\eqref{it:filt-alpha} are straightforward from the definitions.  For~\eqref{it:filt-hom}, the vanishing of $\Hom(X,Y)$ follows from the corresponding statement in Lemma~\ref{lem:addfilt}\eqref{it:ntfilt-hom}.  Because the isomorphisms in~\eqref{eqn:ntfilt-hom} are natural, they induce corresponding isomorphisms in the additive category of chain complexes over $F\scA$.  Furthermore, the latter isomorphisms respect homotopy, and so descend to $\Kb(F\scA)$.

For part~\eqref{it:filt-dt}, given an object $X = (X^\bullet, d_X) \in \Kb(F\scA)$, let us endow each term of the chain complex with a decomposition $X^i = A^i \oplus B^i$ with $A^i \in F\scA({\ge 1})$ and $B^i \in F\scA({\le 0})$, as in Lemma~\ref{lem:addfilt}\eqref{it:ntfilt-dt}.  Each differential $d_X^i: X^i \to X^{i+1}$ can then be written as a matrix
\[
d_X^i = \begin{bmatrix}
d_A^i & \delta^i \\
0 & d_B^i
\end{bmatrix},
\]
where the lower left-hand entry is $0$ because $\Hom(A^i,B^{i+1}) = 0$ by Lemma~\ref{lem:addfilt}\eqref{it:ntfilt-hom}.  Then $A = (A^\bullet, d_A^\bullet)$ is a chain complex in $\Kb(F\scA)({\ge 1})$, and $B = (B^\bullet, d_B^\bullet) \in \Kb(F\scA)({\le 0})$.  We have obvious chain maps $A \to X \to B$, and this diagram extends to a distinguished triangle by Lemma~\ref{lem:chain-ses}.

If we apply this construction to a chain complex $X$ with $X^i \in F\scA({\ge 0})$ for all $i$, then we find that $B^i \in F\scA({\le 0}) \cap F\scA({\ge 0})$.  If $X$ also lies in $\Kb(F\scA)({\le 0})$, then we must have $A \cong 0$ (because $\Hom(A,X) = \Hom(A, B[-1]) = 0$), so $X \cong B$.  This establishes part~\eqref{it:filt-supp}.  It follows from that statement that the inclusion functor
\[
\Kb(F\scA({\le 0}) \cap F\scA({\ge 0})) \simto
\Kb(F\scA)({\le 0}) \cap \Kb(F\scA)({\ge 0})
\]
is an equivalence of categories.  Part~\eqref{it:filt-heart} then follows from Lemma~\ref{lem:addfilt}\eqref{it:ntfilt-heart}.
\end{proof}

In the terminology of~\cite[Appendix]{bei}, the preceding lemma states that $\Kb(F\scA)$, together with the data consisting of $s$, $j$, and $\alpha$, is an \emph{$f$-category over $\Kb(\scA)$}.  The machinery of {\em loc.~cit.}~then gives us the following result.

\begin{thm}[{\cite[\S A.7]{bei}}]\label{thm:realization}
Let $\scA$ be an additive category, and let $\scC$ be the heart of a $t$-structure on $\Kb(\scA)$.  There is a $t$-exact functor of triangulated categories $\real: \Db(\scC) \to \Kb(\scA)$ with the property that $\real|_{\scC} \cong \id_{\scC}$.\qed
\end{thm}

Together, this result and those in~\cite[\S 3.1]{bbd} and~\cite{bei} cover all the cases we need.  We will henceforth make use of realization functors whenever necessary without further explanation. 

\section{Infinitesimal extensions of triangulated categories}
\label{sect:infext}

In this section, we will study a kind of ``thickened'' version of a triangulated category, with extra morphisms (called infinitesimal morphisms) that do not have cones.  Such a category looks bizarre from the usual perspective of homological algebra, but they arise naturally in the setting of \'etale $\ell$-adic sheaves on certain varieties, cf.~Section~\ref{sect:affable}.

\subsection{Basic properties of infinitesimal extensions}

Let $\scD$ be a triangulated category.  Let $\pse\scD$ be the category whose objects are the same as those of $\scD$, but whose $\Hom$-spaces are given by
\begin{equation}\label{eqn:pseudotri-hom}
\Hom_{\pse\scD}(X,Y) = \Hom_{\scD}(X,Y) \oplus \Hom_{\scD}(X, Y[-1]),
\end{equation}
and where composition of morphisms is given by the rule
\begin{equation}\label{eqn:pseudotri-comp}
(g_0,g') \circ (f_0,f') = (g_0 \circ f_0, g_0[-1] \circ f' + g' \circ f_0).
\end{equation}
There are obvious functors $\incl: \scD \to \pse\scD$ and $\pg: \pse\scD \to \scD$ that send objects to themselves, and for which the induced maps
\[
\incl: \Hom_{\scD}(X,Y) \to \Hom_{\pse\scD}(X,Y)
\qquad\text{and}\qquad
\pg: \Hom_{\pse\scD}(X,Y) \to \Hom_{\scD}(X,Y)
\]
are the inclusion and projection maps, respectively, for $\Hom_{\scD}(X,Y)$ as a direct summand of $\Hom_{\pse\scD}(X,Y)$.  We also have the inclusion map
\begin{equation}\label{eqn:iinf-defn}
\iinf: \Hom_{\scD}(X,Y[-1]) \to \Hom_{\pse\scD}(\incl X,\incl Y).
\end{equation}
It follows from~\eqref{eqn:pseudotri-comp} that $\iinf$ is a natural transformation.

\begin{defn}\label{defn:infext}
The category $\pse\scD$ defined above is called the \emph{infinitesimal extension} of $\scD$.  A morphism $f = (f_0,f'): X \to Y$ in $\pse\scD$ is said to be \emph{infinitesimal} if $\pg(f) = 0$, or, equivalently, if $f_0 = 0$.  On the other hand, $f$ is \emph{genuine} if $f = \incl(f_0)$, i.e., if $f' = 0$.

A diagram $X \to Y \to Z \to X[1]$ is called a \emph{distinguished triangle} if there is a commutative diagram
\[
\xymatrix{
X \ar[r] \ar[d]_{\theta}^{\wr} &
Y \ar[r] \ar[d]^{\wr} &
Z \ar[r] \ar[d]^{\wr} &
X[1] \ar[d]_{\theta[1]}^{\wr} \\
\incl(X') \ar[r]^{\incl(f)} &
\incl(Y') \ar[r]^{\incl(g)} &
\incl(Z') \ar[r]^{\incl(h)} &
\incl(X'[1]) }
\]
where $X' \ovto{f} Y' \ovto{g} Z' \ovto{h} X'[1]$ is some distinguished triangle in $\scD$, and where the vertical maps are isomorphisms.  A morphism is said to \emph{have a cone} if it occurs in some distinguished triangle.
\end{defn}

\begin{rmk}
For morphisms in $\pse\scD$, the property of being genuine is \emph{not} natural. In particular, a genuine morphism may be conjugate to a morphism that is not genuine.  In contrast, being infinitesimal is a natural notion.
\end{rmk}

It is clear from the definitions of $\pg$ and $\incl$ that
\begin{equation}\label{eqn:infext-id}
\pg \circ \incl \cong \id_{\scD}.
\end{equation}
Note that a morphism that has a cone must be conjugate to a genuine morphism, and so cannot be infinitesimal.  In other words, infinitesimal morphisms do not have cones, so $\pse\scD$ cannot be a triangulated category unless $\scD = 0$.

\begin{lem}\label{lem:infext-basic}
\begin{enumerate}
\item If $f: X \to Y$ and $g: Y \to Z$ are both infinitesimal morphisms in $\pse\scD$, then $g \circ f = 0$.
\item A morphism $f = (f_0,f')$ in $\pse\scD$ is an isomorphism if and only if $\pg(f) = f_0$ is an isomorphism in $\scD$.
\end{enumerate}
\end{lem}
\begin{proof}
The first assertion is immediate from~\eqref{eqn:pseudotri-comp}.  If $f$ is an isomorphism, then it is clear that $f_0 = \pg(f)$ must be as well.  If $f_0$ is an isomorphism, then one may check that $g = (f_0^{-1}, -f_0^{-1}[-1] \circ f' \circ f_0^{-1})$ is an inverse for $f$.
\end{proof}

Let $\itr: \pse\scD \to \scD$ be the functor defined as follows: for an object $X$, we put
\[
\itr(X) = X \oplus X[-1],
\]
and for a morphism $f = (f_0,f'): X \to Y$, we put
\[
\itr(f) =
\begin{bmatrix}
f_0 & \\ f' & f_0[-1]
\end{bmatrix}
: X \oplus X[-1] \to Y \oplus Y[-1].
\]

\begin{lem}\label{lem:infext-biadj}
The functor $\incl: \scD \to \pse\scD$ is left adjoint to $\itr$ and right adjoint to $\itr[1]$.
\end{lem}
\begin{proof}
We will prove the first assertion by explicitly constructing the unit $\eta: \id_{\scD} \to \itr \incl$ and the counit $\epsilon: \incl \itr \to \id_{\pse\scD}$.  For an object $X \in \scD$, define
\[
\eta_X: X \to X \oplus X[-1]
\qquad\text{by}\qquad
\eta_X = \begin{bmatrix} \id_X \\ 0 \end{bmatrix}.
\]
It is straighforward to check that for a morphism $f: X \to Y$ in $\scD$, we have $\itr(\incl(f)) \circ \eta_X = \eta_Y \circ f$, so this is indeed a morphism of functors.  Next, define
\[
\epsilon_X: X \oplus X[-1] \to X
\qquad\text{by}\qquad
\epsilon_X = \begin{bmatrix} (\id_X,0) & (0,\id_{X[-1]}) \end{bmatrix}.
\]
Here, the notation ``$\id$'' denotes identity morphisms in $\scD$, of course.  Consider a morphism $f = (f_0,f'): X \to Y$ in $\pse\scD$.  The following equation shows that $\epsilon$ is a morphism of functors:
\[
(f_0,f') \circ \begin{bmatrix} (\id_X,0) & (0,\id_{X[-1]}) \end{bmatrix}
=
\begin{bmatrix} (\id_Y,0) & (0,\id_{Y[-1]}) \end{bmatrix}
\begin{bmatrix}
(f_0,0) & \\ (f',0) & (f_0[-1],0)
\end{bmatrix}.
\]
Next, we must show that $\itr \epsilon \circ \eta \itr = \id: \itr \to \itr$.  This follows from the following calculations:
\begin{gather*}
\eta\itr_X = 
\begin{bmatrix}
\id_X & 0 \\ 0 & \id_{X[-1]} \\ 0 & 0 \\ 0 & 0 \end{bmatrix}
: X \oplus X[-1] \to (X \oplus X[-1]) \oplus (X \oplus X[-1])[-1],
\\
\itr\epsilon_X =
\begin{bmatrix}
\id_X & 0 & 0 & 0 \\
0 & \id_{X[-1]} & \id_{X[-1]} & 0
\end{bmatrix}
:
\begin{aligned}
 (X \oplus X[-1]) \oplus (X \oplus X&[-1])[-1)) \\ &\to  X \oplus X[-1].
\end{aligned}
\end{gather*}
The proof that $\epsilon \incl \circ \incl\eta = \id: \incl \to \incl$ is similar.  Thus, $\incl$ is left adjoint to $\itr$.

For the other adjunction, we record below the formulas for the unit $\eta: \id_{\pse\scD} \to \incl \itr[1]$ and the counit $\epsilon: \itr[1] \incl \to \id_{\scD}$ but otherwise omit further details.
\[
\eta_X =
\begin{bmatrix} (0,\id_X) \\ (\id_X,0) \end{bmatrix}
: X \to X[1] \oplus X,
\qquad
\epsilon_X =
\begin{bmatrix}
0 & \id_X
\end{bmatrix}
: X[1] \oplus X \to X.
\qedhere
\]
\end{proof}

\subsection{Distinguished triangles in an infinitesimal extension}

A number of familiar facts from homological algebra remain valid in $\pse\scD$, even though that category is not triangulated.  We prove a few of these in the next two lemmas.

\begin{lem}\label{lem:infext-hom-les}
Let $X \to Y \to Z \to$ be a distinguished triangle in $\pse\scD$.  For any object $A \in \pse\scD$, the following two sequences are exact:
\begin{gather*}
\cdots \to \Hom(A,X) \to \Hom(A,Y) \to \Hom(A,Z) \to \Hom(A,X[1]) \to \cdots \\
\cdots \to \Hom(X[1],A) \to \Hom(Z,A) \to \Hom(Y,A) \to \Hom(X,A) \to \cdots
\end{gather*}
\end{lem}
\begin{proof}
By replacing the given triangle $X \to Y \to Z \to$ by an isomorphic one if necessary, we may assume that it arises by applying $\incl$ to a distinguished triangle $X' \to Y' \to Z' \to$ in $\scD$.  By Lemma~\ref{lem:infext-biadj}, applying the functor $\Hom_{\pse\scD}(A,\cdot)$ to the given triangle is equivalent to applying $\Hom_{\scD}(\itr(A[1]), \cdot)$ to a triangle in $\scD$, so the resulting sequence is exact.  Similar reasoning applies to $\Hom_{\pse\scD}(\cdot, A)$.
\end{proof}

\begin{lem}\label{lem:infext-sq-comp}
Consider a commutative diagram
\begin{equation}\label{eqn:infext-sq-comp}
\vcenter{\xymatrix{
X \ar[r]^f \ar[d]_p & Y \ar[d]_q \\
X' \ar[r]_{i} & Y' }}
\end{equation}
in $\pse\scD$.  If $f$ and $i$ both have cones, then this diagram can be completed to a morphism of distinguished triangles
\begin{equation}\label{eqn:infext-sq-completed}
\vcenter{\xymatrix{
X \ar[r]^f \ar[d]_p & Y \ar[r]^g \ar[d]_q &
  Z \ar[r]^h \ar[d]_r & X[1] \ar[d]^{p[1]} \\
X' \ar[r]_{i} & Y' \ar[r]_{j} &
  Z' \ar[r]_{k} & X'[1] }}
\end{equation}
Moreover, if $p$ and $q$ are isomorphisms, then $r$ is an isomorphism as well.
\end{lem}
\begin{proof}
By replacing $f$ and $i$ by isomorphic maps, we may assume that they are both genuine.  (We cannot assume that $p$ and $q$ are genuine, however.)  Let us write these maps as pairs:
\[
f = (f_0,0), \quad
i = (i_0,0), \quad
p = (p_0,p'), \quad
q = (q_0,q').
\]
The commutativity of~\eqref{eqn:infext-sq-comp} implies that the following squares in $\scD$ each commute:
\begin{equation}\label{eqn:infext-sq-sep}
\vcenter{\xymatrix{
X \ar[r]^{f_0} \ar[d]_{p_0} & Y \ar[d]^{q_0} \\
X' \ar[r]_{i_0} & Y' }}
\qquad\text{and}\qquad
\vcenter{\xymatrix{
X \ar[r]^{f_0} \ar[d]_{p'} & Y \ar[d]^{q'} \\
X'[-1] \ar[r]_{i_0[-1]} & Y'[-1] }}
\end{equation}
We can complete each of these to a morphism of distinguished triangles as follows:
\[
\xymatrix@C=20pt{
X \ar[r]^{f_0} \ar[d]^{p_0} & Y \ar[r]^{g_0} \ar[d]_{q_0} &
  Z \ar[r]^{h_0} \ar[d]_{r_0} & X[1] \ar[d]_{p_0[1]} \\
X' \ar[r]_{i_0} & Y' \ar[r]_{j_0} &
  Z' \ar[r]_{k_0} & X'[1] }
\ 
\xymatrix{
X \ar[r]^{f_0} \ar[d]^{p'} & Y \ar[r]^{g_0} \ar[d]_{q'} &
  Z \ar[r]^{h_0} \ar[d]_{r'} & X[1] \ar[d]_{p'[1]} \\
X'[-1] \ar[r]_{i_0[-1]} & Y'[-1] \ar[r]_{j_0[-1]} &
  Z'[-1] \ar[r]_-{k_0[-1]} & X' }
\]
Note that we have chosen the same objects $Z$ and $Z'$ and the same morphisms $g_0$, $h_0$, $j_0$, $k_0$ in both diagrams.  Let us put
\[
g = \incl(g_0),\quad
h = \incl(h_0),\quad
j = \incl(j_0),\quad
k = \incl(k_0),
\]
and let $r = (r_0,r') \in \Hom(Z,Z')$.  Then~\eqref{eqn:infext-sq-completed} commutes.

For any object $A \in \pse\scD$, applying $\Hom(A,\cdot)$ to the diagram~\eqref{eqn:infext-sq-completed} gives us a morphism of long exact sequences, by Lemma~\ref{lem:infext-hom-les}.  If $p$ and $q$ are isomorphisms, then, by the $5$-lemma, the map $\Hom(A,Z) \to \Hom(A, Z')$ induced by $r$ is always an isomorphism.  By Yoneda's lemma, $r$ itself is an isomorphism.
\end{proof}

\subsection{Pseudotriangulated functors}
\label{ss:pseudotri}

We will now study functors that respect the structure of an infinitesimal extension of a triangulated category.

\begin{defn}\label{defn:pseudotri}
An additive functor $F: \pse\scD \to \pse\scD'$ is said to be \emph{pseudotriangulated} if the following two conditions hold:
\begin{enumerate}
\item It commutes with $[1]$ and takes distinguished triangles to distinguished triangles.\label{it:pseudotri-dt}
\item It commutes with $\iinf \circ \pg$.\label{it:pseudotri-commute}
\end{enumerate}
We also use the term \emph{pseudotriangulated} for functors $\scD \to \pse\scD'$ satisfying just condition~\eqref{it:pseudotri-dt}.
\end{defn}

The last condition means that the following diagram commutes:
\[
\xymatrix{
\Hom_{\pse\scD}(X,Y[-1]) \ar[r]^-{\pg} \ar[d]_F &
  \Hom_{\scD}(X,Y[-1]) \ar[r]^-{\iinf} & \Hom_{\pse\scD}(X,Y) \ar[d]^F \\
\Hom_{\pse\scD'}(FX,FY[-1]) \ar[r]_-{\pg} &
  \Hom_{\scD'}(FX,FY[-1]) \ar[r]_-{\iinf} & \Hom_{\pse\scD'}(FX,FY) }
\]
The following basic facts about pseudotriangulated functors are immediate consequences of the definition.

\begin{lem}\label{lem:pseudotri-basic}
Let $F: \pse\scD \to \pse\scD'$ be a pseudotriangulated functor.  Then:
\begin{enumerate}
\item $F$ takes infinitesimal morphisms to infinitesimal morphisms.\label{it:pseudotri-inf}
\item We have $\pg \circ F \circ \incl \circ \pg \cong \pg \circ F$. \qed\label{it:pseudotri-pg}
\end{enumerate}
\end{lem}

\begin{lem}\label{lem:infext-induc}
For any pseudotriangulated functor $F: \pse\scD \to \pse\scD'$, there is a functor of triangulated categories $\tilde F: \scD \to \scD'$, unique up to isomorphism, such that $\pg \circ F \cong \tilde F \circ \pg$.  
\end{lem}
\begin{defn}\label{defn:infext-induc}
The functor $\tilde F: \scD \to \scD'$ is said to be \emph{induced} by $F$.
\end{defn}
\begin{proof}
Let $\tilde F = \pg \circ F \circ \incl$.  By Lemma~\ref{lem:pseudotri-basic}\eqref{it:pseudotri-pg}, we have that $\pg \circ F \cong \tilde F \circ \pg$.  For uniqueness, suppose we have an isomorphism $\phi: \tilde F \circ \pg \simto G \circ \pg$ for some $G: \scD \to \scD'$.  Since the objects of $\scD$ are the same as those of $\pse\scD$, we can define a morphism $\phi': \tilde F \to G$ simply by setting $\phi'_X = \phi_X: \tilde F(X) \to G(X)$, and this is clearly an isomorphism.
\end{proof}

The lemma above may be thought of as saying that pseudotriangulated functors are ``automatically'' compatible with $\pg$.  The analogous property for $\incl$, however, is not automatic, and turns out to be rather more difficult to study.

\begin{defn}\label{defn:pseudotri-genuine}
Let $\scD$ and $\scD'$ be two triangulated categories.  A pseudotriangulated functor $F: \pse\scD \to \pse\scD'$ is said to be \emph{genuine} if the induced functor $\tilde F$ satisfies $\incl \circ \tilde F \cong F \circ \incl$.
\end{defn}

There is still a uniqueness property like that in Lemma~\ref{lem:infext-induc}.

\begin{lem}\label{lem:genuine-unique}
Let $F: \pse\scD \to \pse\scD'$ be a pseudotriangulated functor.  If $G: \scD \to \scD'$ is a functor of triangulated categories such that $\incl \circ G \cong F \circ \incl$, then $G \cong \tilde F$.
\end{lem}
\begin{proof}
Composing on both sides with $\pg$ and using Lemma~\ref{lem:pseudotri-basic}\eqref{it:pseudotri-pg}, we find that $\pg \circ \incl \circ G \circ \pg \cong \pg \circ F \circ \incl \circ \pg \cong \tilde F \circ \pg$.  From~\eqref{eqn:infext-id}, we see that $G \circ \pg \cong \tilde F \circ \pg$, so $G \cong \tilde F$ by Lemma~\ref{lem:infext-induc}.
\end{proof}

Genuineness for functors is quite a subtle condition, and we will only be able to establish it when $\scD$ and $F$ obey rather strong constraints.  The next lemma tells us how this notion is related to genuineness for morphisms, but since the latter is not a natural property, it seems difficult to prove that a functor $F$ is genuine by reasoning directly with morphisms.  Instead, our strategy will be to seek indirect ways of showing that $\incl \circ \tilde F$ and $F \circ \incl$ are isomorphic.

\begin{lem}\label{lem:infext-gen-mor}
A pseudotriangulated functor $F: \pse\scD \to \pse\scD'$ is genuine if and only if it is isomorphic to a pseudotriangulated functor $F': \pse\scD \to \pse\scD'$ that sends genuine morphisms to genuine morphisms.
\end{lem}
\begin{proof}
If $F$ sends genuine morphisms to genuine morphisms, then it is easy to see that $\incl \circ \pg \circ F \circ \incl \cong F \circ \incl$.  In other words, $\incl \circ \tilde F \cong F \circ \incl$, so $F$ is genuine.  For the converse, suppose $F$ is genuine, and fix an isomorphism $\theta: \incl \circ \tilde F \to F \circ \incl$.  For a morphism $f_0: X \to Y$ in $\scD$, we have a commutative diagram
\[
\xymatrix@C=60pt{
\tilde F(X) \ar[r]^{(\tilde F(f_0), 0)} \ar[d]_{\theta_X} &
  \tilde F(Y) \ar[d]^{\theta_Y} \\
F(X) \ar[r]_{F(f_0,0)} & F(Y) }
\]
Form the analogous diagram for another morphism $f': X \to Y[-1]$ in $\scD$.  Applying the natural transformation $\iinf: \Hom_{\scD}(F(X),F(Y[-1])) \to \Hom_{\pse\scD}(F(X), F(Y))$ to that diagram, we obtain
\[
\xymatrix@C=60pt{
\tilde F(X) \ar[r]^{(0,\tilde F(f'))} \ar[d]_{\theta_X} &
  \tilde F(Y) \ar[d]^{\theta_Y} \\
F(X) \ar[r]_{\iinf(\pg F(f',0))} & F(Y) }
\]
Since $F$ commutes with $\iinf \circ \pg$, we have $\iinf(\pg F(f',0)) = F(0,f')$.  Combining the two diagrams, we find that
\[
\xymatrix@C=60pt{
\tilde F(X) \ar[r]^{(\tilde F(f_0), \tilde F(f'))} \ar[d]_{\theta_X} &
  \tilde F(Y) \ar[d]^{\theta_Y} \\
F(X) \ar[r]_{F(f_0,f')} & F(Y) }
\]
commutes.  Let $F': \pse\scD \to \pse\scD'$ be the functor given by $F'(X) = F(X)$ for objects $X$, and by $F'(f) = (\tilde F(f_0), \tilde F(f'))$ for morphisms $f = (f_0,f')$.  The commutative diagram above shows that the collection $\{\theta_X\}$ can be regarded as an isomorphism of functors $\theta: F' \simto F$.  Moreover, $F'$ clearly sends genuine morphisms to genuine morphisms.
\end{proof}

\begin{lem}\label{lem:infext-gen-itr}
If $F: \pse\scD \to \pse\scD'$ is genuine, then $\itr \circ F \cong \tilde F \circ \itr$.
\end{lem}
\begin{proof}
It is clear that for an object $X$ in $\pse\scD$, we have $\itr(F(X)) \cong \tilde F(\itr(X)) \cong \tilde F(X) \oplus \tilde F(X)[-1]$.  By Lemma~\ref{lem:infext-gen-mor}, we may assume that $F$ sends genuine morphisms to genuine morphisms.  Indeed, we may assume that for a morphism $f = (f_0,f')$ in $\pse\scD$, we have $F(f) = (\tilde F(f_0), \tilde F(f'))$.   The result follows from the observation that
\[
\itr(F(f)) = 
\begin{bmatrix}
\tilde F(f_0) & \\ \tilde F(f') & \tilde F(f_0[-1])
\end{bmatrix}
= \tilde F(\itr(f)).\qedhere
\]
\end{proof}

\begin{lem}\label{lem:infext-ind-adjoint}
Let $F: \pse\scD \to \pse\scD'$ and $G: \pse\scD' \to \pse\scD$ be a pair of pseudotriangulated functors.  If $F$ is left adjoint to $G$, then the induced functor $\tilde F$ is left adjoint to $\tilde G$.
\end{lem}
\begin{proof}
We begin by showing that the adjunction morphism
\[
\Phi: \Hom_{\pse\scD'}(F(X), Y) \simto \Hom_{\pse\scD}(X, G(Y))
\]
has the property that $\Phi(f)$ is infinitesimal if and only if $f$ is infinitesimal.  Let $\eta: \id_{\pse\scD'} \to G \circ F$ be the unit of the adjunction, and recall that $\Phi$ is given by $\Phi(f) = G(f) \circ \eta_X$.  If $f$ is infinitesimal, then $G(f)$ is infinitesimal, and then any composition with $G(f)$ is also infinitesimal.  The opposite implication is similar, using the fact that $\Phi^{-1}(g) = \epsilon_Y \circ F(g)$, where $\epsilon: F \circ G \to \id_{\pse\scD}$ is the counit.

Now $\Hom_{\scD'}(\tilde F(X), Y) \cong \Hom_{\scD'}(\pg F(\incl X), Y)$ is canonically isomorphic to the quotient of $\Hom_{\pse\scD'}(F(X),Y)$ by the subspace of infinitesimal morphisms.  The same holds for $\Hom_{\scD}(X, \tilde G(Y))$, so we see that $\Phi$ induces a canonical isomorphism $\Hom_{\scD'}(\tilde F(X),Y) \simto \Hom_{\scD}(X, \tilde G(Y))$.
\end{proof}

\begin{thm}\label{thm:adjoint-genuine}
Let $F: \pse\scD \to \pse\scD'$ be a genuine pseudotriangulated functor.  If $F$ has a right adjoint (resp.~left adjoint) pseudotriangulated functor $G: \pse\scD' \to \pse\scD$, then $G$ is also genuine. 
\end{thm}
\begin{proof}
We will treat the case where $G$ is right adjoint to $F$; the other case is similar.  By Lemmas~\ref{lem:infext-biadj} and~\ref{lem:infext-ind-adjoint}, $\incl \circ \tilde G$ is right adjoint to $\tilde F \circ \itr[1]$, and $G \circ \incl$ is right adjoint to $\itr[1] \circ F$.  But $\tilde F \circ \itr[1] \cong \itr[1] \circ F$ by Lemma~\ref{lem:infext-gen-itr}.  Since adjoint functors are unique up to isomorphism, it follows that $\incl \circ \tilde G \cong G \circ \incl$.
\end{proof}

\section{Homotopy categories of Orlov categories}
\label{sect:hot-orlov}

Let $\scA$ and $\scB$ be two additive categories, and consider their bounded homotopy categories $\Kb(\scA)$ and $\Kb(\scB)$.  In the sequel, we will encounter the problem of showing that two functors $F,F': \Kb(\scA) \to \Kb(\scB)$ are isomorphic without having any explicit way to construct a morphism between them.

The main results of this section (Theorems~\ref{thm:orlov} and~\ref{thm:orlov-inf}) give us a way to solve this problem, provided that the additive categories $\scA$ and $\scB$ satisfy the conditions of the following definition.  The idea of using properties of the categories to prove an isomorphism of functors is due to Orlov~\cite{orl:edc}.

\begin{defn}\label{defn:orlov}
Let $\scA$ be an additive category equipped with a function $\deg: \Ind(\scA) \to \Z$. $\scA$ is said to be an \emph{Orlov category} if the following conditions hold:
\begin{enumerate}
\item All $\Hom$-spaces in $\scA$ are finite-dimensional.\label{it:orlov-homfin}
\item For any $S \in \Ind(\scA)$, we have $\End(S) \cong \Bbbk$.\label{it:orlov-endind}
\item If $S, S' \in \Ind(\scA)$ with $\deg(S) \le \deg(S')$ and $S \not\cong S'$, then $\Hom(S,S') = 0$.\label{it:orlov-deg}
\end{enumerate}
An object $X \in \scA$ is said to be \emph{homogeneous} of degree $n$ if it is isomorphic to a direct sum of indecomposable objects of degree $n$.  An additive functor $F: \scA \to \scB$ between two Orlov categories is said to be \emph{homogeneous} if it takes homogeneous objects of degree $n$ in $\scA$ to homogeneous objects of degree $n$ in $\scB$.
\end{defn}

It follows from conditions~\eqref{it:orlov-homfin} and~\eqref{it:orlov-endind} above that any Orlov category is \emph{Karoubian} (every idempotent endomorphism splits) and \emph{Krull--Schmidt} (every object is a direct sum of finitely many indecomposable objects, whose isomorphism classes and multiplicities are uniquely determined).

\subsection{Preliminaries on Orlov categories}

We first require some additional notation and lemmas.  For an object $X = (X^\bullet, d_X) \in \Kb(\scA)$, let us define the \emph{support} of $X$ to be the subset $\supp X \subset \Z \times \Z$ such that
\[
(i,j) \in \supp X\qquad \text{if and only if}\qquad
\txt{$X^i$ contains a nonzero homogeneous\\ direct summand of degree $j$.}
\]
Note that this notion is not homotopy-invariant: isomorphic objects of $\Kb(\scA)$ may have different supports.  For any subset $\Sigma \subset \Z \times \Z$, let $\Kb(\scA)_\Sigma$ denote the following full subcategory of $\Kb(\scA)$:
\[
\Kb(\scA)_\Sigma = \{X \in \Kb(\scA) \mid
\text{$X$ is isomorphic to an object $X'$ with $\supp X' \subset \Sigma$} \}.
\]
Clearly, every object of $\Kb(\scA)$ belongs to some $\Kb(\scA)_\Sigma$ with $\Sigma$ finite.  Let us endow $\Z \times \Z$ with the lexicographic order:
\begin{equation}\label{eqn:lexico}
(i,j) \leq (i',j')
\qquad
\text{if $i < i'$, or if $i = i'$ and $j \le j'$.}
\end{equation}
With respect to this order, any finite set $\Sigma \subset \Z \times \Z$ has a largest element.

\begin{lem}\label{lem:orlov-summand}
Let $X = (X^\bullet, d_X) \in \Kb(\scA)$, and suppose $X \cong A[n] \oplus Y$, where $A$ is an object of $\scA$.  Then $Y$ is isomorphic to a chain complex $Y' = (Y'{}^\bullet, d_{Y'})$ with $\supp Y' \subset \supp X$.
\end{lem}
\begin{proof}
Let $i: A[n] \to X$ and $p: X \to A[n]$ be the inclusion and projection maps coming from the given direct sum decomposition.  Let us represent these by chain maps: $i = (i^k)_{k \in \Z}$ and $p = (p^k)_{k \in \Z}$.  The statement that $p \circ i = \id_{A[n]}$ in $\Kb(\scA)$ is equivalent to the statement that the chain map $(p^k \circ i^k)_{k \in \Z}$ is homotopic to $\id_{A[n]}$.  But $A[n]$ clearly admits no nonzero homotopies, so the composition $(p^k \circ i^k)_{k \in \Z}$ is equal to $\id_{A[n]}$ as a chain map.  In particular, $p^{-n} \circ i^{-n} = \id_A$.  It follows that $i^{-n} \circ p^{-n} \in \End(X^{-n})$ is an idempotent.

Recall that $\scA$ is a Karoubian category.  Therefore, there is some object $B \in \scA$ such that we can identify $X^{-n} \cong A \oplus B$, and such that under this identification, $p^{-n}$ and $i^{-n}$ are the projection and inclusion maps for the direct summand $A$.  With respect to this direct-sum decomposition, we may write the differentials $d_X^{-n-1}: X^{-n-1} \to X^{-n}$ and $d_X^{-n}: X^{-n} \to X^{-n+1}$ as matrices
\[
d_X^{-n-1} = \begin{bmatrix} a \\ b \end{bmatrix}
\qquad\text{and}\qquad
d_X^{-n} = \begin{bmatrix} s & t \end{bmatrix}.
\]
Since $(p^\bullet)$ is a chain map, we see that $a = p^{-n} \circ d_X^{-n-1} = 0$.  Similarly, $s = d_X^{-n} \circ i^{-n} = 0$.  It follows that $t \circ b = d_X^{-n} \circ d_X^{-n-1} = 0$.  Define the chain complex $Y' = (Y'{}^\bullet, d_{Y'})$ by
\[
Y'{}^k =
\begin{cases}
X^k & \text{if $k \ne -n$,} \\
B & \text{if $k = n$,}
\end{cases}
\qquad
d_{Y'}^k =
\begin{cases}
d_X^k & \text{if $k \ne -n-1, -n$,} \\
b & \text{if $k = -n-1$,} \\
t & \text{if $k = -n$.}
\end{cases}
\]
$Y'$ has support contained in that of $X$, and we clearly have $X \cong A[n] \oplus Y'$, so $Y'$ becomes isomorphic to $Y$ after passing to the homotopy category $\Kb(\scA)$.
\end{proof}

\begin{lem}\label{lem:homog-coker}
Let $f: A \to B$ be a morphism in an Orlov category $\scA$, and assume that $B$ is homogeneous of degree $n$.  Then $f$ has a ``homogeneous cokernel.''  That is, there is a morphism $q: B \to Q$, where $Q$ is also homogeneous of degree $n$, such that
\begin{enumerate}
\item We have $q \circ f = 0$.
\item If $g: B \to C$ is any morphism such that $g \circ f = 0$ and $C$ is homogeneous of degree $n$, then there is a unique morphism $r: Q \to C$ such that $g = r \circ q$.
\end{enumerate}
In fact, there is an isomorphism $u: Q^\perp \oplus Q \to B$ (for some homogeneous object $Q^\perp$) such that $q \circ u$ is simply the canonical projection map $Q^\perp \oplus Q \to Q$.
\end{lem}
\begin{proof}
Let $I_n = \{S \in \Ind(\scA) \mid \deg S = n\}$.  For any $S \in I_n$, we have a natural pairing
\[
\Hom(S,B) \otimes \Hom(B,S) \to \Hom(S,S) \cong \Bbbk.
\]
This pairing is nondegenerate: to see this, it suffices to consider the special case where $B$ is itself indecomposable, and in that case, the nondegeneracy is obvious from condition~\eqref{it:orlov-deg} of Definition~\ref{defn:orlov}.  Let us write $H_S = \Hom(B,S)$ for brevity.  The pairing above gives us a canonical isomorphism $H_S^* \cong \Hom(S,B)$.  (Here $H_S^* = \Hom(H_S,\Bbbk)$.)

Recall that in any $\Bbbk$-linear additive category, it makes sense to form tensor products of objects with finite-dimensional $\Bbbk$-vector spaces.  Note that $H_S$ is always finite-dimensional, and it vanishes for all but finitely many $S$, so the direct sum $\bigoplus_{S \in I_n} H_S^* \otimes S$ is a well-defined object of $\scA$.   We claim that there is a canonical isomorphism
\begin{equation}\label{eqn:homog-canon-form}
\bigoplus_{S \in I_n} H_S^* \otimes S \simto B.
\end{equation}
Indeed, there is a natural map $\bigoplus_{S \in I_n} \Hom(S,B) \otimes S \to B$; this map is evidently an isomorphism when $B$ is indecomposable, so it is an isomorphism in general.

For each $S \in I_n$, let $E_S$ denote the kernel of the map $\Hom(B,S) \to \Hom(A,S)$ induced by $f$.  Let $i_S: E_S \to H_S$ denote the inclusion map, and let $q_S: H_S^* \to E_S^*$ denote its dual.  Next, let
\[
Q = \bigoplus_{S \in I_n} E_S^* \otimes S,
\]
and let $q: B \to Q$ be the map given by $q = \bigoplus q_S \otimes \id_S$ (using the identification~\eqref{eqn:homog-canon-form}).  To describe $q$ another way, consider the chain of natural isomorphisms
\[
\Hom(B,Q) \cong \bigoplus_{S \in I_n} E_S^* \otimes \Hom(B,S) \cong \bigoplus_{S \in I_n} \Hom(E_S, \Hom(B,S)).
\]
We have $i_S \in \Hom(E_S, \Hom(B,S))$, and under these identifications, we have $q = \bigoplus i_S$.  Observe that the map $f$ gives rise to a commutative diagram
\[
\xymatrix{
\Hom(B,Q) \ar@{=}[r]^-{\sim} \ar[d]_{\cdot \circ f} &
  \bigoplus_{S \in I_n} \Hom(E_S, \Hom(B, S)) \ar[d] \\
\Hom(A,Q) \ar@{=}[r]^-{\sim} &
  \bigoplus_{S \in I_n} \Hom(E_S, \Hom(A, S)) }
\]
By the definition of $E_S$, we see that each $i_S$ is in the kernel of the map
\[
\Hom(E_S, \Hom(B,S)) \to \Hom(E_S, \Hom(A,S)).
\]
It follows that $q$ is in the kernel of $\Hom(B,Q) \to \Hom(A,Q)$.  In other words, $q \circ f = 0$, as desired.

It is easy to see from the above construction that for any $T \in I_n$, there are natural isomorphisms $\Hom(Q,T) \cong E_T$ and $\Hom(B,T) \cong H_T$, and that the map $\Hom(Q,T) \to \Hom(B,T)$ induced by $q$ is none other than $i_T: E_T \to H_T$.  In other words, if $g \in \Hom(B,T)$ is any morphism in $E_T$, i.e., such that $g \circ f = 0$, then there is a unique morphism $r \in \Hom(Q,T)$ such that $g = r \circ q$.  Thus, we have just proved a special case of the desired universal property of $q$.  Since the universal property holds for indecomposable objects, it holds in general.

Finally, each $q_S: H_S^* \to E_S^*$ is a surjective map of vector spaces, so there certainly exists some isomorphism $u_S: (\ker q_S) \oplus E_S^* \simto H_S^*$ such that $q_S \circ u_S$ is the projection map onto the second direct summand.  Let $Q^\perp = \bigoplus (\ker q_S) \otimes S$.  Then there is an obvious isomorphism $u: Q^\perp \oplus Q \to B$ such that $q \circ u$ is also a projection map.
\end{proof}

\begin{cor}\label{cor:homog-coker-idemp}
Let $f: A \to B$ be a morphism in an Orlov category $\scA$, and assume that $B$ is homogeneous of degree $n$.  There is an idempotent endomorphism $\theta: B \to B$ such that for any morphism $g: B \to C$ with $C$ also homogeneous of degree $n$, the following two conditions are equivalent:
\begin{enumerate}
\item $g \circ \theta = g$.
\item $g \circ f = 0$.
\end{enumerate}
\end{cor}
\begin{proof}
Let $q: B \to Q$ and $u: Q^\perp \oplus Q \simto B$ be as in Lemma~\ref{lem:homog-coker}, and let $i: Q \to Q^\perp \oplus Q$ be the inclusion map of the second summand.  Then $\theta = u \circ i \circ q$ has the required properties.
\end{proof}

\begin{lem}\label{lem:dt-orlov}
Let $\Sigma \subset \Z \times \Z$ be a finite set with largest element $(i,j)$, and let $\Sigma' = \Sigma \smallsetminus \{(i,j)\}$.
\begin{enumerate}
\item For any object $X \in \Kb(\scA)_\Sigma$, there is a distinguished triangle\label{it:dt-orlov-exist}
\[
P \to X \to Y \to P[1]
\]
with $P \in \Kb(\scA)_{\{(i,j)\}}$ and $Y \in \Kb(\scA)_{\Sigma'}$.
\item Suppose we have a commutative diagram in $\Kb(\scA)$ as follows, in which the horizontal rows are distinguished triangles:\label{it:dt-orlov-unique}
\begin{equation}\label{eqn:dt-orlov}
\vcenter{\xymatrix{
P \ar[r]^f\ar[d]_p & X \ar[r]^g & Y \ar[r]^h\ar[d]^r & P[1] \ar[d]^{p[1]} \\
P' \ar[r]_{f'} & X' \ar[r]_{g'} & Y' \ar[r]_{h'} & P'[1] }}
\end{equation}
If $P, P' \in \Kb(\scA)_{\{(i,j)\}}$, $X, X' \in \Kb(\scA)_\Sigma$, and $Y,Y' \in \Kb(\scA)_{\Sigma'}$, then there is a unique map $q: X \to X'$ that makes the above diagram commute.
\item Let $P \ovto{f} X \to Y \to P[1]$ be a distinguished triangle, and let $f': P' \to X$ be any morphism, where $P, P' \in \Kb(\scA)_{\{(i,j)\}}$, $X \in \Kb(\scA)_\Sigma$, and $Y \in \Kb(\scA)_{\Sigma'}$.  Form the morphism $P \oplus P' \to X$, and complete it to a distinguished triangle \label{it:dt-orlov-cone}
\[
\xymatrix{P \oplus P' \ar[r]^-{[\begin{smallmatrix} f & f'\end{smallmatrix}]} &
 X \ar[r] & Z \ar[r] & (P \oplus P')[1].}
\]
Then we have $Z \in \Kb(\scA)_{\Sigma''}$, where $\Sigma'' = \Sigma' \cup \{(i-1,j)\}$.
\end{enumerate}
\end{lem}
\begin{proof}
\eqref{it:dt-orlov-exist}~We may assume that $X$ is a chain complex $(X^\bullet, d_X)$ with $\supp X^\bullet \subset \Sigma$.  Choose a direct-sum decomposition $X^i = A \oplus B$, where $A$ has degrees${}<j$ and $B$ is homogeneous of degree $j$.  With respect to this direct-sum decomposition, we may write the differential $d_X^{i-1}: X^{i-1} \to X^i$ as a matrix $d_X^{i-1} =
[\begin{smallmatrix} a \\ b \end{smallmatrix}]$.  Let $Y = (Y^\bullet,d_Y)$ be the chain complex given by
\[
Y^k = \begin{cases}
X^k & \text{if $k \ne i$,} \\
A & \text{if $k = i$}
\end{cases}
\qquad\text{and}\qquad
d_Y^k = \begin{cases}
d_X^k & \text{if $k \ne i, i-1$,} \\
a & \text{if $k = i - 1$,} \\
0 & \text{if $k = i$.}
\end{cases}
\]
Clearly, $Y \in \Kb(\scA)_{\Sigma'}$.  Let $P = (P^\bullet, d_P)$ be the complex given by
\[
P^k = \begin{cases}
0 & \text{if $k \ne i$,} \\
B & \text{if $k = i$}
\end{cases}
\qquad\text{and}\qquad
d_P = 0.
\]
We clearly have $P \in \Kb(\scA)_{\{(i,j)\}}$.  Consider the morphism $\delta: Y[-1] \to P$ where $\delta^i: Y^{i-1} \to P^i$ is the map $b$.  It is easy to see that the cone of $\delta$ is isomorphic to $X$, so we have a distinguished triangle $P \to X \to Y \to$, as desired.

\eqref{it:dt-orlov-unique}~Assume that $P$ and $P'$ (resp.~$Y$ and $Y'$) are represented by chain complexes whose support is contained in the set $\{(i,j)\}$ (resp.~$\Sigma'$).  We may assume that the terms of the chain complex $X = (X^\bullet, d_X)$ can be identified with terms of $Y$ and $P$ as follows:
\[
X^k \cong
\begin{cases}
Y^k & \text{if $k \ne i$,} \\
Y^i \oplus P^i & \text{if $k = i$.}
\end{cases}
\]
Let us denote the inclusion and projection maps for the direct sum $X^i \cong Y^i \oplus P^i$ as follows:
\[
\xymatrix{
Y^i \ar@<0.5ex>[r]^{\iota_1} & X^i \ar@<0.5ex>[l]^{\pi_1} }
\qquad\qquad
\xymatrix{
P^i \ar@<0.5ex>[r]^{\iota_2} & X^i \ar@<0.5ex>[l]^{\pi_2} }
\]
The maps $f: P \to X$, $g: X \to Y$ are then given by
\[
f^k = \begin{cases}
0 & \text{if $k \ne i$,} \\
\iota_2 & \text{if $k = i$,}
\end{cases}
\qquad\qquad
g^k = \begin{cases}
\id & \text{if $k \ne i$,} \\
\pi_1 & \text{if $k = i$.}
\end{cases}
\]
We fix analogous identifications for the objects and morphisms in the triangle $P' \to X' \to Y' \to$.

The existence of $q$ follows from general properties of triangulated categories.  For uniqueness, it is sufficient to consider the special case where $p = 0$ and $r = 0$.  Suppose $q: X \to X'$ is a map making~\eqref{eqn:dt-orlov} commute; we must show that $q = 0$.  Since $g' \circ q = 0$, $q$ must factor through $f'$.  Let $\tilde q: X \to P'$ be a map such that $q = f' \circ \tilde q$.  We may assume that this equality holds at the level of chain maps (not just up to homotopy).  In particular,
\[
q^k =
\begin{cases}
0 & \text{if $k \ne i$,} \\
\iota'_2 \circ \tilde q^i & \text{if $k = i$.}
\end{cases}
\]
Now, the map $\tilde q^i: Y^i \oplus P^i \to P'{}^i$ can be written as a matrix $\tilde q^i = \begin{bmatrix} 0 & a \end{bmatrix}$, where the left-hand entry is $0$ because $\Hom(Y^i, P'{}^i) = 0$ by part~\eqref{it:orlov-deg} of Definition~\ref{defn:orlov}.  We therefore have
\[
q^i = \iota'_2 \circ \tilde q^i =
\begin{bmatrix}
0 & 0 \\
0 & a
\end{bmatrix}
\qquad
\text{with $a: P^i \to P'{}^i$.}
\]
Using these identifications $X^i \cong Y^i \oplus P^i$ and $X'{}^i \cong Y'{}^i \oplus P'{}^i$, we can write the differentials $d_X^{i-1}$ and $d_{X'}^{i-1}$ as matrices
\[
d_{X}^{i-1} =
\begin{bmatrix}
u \\ v
\end{bmatrix}
\qquad\text{and}\qquad
d_{X'}^{i-1} =
\begin{bmatrix}
u' \\ v'
\end{bmatrix}.
\]
Now, $q$ is a chain map, and since $q^{i-1} = 0$, we must have $q^i \circ d_X^{i-1} = 0$, or
\[
q^i \circ d_X^{i-1} =
\begin{bmatrix}
0 & 0 \\
0 & a
\end{bmatrix}
\begin{bmatrix}
u \\ v
\end{bmatrix}
=
\begin{bmatrix}
0 \\
av
\end{bmatrix}
= 0.
\]
We will now make use of the ``homogeneous cokernel'' of $v: X^{i-1} \to P^i$.  Corollary~\ref{cor:homog-coker-idemp} lets us associate to this map an idempotent endomorphism $\theta: P^i \to P^i$.  Since $a \circ v = 0$, we have $a \circ \theta = a$, and so
\begin{equation}\label{eqn:dt-orlov-idemp}
q^i \circ \iota_2 \circ \theta =
\begin{bmatrix}
0 \\ a
\end{bmatrix}
\begin{bmatrix} \theta \end{bmatrix}
= 
\begin{bmatrix}
0 \\ a
\end{bmatrix}
= q^i \circ \iota_2.
\end{equation}

Next, we have $q \circ f = 0$.  This is a statement about the existence of a certain homotopy; specifically, there is a map $h^i: P^i \to X'{}^{i-1}$ such that
\begin{equation}\label{eqn:dt-orlov-homotop}
q^i \circ \iota_2 = d_{X'}^{i-1} \circ h^i.
\end{equation}
Of course, the homotopy is not unique in general.  Indeed, in view of~\eqref{eqn:dt-orlov-idemp}, we could replace $h^i$ by $h^i \circ \theta$, and~\eqref{eqn:dt-orlov-homotop} would still hold.  By carrying out this replacement if necessary, we henceforth assume that $h^i \circ \theta = h^i$, or equivalently,
\begin{equation}\label{eqn:dt-orlov-homotop2}
h^i \circ v = 0.
\end{equation}

Let us define a collection of maps $\tilde h^k: X^k \to X'{}^{k-1}$ by
\[
\tilde h^k =
\begin{cases}
0 & \text{if $k \ne i$,} \\
\big[\begin{matrix} 0 & h^i \end{matrix}\big]
: Y^i \oplus P^i \to X'{}^{i-1} & \text{if $k = i$.}
\end{cases}
\]
We claim that
\begin{equation}\label{eqn:dt-orlov-null-homotop}
q^k = \tilde h^{k+1} \circ d_X^k + d_{X'}^{k-1} \circ \tilde h^k
\end{equation}
for all $k$.  If $k < i-1$ or $k > i$, both sides are obviously $0$. If $k = i$, we have $\tilde h^{i+1} = 0$, so this is essentially a restatement of~\eqref{eqn:dt-orlov-homotop}.  Finally, if $k = i-1$, we know that $q^{i-1} = 0$ and $\tilde h^{i-1} = 0$; we must check that $\tilde h^i \circ d_X^{i-1} = 0$.  But this follows from~\eqref{eqn:dt-orlov-homotop2}.  We see from~\eqref{eqn:dt-orlov-null-homotop} that $q$ is null-homotopic, as desired.

\eqref{it:dt-orlov-cone}~Consider the following octahedral diagram, which is associated with the composition $P \to P \oplus P' \to X$:
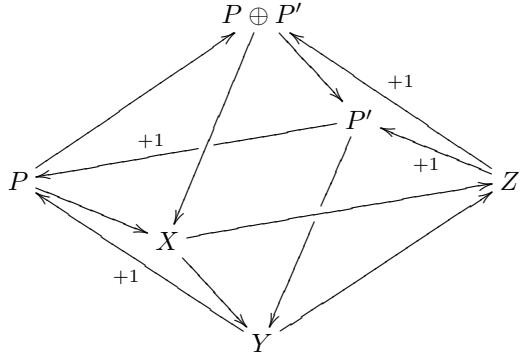
\[
\xymatrix@=10pt{
&&&& P \oplus P' \ar[ddddl]\ar[ddr] \\ \\
&&&&& P' \ar[dlllll]|!{[uul];[ddll]}\hole_(.6){+1}
  \ar[ddddl]|!{[ddll];[drrr]}\hole \\
P \ar[uuurrrr]\ar[drrr] &&&&&&&& Z \ar[uuullll]_{+1}
\ar[ulll]^{+1} \\
&&& X \ar[urrrrr]\ar[ddr] \\ \\
&&&& Y \ar[uuullll]^{+1}\ar[uuurrrr] }
\]
We see that there is a distinguished triangle $Y \to Z \to P'[1] \to$.  The category $\Kb(\scA)_{\Sigma''}$ is stable under extensions, and since $Y$ and $P'[1]$ both belong to it, it follows that $Z \in \Kb(\scA)_{\Sigma''}$ as well.
\end{proof}

Part~\eqref{it:dt-orlov-unique} of Lemma~\ref{lem:dt-orlov} has an analogue in the infinitesimal extension $\pse\Kb(\scA)$.  Recall that the objects of this category are the same as those of $\Kb(\scA)$, so the notion of support makes sense here as well.  The full subcategories $\pse\Kb(\scA)_{\Sigma}$ (for $\Sigma \subset \Z \times \Z$) are defined in the same way as $\Kb(\scA)_{\Sigma}$.

\begin{lem}\label{lem:dt-orlov-inf}
Suppose we have a commutative diagram in $\pse\Kb(\scA)$ as follows, in which the horizontal rows are distinguished triangles:
\begin{equation}\label{eqn:dt-orlov-inf}
\vcenter{\xymatrix{
P \ar[r]^f\ar[d]_p & X \ar[r]^g & Y \ar[r]^h\ar[d]^r & P[1] \ar[d]^{p[1]} \\
P' \ar[r]_{f'} & X' \ar[r]_{g'} & Y' \ar[r]_{h'} & P'[1] }}
\end{equation}
If $P, P' \in \pse\Kb(\scA)_{\{(i,j)\}}$, $X, X' \in \pse\Kb(\scA)_\Sigma$, and $Y,Y' \in \pse\Kb(\scA)_{\Sigma'}$, then there is a unique map $q: X \to X'$ that makes the above diagram commute.
\end{lem}
\begin{proof}
The existence of $q$ follows from Lemma~\ref{lem:infext-sq-comp}; we must prove uniqueness.  As in Lemma~\ref{lem:dt-orlov}\eqref{it:dt-orlov-unique}, it suffices to consider the case where $p = 0$ and $r = 0$.  Furthermore, every distinguished triangle is, by definition, isomorphic to a diagram obtained by applying $\incl: \Kb(\scA) \to \pse\Kb(\scA)$ to a distinguished triangle in $\Kb(\scA)$, so we may assume that the morphisms $f, g, h, f', g', h'$ are all genuine.  

Suppose $q: X \to X'$ makes the diagram commute.  Since $g' \circ q = 0$, it follows from Lemma~\ref{lem:infext-hom-les} that $q$ factors through $f'$.  Let $\tilde q: X \to P'$ be a map such that $q = f' \circ \tilde q$.  It is clear from the support assumptions that $\Hom_{\Kb(\scA)}(X, P'[-1]) = 0$, so in fact, $\tilde q$ must be genuine; it can have no nonzero infinitesimal component.  The same must then hold for $q$.  Since every morphism in our diagram is in the image of $\incl$, we have actually reduced the problem to the setting of Lemma~\ref{lem:dt-orlov}\eqref{it:dt-orlov-unique}, where the desired uniqueness is already known.
\end{proof}

\subsection{Morphisms of functors}
\label{ss:orlov-main}

We are now ready to prove the main results of this section.  Their proofs are adaptations of an argument due to Orlov~\cite[Proposition~2.16]{orl:edc}.

\begin{thm}\label{thm:orlov}
Let $\scA$ and $\scB$ be two Orlov categories.  Let $F,F': \Kb(\scA) \to \Kb(\scB)$ be two functors of triangulated categories.  Assume that $F(\scA) \subset \scB$ and $F'(\scA) \subset \scB$, and that the induced functors $F|_{\scA}, F'|_{\scA}: \scA \to \scB$ are homogeneous.  Any morphism of additive functors
\[
\theta^\circ: F|_{\scA} \to F'|_{\scA}
\]
can be extended to a morphism $\theta: F \to F'$ of functors of triangulated categories in such a way that if $\theta^\circ$ is an isomorphism, then $\theta$ is as well.
\end{thm}

\begin{rmk}\label{rmk:orlov}
The statement of the theorem is equivalent to the following fact.  Let $F: \Kb(\scA) \to \Kb(\scB)$ be a functor that satisfies the hypotheses of the theorem, and consider the functor $F' := \Kb(F|_{\scA}): \Kb(\scA) \to \Kb(\scB)$ induced by $F|_{\scA}: \scA \to \scB$.  Then there is an isomorphism of functors $F \simto F'$.
\end{rmk}

There is an analogous statement in which the codomain category is replaced by an infinitesimal extension.  Note that for objects $X, Y \in \scB$, we have that $\Hom(X,Y[-1]) = 0$ in $\Kb(\scB)$.  Therefore, the composition $\scB \hookrightarrow \Kb(\scB) \ovto{\incl} \pse\Kb(\scB)$ is full and faithful.  In other words, we can identify $\scB$ with a full subcategory of $\pse\Kb(\scB)$ just as we do with $\Kb(\scB)$.  

\begin{thm}\label{thm:orlov-inf}
Let $\scA$ and $\scB$ be two Orlov categories.  Let $F,F': \Kb(\scA) \to \pse\Kb(\scB)$ be two pseudotriangulated functors.  Assume that $F(\scA) \subset \scB$ and $F'(\scA) \subset \scB$, and that the induced functors $F|_{\scA}, F'|_{\scA}: \scA \to \scB$ are homogeneous.  Any morphism of additive functors
\[
\theta^\circ: F|_{\scA} \to F'|_{\scA}
\]
can be extended to a morphism $\theta: F \to F'$ of pseudotriangulated functors in such a way that if $\theta^\circ$ is an isomorphism, then $\theta$ is as well.
\end{thm}

\begin{rmk}\label{rmk:orlov-inf}
Equivalently, this theorem says that any functor $F: \Kb(\scA) \to \pse\Kb(\scB)$ satisfying the hypotheses of the theorem is isomorphic to the composition
\[
\Kb(\scA) \xto{\Kb(F|_{\scA})}  \Kb(\scB) \ovto{\incl} \pse\Kb(\scB).
\]
\end{rmk}

\begin{proof}[Proof of Theorems~\ref{thm:orlov} and~\ref{thm:orlov-inf}]
Constructing a morphism of functors $\theta: F \to F'$ consists of the following three steps:
\begin{enumerate}
\item For each object $X \in \Kb(\scA)$, construct a morphism $\theta_X: F(X) \to F'(X)$ in $\Kb(\scB)$ or $\pse\Kb(\scB)$, and show that it is an isomorphism if $\theta^\circ$ is.
\item Show that $\theta_X$ is independent of choices in the construction.
\item Show that for any morphism $s: X \to X'$ in $\Kb(\scA)$, we have $F'(s) \circ \theta_X = \theta_{X'} \circ F(s)$.
\end{enumerate}
We will carry out these steps by an induction argument involving the support of an object.  We say that a subset $\Sigma \subset \Z \times \Z$ is a \emph{paragraph} if it is of the form
\[
\Sigma = \big(\{a, a+1, \ldots, b-1\} \times \{c, c+1, \ldots, d\}\big) \cup \{(b,c),(b,c+1), \ldots, (b,e)\}
\]
for some $e$ with $c \le e \le d$.  We also say that such a paragraph $\Sigma$ has $b-a+1$ \emph{lines}.  For $\Sigma$ as above, we see that the largest element is $(b,e)$.  Note that if $\Sigma$ has at least $2$ lines, then $(b-1,e) \in \Sigma$ as well.  Obviously, the support of any object is contained in some paragraph.

To begin the induction, let $\Sigma$ be a paragraph with a single line, so that $\Sigma \subset \{n \} \times \Z$ for some $n \in \Z$.  For any object $X \in \Kb(\scA)_{\Sigma}$, we have $X[n] \in \scA$, so we have available a morphism $\theta^\circ_{X[n]}: F(X[n]) \to F'(X[n])$.  Define $\theta_X : F(X) \to F'(X)$ by $\theta_X = \theta^\circ_{X[n]}[-n]$.  Trivially, statements~(1)--(3) hold for objects $X, X'$ whose support is contained in $\{n\} \times \Z$.  Moreover, $\theta_X$ is an isomorphism of objects if $\theta^\circ$ is an isomorphism of functors.

For the inductive step, let us assume that $\Sigma$ is a paragraph with at least two lines.  Let $(b,e)$ denote its largest element, and let $\Sigma' = \Sigma \smallsetminus \{(b,e)\}$.  Then $\Sigma'$ is also a paragraph, and it contains $(b-1,e)$.  Assume that steps~(1)--(3) above have already been carried out for objects and morphisms of $\Kb(\scA)_{\Sigma'}$.  We will now carry them out for objects and morphisms of $\Kb(\scA)_{\Sigma}$.

{\it Step~1}.  For an object $X \in \Kb(\scA)_{\Sigma}$, we can find, by Lemma~\ref{lem:dt-orlov}\eqref{it:dt-orlov-exist}, a triangle
\begin{equation}\label{eqn:dt-orlov-choice}
P \to X \to Y \ovto{\delta} P[1]
\end{equation}
with $\supp P \subset \{(b,e)\}$ and $\supp Y \subset \Sigma'$. Note that $\delta: Y \to P[1]$ is a morphism in $\Kb(\scA)_{\Sigma'}$, so we already have morphisms $\theta_Y$ and $\theta_{P[1]}$ such that the diagram
\[
\xymatrix{
F(Y) \ar[r]^{F(\delta)} \ar[d]_{\theta_Y} & F(P[1]) \ar[d]^{\theta_{P[1]}} \\
F'(Y) \ar[r]_{F'(\delta)} & F'(P[1]) }
\]
commutes.  We define $\theta_X: F(X) \to F'(X)$ by completing this square to a morphism of distinguished triangles (invoking Lemma~\ref{lem:infext-sq-comp} in the case of $\pse\Kb(\scB)$):
\begin{equation}\label{eqn:dt-orlov-theta}
\vcenter{\xymatrix{
F(P) \ar[r] \ar[d]_{\theta_P = \theta_{P[1]}[-1]} & F(X) \ar[r] \ar@{.>}[d]_{\theta_X} &
F(Y) \ar[r]^{F(\delta)} \ar[d]_{\theta_Y} & F(P[1]) \ar[d]^{\theta_{P[1]}} \\
F'(P) \ar[r] & F'(X) \ar[r] &
F'(Y) \ar[r]_{F'(\delta)} & F'(P[1]) }}
\end{equation}
Note that if $\theta^\circ$ is an isomorphism, then we know inductively that $\theta_P$ and $\theta_Y$ are isomorphisms, so it follows (perhaps by Lemma~\ref{lem:infext-sq-comp} again) that $\theta_X$ is as well.

{\it Step 2}.  We must now show that $\theta_X$ is independent of the choices made above.  Either Lemma~\ref{lem:dt-orlov}\eqref{it:dt-orlov-unique} or Lemma~\ref{lem:dt-orlov-inf} tells us that $\theta_X$ is uniquely determined once the triangle~\eqref{eqn:dt-orlov-choice} is fixed, but we must also prove independence of the choice of that triangle.  Let
\begin{equation}\label{eqn:dt-orlov-pp}
P' \to X \to Y' \to P'[1]
\end{equation}
be another such triangle, and let $\theta'_X: F(X) \to F'(X)$ be the morphism obtained from it by the construction above.  We must show that $\theta'_X = \theta_X$.

To do this, we will construct a third triangle as an intermediary.  Let $P'' = P \oplus P'$.  Consider the obvious map $P'' \to X$, and let $Y''$ denote its cone.  By Lemma~\ref{lem:dt-orlov}\eqref{it:dt-orlov-cone}, we have $Y'' \in \Kb(\scA)_{\Sigma'}$, and the construction above gives us a third morphism $\theta''_X: F(X) \to F'(X)$.  Note that $P \to X$ factors through $P'' \to X$, so we can form a morphism of triangles as follows:
\begin{equation}\label{eqn:dt-orlov-pchoice}
\vcenter{\xymatrix{
P \ar[r]\ar[d]_{f} & X \ar[r]\ar@{=}[d] & Y \ar[r]\ar[d]_{g} & P[1] \ar[d]^{f[1]} \\
P'' \ar[r] & X \ar[r] & Y'' \ar[r] & P''[1] }}
\end{equation}
Applying $F$ and $F'$ to this diagram, we obtain the following diagram:
\begin{equation}\label{eqn:dt-prism}
\vcenter{\tiny\xymatrix@=15pt{
F(P) \ar[rr] \ar[dr]^{F(f)} \ar[ddd]^(.7){\theta_{P}} && 
  F(X) \ar[rr] \ar@{=}[dr] \ar'[d][ddd]^(.55555){\theta_X} &&
  F(Y) \ar[rr] \ar[dr]^{F(g)} \ar'[d][ddd]^(.55555){\theta_Y}  && 
  F(P[1]) \ar[dr]^{F(f[1])} \ar'[d][ddd]^(.55555){\theta_{P[1]}} \\
& F(P'') \ar[rr] \ar[ddd]^(.3){\theta_{P''}} && F(X) \ar[rr] \ar[ddd]^(.3){\theta''_X} &&
  F(Y'') \ar[rr] \ar[ddd]^(.3){\theta_{Y''}} && 
  F(P''[1]) \ar[ddd]^(.3){\theta_{P''[1]}} \\ \\
F'(P) \ar'[r][rr] \ar[dr]^{F'(f)} && 
  F'(X) \ar'[r][rr] \ar@{=}[dr] &&
  F'(Y) \ar'[r][rr] \ar[dr]^{F'(g)} && F'(P[1]) \ar[dr]^{F'(f[1])} \\
& F'(P'') \ar[rr]  && F'(X) \ar[rr]  &&
  F'(Y'') \ar[rr]  && F'(P''[1]) }}
\end{equation}
Some care is required in assessing the commutativity of this diagram.  There are four morphisms of distinguished triangles in this diagram: the ``top'' and ``bottom,'' each obtained by applying a functor to~\eqref{eqn:dt-orlov-pchoice}, and the ``front'' and ``back,'' each of which is an instance of~\eqref{eqn:dt-orlov-theta}.  We also know by induction that $F'(g) \circ \theta_{Y} = \theta_{Y''} \circ F(g)$ and $F'(f[1]) \circ \theta_{P'[1]} = \theta_{P[1]} \circ F(f[1])$.  We thus obtain two morphisms of triangles from the ``top back'' to the ``bottom front,'' which we write down together as follows:
\[
\xymatrix{
F(P) \ar[r]\ar[d] & F(X) \ar[r] \ar@/_1ex/[d]_{\theta_X} \ar@/^1ex/[d]^{\theta''_X} &
  F(Y) \ar[r]\ar[d] & F(P[1]) \ar[d] \\
F'(P'') \ar[r] & F'(X) \ar[r] & F'(Y'') \ar[r] & F'(P''[1]) }
\]
We can now invoke Lemma~\ref{lem:dt-orlov}\eqref{it:dt-orlov-unique} or Lemma~\ref{lem:dt-orlov-inf} again to deduce that $\theta_X = \theta''_X$.  But since $P' \to X$ also factors through $P'' \to X$, the same argument shows that $\theta'_X = \theta''_X$ as well, so $\theta_X = \theta'_X$, as desired.

{\it Step 3}.  Let $s:X \to X'$ be a morphism in $\Kb(\scA)_{\Sigma}$.  Choose distinguished triangles 
\[
P \to X \to Y \to P[1],\qquad
P' \to X' \to Y' \to P'[1]
\]
as in Lemma~\ref{lem:dt-orlov}\eqref{it:dt-orlov-exist}.  We may assume without loss of generality that the composition $P \to X \to X'$ factors through $P' \to X'$: if not, simply replace $P'$ by $P \oplus P'$. (The cone of $P \oplus P' \to X'$ is still in $\Kb(\scA)_{\Sigma'}$ by Lemma~\ref{lem:dt-orlov}\eqref{it:dt-orlov-cone}, so the new triangle is still of the required form.)  We then have a morphism of distinguished triangles
\[
\xymatrix{
P \ar[r]\ar[d] & X \ar[r]\ar[d]_{s} & Y \ar[r]\ar[d] & P[1] \ar[d] \\
P' \ar[r] & X' \ar[r] & Y' \ar[r] & P'[1] }
\]
Applying $F$ and $F'$, we obtain a large diagram analogous to~\eqref{eqn:dt-prism}:
\[
\vcenter{\tiny\xymatrix@=15pt{
F(P) \ar[rr] \ar[dr] \ar[ddd]^(.7){\theta_{P}} && 
  F(X) \ar[rr] \ar[dr]^{F(s)} \ar'[d][ddd]^(.55555){\theta_X} &&
  F(Y) \ar[rr] \ar[dr] \ar'[d][ddd]^(.55555){\theta_Y}  && 
  F(P[1]) \ar[dr] \ar'[d][ddd]^(.55555){\theta_{P[1]}} \\
& F(P') \ar[rr] \ar[ddd]^(.3){\theta_{P'}} && 
  F(X') \ar[rr] \ar[ddd]^(.3){\theta_{X'}} &&
  F(Y') \ar[rr] \ar[ddd]^(.3){\theta_{Y'}} && 
  F(P'[1]) \ar[ddd]^(.3){\theta_{P''[1]}} \\ \\
F'(P) \ar'[r][rr] \ar[dr] && 
  F'(X) \ar'[r][rr] \ar[dr]^{F'(s)} &&
  F'(Y) \ar'[r][rr] \ar[dr] && 
  F'(P[1]) \ar[dr] \\
& F'(P'') \ar[rr]  && F'(X') \ar[rr]  &&
  F'(Y'') \ar[rr]  && F'(P''[1]) }}
\]
As with~\eqref{eqn:dt-prism}, by studying the parts of this diagram known to be commutative, we obtain two morphisms of distinguished triangles
\[
\xymatrix{
F(P) \ar[r]\ar[d] & F(X) \ar[r] \ar@/_1ex/[d]_{F'(s)\circ\theta_X} \ar@/^1ex/[d]^{\theta_{X'}\circ F(s)} &
  F(Y) \ar[r]\ar[d] & F(P[1]) \ar[d] \\
F'(P') \ar[r] & F'(X') \ar[r] & F'(Y') \ar[r] & F'(P'[1]) }
\]
By Lemma~\ref{lem:dt-orlov}\eqref{it:dt-orlov-unique} or~\ref{lem:dt-orlov-inf}, we must have $F'(s) \circ \theta_X = \theta_{X'} \circ F(s)$, as desired.

If $\theta^\circ$ is an isomorphism, we noted at the end of Step~1 that $\theta_X$ is an isomorphism for all $X$, so $\theta$ is an isomorphism of functors.
\end{proof}

\begin{thm}\label{thm:orlov-genuine}
Let $\scA$ and $\scB$ be two Orlov categories, and let $F: \pse\Kb(\scA) \to \pse\Kb(\scB)$ be a pseudotriangulated functor.  If $F(\scA) \subset \scB$, and if the induced functor $F|_{\scA}: \scA \to \scB$ is homogeneous, then $F$ is genuine.
\end{thm}
\begin{proof}
Let $\tilde F: \Kb(\scA) \to \Kb(\scB)$ be the functor induced by $F$ as in Definition~\ref{defn:infext-induc}.  According to Remark~\ref{rmk:orlov}, we have $\tilde F \cong \Kb(F|_{\scA})$.  On the other hand, by Remark~\ref{rmk:orlov-inf}, we have $F \circ \incl \cong \incl \circ \Kb(F|_{\scA})$.  Therefore, $\incl \circ \tilde F \cong F \circ \incl$, as desired.
\end{proof}

\subsection{Bifunctors}
\label{ss:bifunc}

The results of Section~\ref{ss:orlov-main} can be generalized to bifunctors.  We briefly indicate how to carry out this generalization.  Let $\scA$, $\scA'$, and $\scB$ be Orlov categories.  An additive bifunctor $F: \scA \times \scA' \to \scB$ is said to be \emph{bihomogeneous} if for any two homogeneous objects $X \in \scA$, $X' \in \scA'$, we have that $F(X,X')$ is homogeneous of degree $\deg X + \deg X'$.  In the setting of infinitesimal extensions, an additive bifunctor $F: \pse\Kb(\scA) \times \pse\Kb(\scA') \to \pse\Kb(\scB)$ is said to be \emph{pseudotriangulated} if the functors $F(X,\cdot)$ and $F(\cdot, X')$ are pseudotriangulated for any fixed objects $X$ and $X'$.  Finally, $F$ is said to be \emph{genuine} if there is a triangulated bifunctor $\tilde F: \Kb(\scA) \times \Kb(\scA') \to \Kb(\scB)$ such that $\incl \circ \tilde F \cong F \circ (\incl \times \incl)$.

\begin{thm}\label{thm:biorlov}
Let $\scA$, $\scA'$, and $\scB$ be Orlov categories, and suppose we have two triangulated bifunctors $F, F': \Kb(\scA) \times \Kb(\scA') \to \Kb(\scB)$, or two pseudotriangulated bifunctors $F, F': \Kb(\scA) \times \Kb(\scA') \to \pse\Kb(\scB)$.  Assume that $F(\scA \times \scA') \subset \scB$ and $F'(\scA \times \scA') \subset \scB$, and that the induced functors $F|_{\scA \times \scA'}, F'|_{\scA \times \scA'}: \scA \times \scA' \to \scB$ are bihomogeneous.  Any morphism of additive bifunctors
\[
\theta^\circ: F|_{\scA \times \scA'} \to F'|_{\scA \times \scA'}
\]
can be extended to a morphism $\theta: F \to F'$ of pseudotriangulated bifunctors in such a way that if $\theta^\circ$ is an isomorphism, then $\theta$ is as well.
\end{thm}
\begin{proof}[Sketch of proof]
We must construct a morphism $\theta_{X,X'}: F(X,X') \to F'(X,X')$ for each object $(X,X') \in \Kb(\scA) \times \Kb(\scA')$.  This construction proceeds by induction on the size of the supports of $X$ and $X'$.  Both supports may be replaced by ``paragraphs.''  If both supports are contained in a single line, then $\theta_{X,X'}$ is easily defined in terms of $\theta^\circ$.  Otherwise, suppose $\supp X$ has at least two lines.  Form a distinguished triangle like~\eqref{eqn:dt-orlov-choice} using Lemma~\ref{lem:dt-orlov}\eqref{it:dt-orlov-exist}, and apply the functors $F(\cdot, X')$ and $F'(\cdot, X')$ to it.  The construction of $\theta_{X,X'}$ and the proof that it is independent of choices involving $X$ are as above.  A similar construction can be carried out if $\supp X'$ instead has at least two lines.

However, when neither $\supp X$ nor $\supp X'$ is contained in a single line, there is a further well-definedness issue: we must check that $\theta_{X,X'}$ does not depend on whether we carried out the above construction using $X$ or using $X'$.  If we apply $F$ to distinguished triangles 
\[
P \to X \to Y \to, \qquad P' \to X' \to Y' \to
\]
coming from Lemma~\ref{lem:dt-orlov}\eqref{it:dt-orlov-exist} in both variables simultaneously, we get a large diagram in $\Kb(\scB)$ or $\pse\Kb(\scB)$ involving nine objects arranged into three horizontal distinguished triangles and three vertical distinguished triangles, as in~\cite[Proposition~1.1.11]{bbd}.  Let us call such a diagram a \emph{distinguished $9$-tuple}.  Applying $F'$ gives us another distinguished $9$-tuple.  To proceed, we must use a ``$27$-lemma,'' stating that a commutative cube (involving the known morphisms $\theta_{Y,Y'}$, $\theta_{P[1],Y'}$, $\theta_{Y,P'[1]}$, and $\theta_{P[1],P'[1]}$) can be extended to a morphism of distinguished $9$-tuples.  That morphism contains a morphism $F(X,X') \to F'(X,X')$, which must coincide with both versions of $\theta_{X,X'}$ by Lemma~\ref{lem:dt-orlov}\eqref{it:dt-orlov-unique} or~\ref{lem:dt-orlov-inf}.
\end{proof}

The proof of Theorem~\ref{thm:orlov-genuine} applies in this setting as well.

\begin{thm}\label{thm:biorlov-genuine}
Let $\scA$, $\scA'$, and $\scB$ be Orlov categories, and let $F: \pse\Kb(\scA) \times \pse\Kb(\scA') \to \pse\Kb(\scB)$ be a pseudotriangulated bifunctor.  If $F(\scA \times \scA') \subset \scB$, and if the induced functor $F|_{\scA \times \scA'}: \scA \times \scA' \to \scB$ is bihomogeneous, then $F$ is genuine. \qed
\end{thm}

\section{Koszul duality from Orlov categories}
\label{sect:kosorlov}

In this section, we will show that there is a very close relationship between Koszul categories and a certain class of Orlov categories, called \emph{Koszulescent} Orlov categories.  Specifically, we will prove in Section~\ref{ss:koszulescent} that there is a one-to-one correspondence 
\begin{equation}\label{eqn:koszul-orlov}
\xymatrix@C=60pt{
{\left\{
\begin{array}{c}
\text{equivalence classes} \\
\text{of split Koszul} \\
\text{abelian categories}
\end{array}
\right\}}
\ar@{<->}[r]^{\sim} &
{\left\{
\begin{array}{c}
\text{equivalence classes} \\
\text{of Koszulescent} \\
\text{Orlov categories}
\end{array}
\right\}.} }
\end{equation}
In one direction, the map is easy to describe: given a split Koszul category $\scM$, it turns out that the category of all pure objects of weight $0$ in $\Db(\scM)$ is a Koszulescent Orlov category.  (The description of the map in the other direction is given in Section~\ref{ss:tstruc}.)  This correspondence may be seen as a generalization of Koszul duality (Theorem~\ref{thm:koszul-duality}); indeed, as promised in Section~\ref{ss:koszul-cat}, we give in Section~\ref{ss:koszul-duality} a new proof of Theorem~\ref{thm:koszul-duality} based on the correspondence~\eqref{eqn:koszul-orlov}.

\subsection{A $t$-structure on $\Kb(\scA)$}
\label{ss:tstruc}

Let $\scA$ be an Orlov category.  Consider the following two subsets of $\Z \times \Z$:
\[
\lhd = \{(i,j) \mid i \le -j\},
\qquad
\rhd = \{(i,j) \mid i \ge -j\}.
\]
As in the previous section, we associate to these subsets certain full subcategories $\Kb(\scA)_{\lhd}$, $\Kb(\scA)_{\rhd}$ of $\Kb(\scA)$.  We begin with some lemmas about these categories.

\begin{lem}\label{lem:kosorlov-orth}
If $X \in \Kb(\scA)_{\lhd}$ and $Y[1] \in \Kb(\scA)_{\rhd}$, then $\Hom(X,Y) = 0$.
\end{lem}
\begin{proof}
We may assume that the underlying chain complexes of $X$ and $Y$ are such that  $\supp X \subset \lhd$ and $\supp Y[1] \subset \rhd$.  Then, for each $i \in \Z$, the homogeneous summands of $X^i$ have degree${}\le -i$, while those of $Y^i = (Y[1])^{i-1}$ have degree${}\ge -i+1$, so there are no nonzero morphisms $X^i \to Y^i$.  It follows that $\Hom(X,Y) = 0$ in $\Kb(\scA)$.
\end{proof}

\begin{lem}\label{lem:kosorlov-cone}
Let $S \in \Ind(\scA)$. 
\begin{enumerate}
\item If $X \in \Kb(\scA)_{\rhd}$, the cone of any nonzero morphism $S[\deg S] \to X$ lies in $\Kb(\scA)_{\rhd}$.
\item If $X \in \Kb(\scA)_{\lhd}$, the cocone of any nonzero morphism $X \to S[\deg S]$ lies in $\Kb(\scA)_{\lhd}$.
\end{enumerate}
\end{lem}
\begin{proof}
We will prove only the first assertion; the second one is similar.  Let $p = \deg S$.  Assume that the chain complex $X = (X^\bullet, d_X)$ is such that $\supp X \subset \rhd$. Write $X^{-p}$ as a direct sum $X^{-p} = A \oplus B$, where $A$ is homogeneous of degree $p$, and $B$ is a direct sum of homogeneous summands whose degrees are${}>p$.  Since $\Hom(S,B) = 0$, any nonzero chain map $f: S[p] \to X$ must have the form
\[
f^k = 
\begin{cases}
0 & \text{if $k \ne -p$,} \\
[\begin{smallmatrix} a \\ 0 \end{smallmatrix}]: S \to A \oplus B & \text{if $k = -p$,}
\end{cases}
\]
where $a: S \to A$ is some nonzero map.  Let $q: A \to A'$ be the homogeneous cokernel of $a$, as in Lemma~\ref{lem:homog-coker}.  That lemma also tells us that $q$ may be regarded as a projection onto a direct summand of $A$.  Here, we claim that $a$ is in fact the inclusion map of a complementary summand.  We can certainly write $A \cong C \oplus A'$ for some homogeneous object $C$.  Since $q \circ a = 0$, we have $a = [\begin{smallmatrix} c\\ 0\end{smallmatrix}]$ for some $c: S \to C$.  But if $c$ is not an isomorphism, then it has its own nonzero cokernel, contradicting the universal property of $q$.  Thus, $a$ identifies $S$ with the summand $C$.

Let $Y^{-p}$ denote the object $A' \oplus B$.  We henceforth identify $X^{-p}$ with the direct sum $S \oplus Y^{-p}$, and $f^{-p}: S \to S \oplus Y^{-p}$ with the inclusion map of the first summand.  Let $r: S \oplus Y^{-p} \to Y^{-p}$ be the projection to the second summand.  We may write the differential $d_X^{-p}: S \oplus Y^{-p} \to X^{-p+1}$ as a matrix
\[
d_X^{-p} = \big[\begin{matrix} 0 & s\end{matrix}\big].
\]
Here, the first entry must be $0$ because $d_X^{-p} \circ f^{-p} = 0$.  Similarly, we may write $d_X^{-p-1}$ as a matrix
\[
d_X^{-p-1} = \begin{bmatrix} u \\ v \end{bmatrix}.
\]
Since $d_X^{-p} \circ d_X^{-p-1} = 0$, we see that $s \circ v = 0$.  We also have $d_X^{-p+1} \circ s = 0$ and $v \circ d_X^{-p-2} = 0$.

We now regard the object $Y^{-p}$ as a term of the chain complex $Y = (Y^\bullet, d_Y)$ given by
\[
Y^k = 
\begin{cases}
X^k & \text{if $k \ne -p$,} \\
A' \oplus B & \text{if $k = -p$,}
\end{cases}
\qquad
d_Y^k =
\begin{cases}
d_X^k & \text{if $k \ne -p, -p-1$,} \\
s & \text{if $k = -p$,} \\
v & \text{if $k = -p-1$.}
\end{cases}
\]
The observations in the preceding paragraph show that $d_Y^k \circ d_Y^{k-1} = 0$ for all $k$, so this is a well-defined chain complex.  Next, we define a morphism $g: Y[-1] \to S[p]$ by
\[
g^k =
\begin{cases}
0 & \text{if $k \ne -p$,} \\
u: Y^{-p-1} \to S & \text{if $k = -p.$}
\end{cases}
\]
Again, this is a chain map since $u \circ d_X^{-p-2} = 0$.  

It is now easy to see that the cone of $u: Y[-1] \to S[p]$ is none other than $X$, and that the second map in the triangle $Y[-1] \to S[p] \to X \to$ is $f$.  Thus, the cone of $f$ is isomorphic to $Y$, which lies in $\Kb(\scA)_{\rhd}$ by construction.
\end{proof}

\begin{lem}\label{lem:kosorlov-trunc}
For any $X \in \Kb(\scA)$, there is a distinguished triangle $A \to X \to B \to$ with $A \in \Kb(\scA)_{\lhd}$ and $B[1] \in \Kb(\scA)_{\rhd}$.
\end{lem}
\begin{proof}
We will make use of the ``$*$'' operation for objects in a triangulated category from~\cite[\S1.3.9]{bbd}.  In this language, we must show that
\begin{equation}\label{eqn:kosorlov-trunc}
X \in \Kb(\scA)_{\lhd} * (\Kb(\scA)_{\rhd}[-1]).
\end{equation}
We proceed by induction on the size of the support of $X$.  If $\supp X$ is a singleton, then $X \cong A[n]$ for some homogeneous object $A \in \scA$ and some $n \in \Z$.  If $n \ge \deg A$, then $X \in \Kb(\scA)_{\lhd}$; otherwise, $X[1] \in \Kb(\scA)_{\rhd}$.  In either case,~\eqref{eqn:kosorlov-trunc} holds trivially.

Next, consider the general case.  Let $\Sigma = \supp X$, and let $(i,j)$ be the largest element of $\Sigma$.  Lemma~\ref{lem:dt-orlov}\eqref{it:dt-orlov-exist} says that there are objects $P$ and $Y$ such that
\[
X \in \{P\} * \{Y\},
\]
with $\supp P = \{(i,j)\}$ and $\supp Y \subset \Sigma' = \Sigma \smallsetminus \{(i,j)\}$.  By induction, we may assume that there exist objects $A' \in \Kb(\scA)_{\lhd}$ and $B' \in \Kb(\scA)_{\rhd}[-1]$ such that $Y \in \{A'\} * \{B'\}$, so that
\[
X \in \{P\} * \{A'\} * \{B'\}.
\]
If $i \le -j$, then $P \in \Kb(\scA)_{\lhd}$, so $\{P\} * \{A'\} \subset \Kb(\scA)_{\lhd}$, and then~\eqref{eqn:kosorlov-trunc} follows.  On the other hand, if $i > -j$, we proceed by induction on the number of indecomposable summands in the homogeneous object $P$.  Write $P = S[-i] \oplus P'$, where $S \in \Ind(\scA)$ is an indecomposable object of degree $j$, and where $P'$ contains fewer indecomposable summands (possibly zero).  We then have
\[
X \in \{S[-i]\} * \{P'\} * \{A'\} * \{B'\}.
\]
By induction, we have $\{P'\} * \{A'\} * \{B'\} \subset \Kb(\scA)_{\lhd} * (\Kb(\scA)_{\rhd}[-1])$.  Thus, there exist objects $A'' \in \Kb(\scA)_{\lhd}$ and $B'' \in \Kb(\scA)_{\rhd}[-1]$ so that
\[
X \in \{S[-i]\} * \{A''\} * \{B''\}.
\]
Consider a distinguished triangle 
\begin{equation}\label{eqn:kosorlov-trunc2}
S[-i] \to Z \to A'' \to S[-i+1].
\end{equation}
If $i-1 > -j$, then $S[-i+1] \in \Kb(\scA)_{\rhd}[-1]$, so we know by Lemma~\ref{lem:kosorlov-orth} that $\Hom(A'', S[-i+1]) = 0$.  Therefore, every such triangle splits, and we have $\{S[-i]\} * \{A''\} \subset \{A''\} * \{S[-i]\}$.  It follows that
\[
X \in \{A''\} * ( \{S[-i]\} * \{B''\}) \subset \Kb(\scA)_{\lhd} * (\Kb(\scA)_{\rhd}[-1]),
\]
as desired.  On the other hand, if $i-1 \not> -j$, we must have $i - 1 = -j$ (recall that $i > -j$).  If the triangle~\eqref{eqn:kosorlov-trunc2} splits, then the preceeding argument still applies.  But if~\eqref{eqn:kosorlov-trunc2} does not split, i.e., if the map $A'' \to S[-i+1]$ is nonzero, then we are in the setting of Lemma~\ref{lem:kosorlov-cone}, which tells us that $Z \in \Kb(\scA)_{\lhd}$.  Since $X \in \{Z\} * \{B''\}$, we see that~\eqref{eqn:kosorlov-trunc} holds in this case as well.
\end{proof}

\begin{prop}\label{prop:kosorlov-tstruc}
For any Orlov category $\scA$, the pair $(\Kb(\scA)_{\lhd}, \Kb(\scA)_{\rhd})$ is a bounded $t$-structure on $\Kb(\scA)$.  The heart
\[
\Kos(\scA) = \Kb(\scA)_{\lhd} \cap \Kb(\scA)_{\rhd}
\]
is a split finite-length abelian category, and the simple objects in $\Kos(\scA)$ are those isomorphic to objects in the set
\[
\Irr(\Kos(\scA)) = \{ S[\deg S] \mid S \in \Ind(\scA) \}.
\]
Moreover, $\Kos(\scA)$ has the structure of a mixed category, with weight function $\wt: \Irr(\Kos(\scA)) \to \Z$ given by $\wt(S[\deg S]) = \deg S$.
\end{prop}
\begin{proof}
It is clear that $\Kb(\scA)_{\lhd}[1] \subset \Kb(\scA)_{\lhd}$ and $\Kb(\scA)_{\rhd}[1] \supset \Kb(\scA)_{\rhd}$.  The other axioms for a $t$-structure have been checked in Lemmas~\ref{lem:kosorlov-orth} and~\ref{lem:kosorlov-trunc}, so the pair $(\Kb(\scA)_{\lhd}, \Kb(\scA)_{\rhd})$ does indeed constitute a $t$-structure.  

Since the support of any object $X$ is finite, it is clear that there exist integers $n$ and $m$ such that $X[n] \in \Kb(\scA)_{\lhd}$ and $X[m] \in \Kb(\scA)_{\rhd}$.  In other words, the $t$-structure is bounded.

For brevity, let us write $\scC = \Kos(\scA)$.  Let $S \in \Ind(\scA)$, and consider the object $E = S[\deg S] \in \scC$.  Let $X$ be any other object of $\scC$.  If $f: E \to X$ is a nonzero morphism, Lemma~\ref{lem:kosorlov-cone} tells us that the cone of $f$ lies in $\Kb(\scA)_{\rhd}$, which means that the kernel of $f$ must be $0$.  This means that $E$ contains no nontrivial subobject in $\cC$.  In other words, $E$ is simple.  

Let us call a simple object of $\scC$ \emph{good} if it is isomorphic to $S[\deg S]$ for some $S \in \Ind(\scA)$.  More generally, an object of $\scC$ is said to be \emph{good} if it has a composition series whose composition factors are good simple objects.  We will now show that every object of $\scC$ is good.  Given an object $X = (X^\bullet, d_X) \in \scC$, let $\Sigma = \supp X$.  Assume that $\Sigma \subset \lhd$.  We proceed by induction on the size of $\Sigma$.  

The base case is that in which $\Sigma$ is a singleton.  Since $X \in \scC$, we must have $X \cong A[n]$, where $A \in \scA$ is homogeneous of degree $n$.  Such an object is evidently a direct sum of good simple objects.

Otherwise, let $(i,j)$ be the largest element of $\Sigma$, and form the distinguished triangle
\[
P \ovto{f} X \to Y \to
\]
as in Lemma~\ref{lem:dt-orlov}\eqref{it:dt-orlov-exist}.  Since $(i,j) \in \lhd$, we have $i \le -j$.  If $i < -j$, then $P \in \Kb(\scA)_{\lhd}[1]$, so $\Hom(P,X) = 0$ by Lemma~\ref{lem:kosorlov-orth}.  It follows that $Y \cong X \oplus P[1]$.  Recall that $Y$ has strictly smaller support than $X$.  By Lemma~\ref{lem:orlov-summand}, $X$ is isomorphic to a chain complex whose support is contained in that of $Y$, so $X$ is already known to be good by induction.

Suppose, on the other hand, that $i = -j$.  Then $P \in \scC$.  Indeed, $P$ is clearly a semisimple object whose direct summands are good simple objects.  Moreover, $f: P \to X$ is a morphism in $\scC$, so we may write $P \cong \ker f \oplus \im f$.  The map $f$ is then the direct sum of an injective map $\im f \to X$ and the zero map $\ker f \to 0$, so its cone is 
\[
Y \cong \cok f \oplus (\ker f)[1].
\]
Using Lemma~\ref{lem:orlov-summand} again, we see that $\cok f$ is an object of $\scC$ with strictly smaller support than $X$, so it is good.  From the short exact sequence
\[
0 \to \im f \to X \to \cok f \to 0,
\]
we see that $X$ is good, as desired.  In other words, we have just shown that every object of $\scC$ has finite length, and that every simple object is isomorphic to some $S[\deg S]$ with $S \in \Ind(\scA)$.  It follows from Definition~\ref{defn:orlov}\eqref{it:orlov-endind} that $\scC$ is split.

Finally, to show that $\scC$ is a mixed category with the weight function given above, we must check that $\Ext^1(S[\deg S], T[\deg T]) = 0$ for $S,T \in \Ind(\scA)$ if $\deg T \ge \deg S$.  By~\cite[Remarque~3.1.17(2)]{bbd}, we have
\[
\Ext^1_{\scC}(S[\deg S], T[\deg T]) \cong \Hom_{\Kb(\scA)}(S[\deg S], T[\deg T+1]),
\]
It is clear that $\Hom(S[\deg S], T[\deg T+1]) = 0$ if $\deg T \ge \deg S$, as desired.
\end{proof}

\begin{cor}\label{cor:kosorlov-supp}
We have $\Kb(\scA)_{\lhd \cap \rhd} = \Kos(\scA)$.  Moreover, for a chain complex $X = (X^\bullet, d_X)$ with $\supp X \subset \lhd \cap \rhd$, the associated graded of the weight filtration on $X$ is given by $\gr^W_k X = X^{-k}[k]$.
\end{cor}
\begin{proof}
It is clear that $\Kb(\scA)_{\lhd \cap \rhd} \subset \Kos(\scA)$, and that the simple objects $S[\deg S]$ (with $S \in \Ind(\scA)$) of $\Kos(\scA)$ lie in $\Kb(\scA)_{\lhd \cap \rhd}$.  Recall that for any subset $\Sigma \subset \Z \times \Z$, the category $\Kb(\scA)_\Sigma$ is stable under extensions.  The smallest strictly full subcategory of $\Kb(\scA)$ containing all the $S[\deg S]$ and stable under extensions is none other than $\Kos(\scA)$, because every object of $\Kos(\scA)$ has finite length.  It follows that $\Kos(\scA) \subset \Kb(\scA)_{\lhd \cap \rhd}$.  The second claim is obvious.
\end{proof}

The calculation at the end of the proof of Proposition~\ref{prop:kosorlov-tstruc} actually shows the stronger statement that $\Ext_{\scC}^1(S[\deg S], T[\deg T]) = 0$ unless $\deg T = \deg S - 1$.  Indeed, the same reasoning gives us the following more general statement: for $S, T \in \Ind(\scA)$, the simple objects $S[\deg S], T[\deg T] \in \Irr(\Kos(\scA))$ have the property that
\begin{equation}\label{eqn:kosorlov-homv}
\Hom^i_{\Kb(\scA)}( S[\deg S], T[\deg T] ) = 0
\qquad\text{if $\deg T \ne \deg S - i$.}
\end{equation}
This is a stronger version of the condition~\eqref{eqn:mixed-hom-vanish}.  An easy induction argument yields a strengthened version of Lemma~\ref{lem:mixed-tri-basic}\eqref{it:mixed-tri-basic-homv}, as follows.

\begin{cor}\label{cor:kosorlov-fakekoszul}
For an Orlov category $\scA$, the mixed structure on $\Kos(\scA)$ makes $\Kb(\scA)$ into a mixed triangulated category.  Moreover, if $X, Y \in \Kb(\scA)$ are objects such that $X$ has weights${}\le w$ and $Y$ has weights${}>w$, then $\Hom(X,Y) = \Hom(Y,X) = 0$. \qed
\end{cor}

\subsection{Koszulescent Orlov categories}
\label{ss:koszulescent}

The vanishing property~\eqref{eqn:kosorlov-homv} and Corollary~\ref{cor:kosorlov-fakekoszul} closely resemble properties of Koszul categories.  To make this resemblance into a precise statement, we must impose the following additional condition on an Orlov category.

\begin{defn}\label{defn:koszulescent}
An Orlov category $\scA$ is said to be \emph{Koszulescent} if the realization functor
\[
\real: \Db(\Kos(\scA)) \to \Kb(\scA)
\]
is an equivalence of categories.
\end{defn}

If $\scA$ is Koszulescent, then~\eqref{eqn:kosorlov-homv} is equivalent to the defining condition~\eqref{eqn:koszul-ext-vanish} for a Koszul category, and Corollary~\ref{cor:kosorlov-fakekoszul} is equivalent to Lemma~\ref{lem:koszul-tri-basic}.  In particular, we have the following observation.

\begin{prop}\label{prop:kosorlov-koszul}
If $\scA$ is a Koszulescent Orlov category, then $\Kos(\scA)$ is a split Koszul abelian category. \qed
\end{prop}

The following result is a sort of converse to the preceding one. 

\begin{prop}\label{prop:kosorlov-orlov}
Let $\scC$ be a split Koszul abelian category, and consider the additive category
\[
\Orl(\scC) = \{\text{pure objects of weight $0$ in $\Db(\scC)$} \}.
\]
The isomorphism classes of indecomposable objects in $\Orl(\scC)$ are given by
\[
\Ind(\Orl(\scC)) = \{ L[-\wt L] \mid L \in \Irr(\scC) \}.
\]
If we define $\deg: \Ind(\Orl(\scC)) \to \Z$ by $\deg (L[-\wt L]) = \wt L$, then $\Orl(\scC)$ becomes a Koszulescent Orlov category.  Moreover, there is a natural equivalence of abelian categories $\scC \simto \Kos(\Orl(\scC))$.
\end{prop}
\begin{proof}
For brevity, let us write $\scA = \Orl(\scC)$.  The description of indecomposable objects in $\scA$ follows from Lemma~\ref{lem:mixed-tri-basic}\eqref{it:mixed-tri-basic-pure}.  For $L, L' \in \Irr(\scC)$, it is obvious that $\Hom(L[-\wt L], L'[-\wt L']) = 0$ in $\Db(\scC)$ if $\wt L < \wt L'$, or if $\wt L = \wt L'$ but $L \not\cong L'$.  Thus, $\scA$ is an Orlov category.  

We will now construct a functor $\tilde Q: \scC \to \Kb(\scA)$.  For an object $X \in \scC$, we use the weight filtration $W_\bullet X$ to construct a short exact sequence
\[
0 \to \gr^W_{k-1} X \to W_k X/ W_{k-2} X \to \gr^W_k X \to 0.
\]
This sequence determines an element $\partial_k = \partial_{X,k} \in \Ext^1(\gr^W_k X, \gr^W_{k-1} X)$.  Because all morphisms in $\scC$ are strictly compatible with $W_\bullet M$, any morphism $f: X \to Y$ gives rise to a morphism of short exact sequences of the above form, and so, in $\Db(\scC)$, a commutative diagram
\begin{equation}\label{eqn:orlovkos-chainmap}
\vcenter{\xymatrix{
\gr^W_k X \ar[d]_{\gr^W_k(f)} \ar[r]^-{\partial_{X,k}}
  & (\gr^W_{k-1} X)[1] \ar[d]^{\gr^W_{k-1}(f)[1]} \\
\gr^W_k Y \ar[r]_-{\partial_{Y,k}} & (\gr^W_{k-1} Y)[1] }}
\end{equation}
Next, consider the Yoneda product $\partial_{k-1} \cdot \partial_{k} \in \Ext^2(\gr^W_k X, \gr^W_{k-2} X)$, which corresponds to the exact sequence
\[
0 \to \gr^W_{k-2} X \to W_{k-1} X/ W_{k-3} X \to W_k X/W_{k-2} X \to \gr^W_k X \to 0.
\]
This sequence arises from a filtration of the object $W_k X/W_{k-3} X$, so it follows that
\begin{equation}\label{eqn:orlovkos-diff}
\partial_{k-1} \cdot \partial_{k} = 0.
\end{equation}
Note that for any $k \in \Z$, the object $(\gr^W_{-k} X)[k] \in \Db(\scC)$ is pure of weight $0$, so it is an object in $\scA$.  Moreover,
\[
\partial_{-k}[k] \in \Hom((\gr^W_{-k} X)[k], (\gr^W_{-k-1} X)[k+1])
\]
is a morphism in $\scA$.  We now define $\tilde Q: \scC \to \Kb(\scA)$ by
\[
\tilde Q(X) = (\tilde Q(X)^\bullet, d^\bullet_{\tilde Q(X)})
\qquad\text{where}\qquad
\tilde Q(X)^k = (\gr^W_{-k} X)[k],
\quad
d^k_{\tilde Q(X)} = \partial_{X,-k}.
\]
That this is a chain complex follows from~\eqref{eqn:orlovkos-diff}. Moreover, from~\eqref{eqn:orlovkos-chainmap}, we see that any morphism $f: X \to Y$ in $\scC$ induces a morphism of chain complexes $\tilde Q(X) \to \tilde Q(Y)$, so this is indeed a functor.

For brevity, let us write $\scC' = \Kos(\scA) = \Kb(\scA)_{\lhd} \cap \Kb(\scA)_{\rhd}$.  Since the $k$th term of the chain complex $\tilde Q(X)$ is a homogeneous object of $\scA$ of degree $-k$, we see that $\tilde Q$ actually takes values in $\scC'$.  Let $Q_0: \scC \to \scC'$ be the functor obtained from $\tilde Q$ by restricting its codomain.

Note that applying the exact functor $\gr^W_k$ to a short exact sequence in $\scC$ yields a (necessarily) split short exact sequence of pure objects.  Therefore, applying $\tilde Q$ to a short exact sequence in $\scC$ yields a sequence of chain complexes satisfying the hypotheses of Lemma~\ref{lem:chain-ses}.  Invoking that lemma, we find that $\tilde Q$ takes short exact sequences in $\scC$ to distinguished triangles in $\Kb(\scA)$.  It follows that $Q_0: \scC \to \scC'$ is an exact functor of abelian categories, so it gives rise to a derived functor $Q_0': \Db(\scC) \to \Db(\scC')$.  Let us compose this with the realization functor $\real: \Db(\scC') \to \Kb(\scA)$ and define
\[
Q = \real {}\circ Q_0': \Db(\scC) \to \Kb(\scA).
\]
It is easy to see that this functor induces isomorphisms
\[
\Hom_{\Db(\scC)}(L[-\wt L], L'[-\wt L'+k]) \simto
\Hom_{\Kb(\scA)}(L[-\wt L], (L'[-\wt L'])[k])
\]
for all $L, L' \in \Irr(\scC)$ and all $k \in \Z$.  (Indeed, both $\Hom$-groups vanish unless $k = 0$.)  Since objects of the form $L[-\wt L]$ generate both $\Db(\scC)$ and $\Kb(\scA)$, it follows that $Q: \Db(\scC) \to \Kb(\scA)$ is an equivalence of categories.  

It follows that $\real$ is full and essentially surjective.  In particular, for all objects $X, Y \in \scC'$ and all $k \ge 0$, the induced map
\begin{equation}\label{eqn:kosorlov-d-func}
\real : \Ext^k_{\scC'}(X,Y) \to \Hom_{\Kb(\scA)}(X,Y[k])
\end{equation}
is surjective.  On the other hand, according to~\cite[Remarque~3.1.17]{bbd}, that map is always an isomorphism for $k = 0,1$, and if it is known to be an isomorphism when $k < n$ for all $X,Y \in \scC'$, then it is injective for $k = n$.  (A similar statement appears in~\cite[Lemma~3.2.3]{bgs}.)  By induction,~\eqref{eqn:kosorlov-d-func} is always an isomorphism.  By the end of the proof of~\cite[Proposition~3.1.16]{bbd}, we conclude that $\real$ is an equivalence of categories, and that $\scA$ is Koszulescent.

Finally, we now see that $Q_0': \Db(\scC) \to \Db(\scC')$ is an equivalence of categories as well.  Since this is the derived functor of $Q_0: \scC \to \scC'$, the latter is an equivalence of abelian categories.  In other words, $\scC \cong \Kos(\scA)$.
\end{proof}

We are now ready to complete the proof of the bijective correspondence~\eqref{eqn:koszul-orlov}.

\begin{thm}\label{thm:kosorlov}
The assignments $\scA \mapsto \Kos(\scA)$ and $\scC \mapsto \Orl(\scC)$ provide bijections, inverse to one another, between equivalence classes of Koszulescent Orlov categories and equivalence classes of split Koszul abelian categories.
\end{thm}
\begin{proof}
In view of Proposition~\ref{prop:kosorlov-orlov}, it remains only to show that if $\scA$ is a Koszulescent Orlov category, then $\scA \cong \Orl(\Kos(\scA))$.  Identifying $\Db(\Kos(\scA)) \cong \Kb(\scA)$, it is straightforward to see that an object of $\Kb(\scA)$ is pure of weight $0$ with respect to the mixed structure of Proposition~\ref{prop:kosorlov-tstruc} if and only if it lies in $\scA$.
\end{proof}

\begin{rmk}
The proof that $\Orl(\scC)$ is an Orlov category does not use Koszulity in any way; it is valid for any split mixed abelian category.  Thus, using Proposition~\ref{prop:kosorlov-tstruc}, we actually have a pair of maps
\[
\xymatrix@C=60pt{
{\left\{
\begin{array}{c}
\text{equivalence classes of split} \\
\text{mixed abelian categories}
\end{array}
\right\}}
\ar@/^1ex/[r]^-{\Orl} &
{\left\{
\begin{array}{c}
\text{equivalence classes} \\
\text{of Orlov categories}
\end{array}
\right\}.
\ar@/^1ex/[l]^-{\Kos} } }
\]
However, in the absence of the Koszulity and Koszulescence conditions, these maps are neither injective nor surjective.
\end{rmk}

\subsection{Koszul duality and Koszulescent Orlov categories}
\label{ss:koszul-duality}

In the previous section, we saw how to construct a Koszulescent Orlov category starting from an arbitrary split Koszul category.  However, when a Koszul category has enough projectives, there is another, more elementary way to build a Koszulescent Orlov category from it, as explained below.

\begin{thm}\label{thm:kosorlov-duality}
Let $\scM$ be a split Koszul category with enough projectives, and in which every object has finite projective dimension, so that we have a natural equivalence of categories
\[
R: \Db(\scM) \simto \Kb(\Proj(\scM)),
\]
where $\Proj(\scM)$ is the additive category of projective objects in $\scM$.  Then $\Proj(\scM)$ is a Koszulescent Orlov category, with degree function
\[
\deg: \Ind(\Proj(\scM)) \to \Z
\qquad\text{given by}\qquad
\deg P = -\wt (P/\rad P).
\]
Moreover, the split Koszul abelian category
\[
\scM^\natural = \Kos(\Proj(\scM))
\]
has enough injectives, and every object has finite injective dimension.
\end{thm}

From the description of the category $\Kos(\Proj(\scM))$ in Proposition~\ref{prop:kosorlov-tstruc}, and in particular the description of its irreducibles and mixed structure, we see that this definition of $\scM^\natural$ coincides with that in~\eqref{eqn:koszul-duality}, and that the theorem is a restatement of Theorem~\ref{thm:koszul-duality}.

\begin{proof}
For clarity, we will not identify $\Db(\scM)$ and $\Kb(\Proj(\scM))$; we will instead explicitly use the functor $R$ to go back and forth between them.  All shifts of objects of $\scM$ should be understood to be objects of $\Db(\scM)$.  We proceed in several steps.

Because every object of $\scM$ has finite length, the Fitting lemma and its consequences hold in $\scM$.  For instance, any object $X \in \scM$ has a unique minimal subobject $\rad X$ (called its \emph{radical}) such that $X/\rad X$ is semisimple.  In the special case where $X$ is an indecomposable projective, $X/\rad X$ is simple.  These facts, and others related to the Fitting lemma, will be used freely throughout the following proof.

{\it Step 1. $\Proj(\scM)$ is an Orlov category.}  Let $P$ be an indecomposable projective in $\scM$, and let $L = P/\rad P$ denote its unique simple quotient.  Consider its weight filtration $W_\bullet P$, and let $w$ be the smallest integer such that $W_w P = P$.  Then $\gr^W_w P = P/W_{w-1}P$ is a semisimple quotient of $P$, so we must in fact have that $\gr^W_w P \cong L$ and $\rad P = W_{w-1}P$.  In particular, the simple object $L$, which has weight $w$, cannot occur as a composition factor of $\rad P$, so it follows that $\dim \Hom(P,P) = 1$.

More generally, suppose that $P'$ is another indecomposable projective, with simple quotient $L' = P'/\rad P'$.  Then $L$ cannot occur as a composition factor of $P'$ if $\wt L > \wt L'$, and if $\wt L = \wt L'$, then $L$ occurs in $P'$ if and only if $L \cong L'$, or, equivalently, if $P \cong P'$.  We have just shown that $\Hom(P,P') = 0$ if $\wt L \ge \wt L'$ and $P \not\cong P'$, so $\Proj(\scM)$ is indeed an Orlov category.

{\it Step 2.  $\scM^\natural$ has enough injectives.}  Let $L \in \Irr(\scM)$, and let $P \to L$ be its projective cover.  Then Proposition~\ref{prop:kosorlov-tstruc} tells us that $R(P[-\wt L])$ is a simple object in $\scM^\natural$, and every simple object arises in this way. Let $P' \to L'$ be another projective cover of a simple object in $\scM$.  We claim that
\begin{equation}\label{eqn:koszul-inj1}
\Hom_{\Db(\scM)}(P'[-\wt L'], L[-\wt L + k]) = 0 \qquad \text{unless $P \cong P'$ and $k = 0$.}
\end{equation}
Indeed, this $\Hom$-group obviously vanishes if $k \ne \wt L - \wt L'$.  If we take $k = \wt L - \wt L'$, then the $\Hom$-group above is isomorphic to $\Hom_{\scM}(P',L)$, which vanishes unless $P' \cong P$, in which case we necessarily have $k = 0$.

On the other hand, we also have
\begin{equation}\label{eqn:koszul-inj2}
\Hom_{\Db(\scM)}(L[-\wt L], P'[-\wt L' + k]) = 0 \qquad\text{if $k < 0$}
\end{equation}
by Lemma~\ref{lem:koszul-tri-basic}, because the object $P'[-\wt L' +k]$ has weights${}\le k$ in $\Db(\scM)$, while $L[-\wt L]$ is pure of weight $0$.

The $\Hom$-vanishing statements above have analogues in $\Kb(\Proj(\scM))$ obtained by applying $R$.  From~\eqref{eqn:koszul-inj2} and the vanishing of~\eqref{eqn:koszul-inj1} for $k < 0$, we conclude that $R(L[-\wt L]) \in \scM^\natural$.  Next, taking~\eqref{eqn:koszul-inj1} for $k = 1$, we find that
\begin{multline*}
\Ext^1_{\scM^\natural}(R(P'[-\wt L']), R(L[-\wt L])) =\\
\Hom_{\Kb(\Proj(\scM))}(R(P'[-\wt L']), R(L[-\wt L+1])) = 0.
\end{multline*}
Thus, $R(L[-\wt L])$ is an injective object of $\scM^\natural$.  It is indecomposable, and there is a nonzero map to it from the simple object $R(P[-\wt L])$, so in fact $R(L[-\wt L])$ is an injective envelope of $R(P[-\wt L])$.  

{\it Step 3. Koszulescence and finite injective dimension.}  Let $\Phi: \Db(\scM^\natural) \to \Db(\scM)$ denote the composition
\[
\Db(\scM^\natural) \xto{\real} \Kb(\Proj(\scM)) \xto{R^{-1}} \Db(\scM).
\]
It is easy to see that for any two objects $L, L' \in \Irr(\scM)$ and any $k \in \Z$, $\Phi$ induces an isomorphism
\begin{multline*}
\Hom_{\Db(\scM^\natural)}( R(L[-\wt L]), R(L'[-\wt L' + k]) ) \\
\simto
\Hom_{\Db(\scM)}(L[-\wt L], L'[-\wt L' + k]),
\end{multline*}
as both $\Hom$-groups vanish unless $k = 0$.  Now, objects of the form $L[-\wt L]$ generate $\Db(\scM)$ as a triangulated category.  If we let $\Db(\scM^\natural)' \subset \Db(\scM^\natural)$ denote the full triangulated subcategory generated by objects of the form $R(L[-\wt L])$, then the above calculation shows that $\Phi$ induces an equivalence
\[
\Phi|_{\Db(\scM^\natural)'}: \Db(\scM^\natural)' \simto \Db(\scM).
\]
To describe $\Db(\scM^\natural)'$ in another way, note that it is the full triangulated subcategory of $\Db(\scM^\natural)$ generated by injective objects of $\scM^\natural$.  Thus, it contains precisely those objects of $\Db(\scM^\natural)$ that can be represented by a bounded chain complex of injectives.  In particular, for an object $X \in \scM^\natural \subset \Db(\scM^\natural)$, we have
\[
X \in \scM^\natural \cap \Db(\scM^\natural)'
\qquad\text{if and only if}\qquad
\text{$X$ has finite injective dimension.}
\]

Let $i: \Db(\scM^\natural)' \to \Db(\scM^\natural)$ denote the inclusion functor, and let $\Psi$ denote the composition
\[
\Db(\scM^\natural)' \xto{i} \Db(\scM^\natural) \xto{\real} \Kb(\Proj(\scM)).
\]
Since $\Psi \cong R \circ (\Phi \circ i)$, it is an equivalence of categories.  Consider an object $X \in \scM^\natural \subset \Db(\scM^\natural)$.  We obviously have
\[
(\real \circ i \circ \Psi^{-1} \circ \real)(X) \cong \real(X).
\]
Because the realization functor commutes with cohomology~\cite[\S3.1.14]{bbd}, it follows that $(i \circ \Psi^{-1} \circ \real)(X) \in \scM^\natural$.  Furthermore, since $\real$ is fully faithful on $\scM^\natural$, we must have $(i \circ \Psi^{-1} \circ \real)(X) \cong X$. In particular, every object $X \in \scM^\natural$ is in the essential image of $i$.

In other words, we have just shown that every object of $\scM^\natural$ has finite injective dimension.  It follows that $\Db(\scM^\natural)' = \Db(\scM^\natural)$, so we have equivalences of categories
\[
\Phi: \Db(\scM^\natural) \to \Db(\scM)
\qquad\text{and}\qquad
\real \cong R \circ \Phi: \Db(\scM^\natural) \to \Kb(\Proj(\scM)).
\]
The latter shows that $\Proj(\scM)$ is Koszulescent.
\end{proof}

\part{Sheaf Theory}
\label{part:sheaf}

\section{Mixed and Weil Categories of Perverse Sheaves}
\label{sect:mixedweil}

As noted in the introduction, the triangulated category of constructible $\Qlb$-complexes introduced by Deligne in~\cite{del2} is too large for many purposes in representation theory.  Over the course of Part~\ref{part:sheaf}, we will study how to replace it by a smaller category, and how to define sheaf functors on the new smaller category.  In the present section, we fix notation and assumptions, and we review some facts about Deligne's category.  We will also define the ``miscible category.'' 

\subsection{Weil complexes and Weil perverse sheaves}

Let $X$ be a variety over $\Fq$.  We write $X \otimes \Fqb$ for the variety $X \times_{\Spec \Fq} \Spec \Fqb$ obtained by extension of scalars. This variety comes with a geometric Frobenius map $\Fr: X \otimes \Fqb \to X \otimes \Fqb$.  Assume that $X$ is equipped with a stratification $\cS = \{X_s\}_{s \in S}$ (for some index set $S$).  All constructible complexes should be understood to be constructible with respect to this stratification.  For any constructible complex $\cF$ on $X$, we denote by $\exs(\cF)$ its pullback to $X \otimes \Fqb$.

Let $\ddel_\cS(X)$ denote the category of ``mixed constructible complexes'' introduced by Deligne in~\cite{del2}.  (This category is often denoted $\Db_{\mathrm{m}}(X)$, cf.~\cite{bbd}.)   We reiterate that the term ``mixed'' will not be used again for this category, because that conflicts with the conventions of Section~\ref{sect:mixed} and~\cite{bgs}.  The term ``mixed'' will be reserved for a category to be introduced in Section~\ref{sect:affable}.

In this setting of $\ddel_\cS(X)$, we have available the theory of weights from~\cite{del2}.  Let us fix, once and for all, a square root of the Tate sheaf on $X$.  This allows us to form Tate twists $\cF(\frac{n}{2})$ of a constructible complex $\cF \in \ddel_\cS(X)$ for any $n \in \Z$.  For the aesthetic benefit of avoiding fractions, we henceforth adopt the notation
\[
\cF\la n\ra = \cF(-\textstyle\frac{n}{2}).
\]
We denote by $\uQlb$ the constant sheaf with value $\Qlb$ on $X$ or on any subvariety.  Next, let $j_s: X_s \to X$ denote the inclusion map of the stratum $X_s$.  The following assumption (cf.~\cite[2.2.10(c)]{bbd}) will be in force whenever we discuss constructible complexes:
\begin{equation}\label{eqn:constr-hyp}
\left\{
\begin{array}{c}
\text{For any $s,t$, the sheaf $H^i(Rj_{s*}\uQlb)|_{X_t}$ is a local system} \\
\text{with irreducible subquotients of the form $\uQlb\la n\ra$.}
\end{array}
\right.
\end{equation}
(In fact, most varieties we will encounter satisfy a much stronger condition; see Section~\ref{sect:affable}.)
Next, consider the simple perverse sheaves
\[
\ICm_s = j_{s!*}\uQlb[\dim X_s]\la-\dim X_s\ra
\qquad\text{and}\qquad
\IC_s = \exs(\ICm_s).
\]
The condition~\eqref{eqn:constr-hyp} implies that each $H^i(\ICm_s)|_{X_t}$ is a local system with irreducible subquotients of the form $\uQlb\la n\ra$.  

Let $\dsw(X)$ denote the full triangulated subcategory of $\ddel_\cS(X)$ that is generated by the objects $\ICm_s$.  Similarly, let $\ds(X)$ denote the full triangulated subcategory of bounded complexes on $X \otimes \Fqb$ generated by the $\IC_s$.  Extension of scalars gives us a functor
\[
\exs: \dsw(X) \to \ds(X).
\]
Let $\Pervsw(X)$, resp.~$\Pervs(X)$, denote the abelian category of perverse $\Qlb$-sheaves in $\dsw(X)$, resp.~$\ds(X)$.  By~\cite[5.1.2]{bbd}, $\Pervsw(X)$ may be thought of as a certain category of perverse sheaves on $X \otimes \Fqb$ equipped with a ``Weil structure,'' but the analogous statement does \emph{not} hold for $\dsw(X)$.

By~\cite[Th\'eor\`eme~5.3.5]{bbd}, every object $\cF \in \Pervsw(X)$ is equipped with a canonical \emph{weight filtration}, denoted $W_\bullet \cF$.  The subquotients $\gr^W_i \cF$ are pure, but not necessarily semisimple, cf.~\cite[Proposition~5.3.9]{bbd}.  (This failure of semisimplicity shows that $\Pervsw(X)$ is not a mixed category.)  All morphisms in $\Pervsw(X)$ are strictly compatible with the weight filtration.

For $\cF, \cG \in \dsw(X)$, let us put
\[
\uRHom(\cF,\cG) = a_* \cRHom(\cF,\cG),
\]
where $a: X \to \Spec \Fq$ is the structure morphism.  We further put
\[
\uHom^i(\cF,\cG) = H^i(Ra_*\cRHom(\cF,\cG)).
\]
Thus, $\uHom^i(\cF,\cG)$ is a $\Qlb$-sheaf over $\Spec\Fq$.  In other words, we regard it as a $\Qlb$-vector space equipped with an automorphism
\begin{equation}\label{eqn:hom-fr}
\Fr: \uHom^i(\cF,\cG) \to \uHom^i(\cF,\cG)
\end{equation}
induced by the Frobenius map.  Because $\exs$ is compatible with all the usual sheaf operations, we have
\begin{equation}\label{eqn:hom-forget}
\Hom_{\ds(X)}^i(\exs(\cF),\exs(\cG)) \simeq \exs(\uHom^i(\cF,\cG)).
\end{equation}
In other words, $\Hom^i(\exs(\cF),\exs(\cG))$ is obtained from $\uHom^i(\cF,\cG)$ by forgetting the automorphism~\eqref{eqn:hom-fr}.

The $\Hom$-groups within $\dsw(X)$ are somewhat different.  By~\cite[(5.1.2.5)]{bbd}, there is a short exact sequence of $\Qlb$-vector spaces
\begin{equation}\label{eqn:dsw-ses}
0 \to \uHom^{i-1}(\cF,\cG)_\Fr \to
\Hom^i(\cF,\cG)
\to \uHom^i(\cF,\cG)^\Fr \to 0,
\end{equation}
where $(\cdot)_\Fr$ and $(\cdot)^\Fr$ denote coinvariants and invariants of $\Fr$ (that is, the cokernel and kernel of $\Fr - \id$), respectively.  Note that the natural morphism
\[
\Hom_{\dsw(X)}(\cF,\cG) \to \Hom_{\ds(X)}(\exs(\cF), \exs(\cG))
\]
factors through the map $\Hom_{\dsw(X)}(\cF,\cG) \to \uHom(\cF,\cG)^\Fr$ of~\eqref{eqn:dsw-ses}.

\subsection{Functors on the Weil category}

The usual sheaf operations are defined on $\ddel_\cS(X)$, so when working with $\dsw(X)$, we must check that that category is preserved by any functors we wish to use.  The following lemma is a useful tool for this.

\begin{lem}\label{lem:weil-obj}
For $\cF \in \ddel_\cS(X)$, the following conditions are equivalent:
\begin{enumerate}
\item $\cF \in \dsw(X)$.\label{it:weil-dsw}
\item For each stratum $j_s: X_s \to X$, we have $j_s^*\cF \in \dsw(X_s)$.\label{it:weil-std}
\end{enumerate}
\end{lem}

For a similar statement on $X \otimes \Fqb$, see~\cite[Lemma~4.4.5]{bgs}.  We will prove this simultaneously with the following result.

\begin{prop}\label{prop:weil-func0}
Let $X$ be a stratified variety.  If $h: Y \to X$ is the inclusion of a locally closed union of strata, then the functors $h^*$ and $h^!$ (resp.~$h_*$ and $h_!$) send objects of $\dsw(X)$ to $\dsw(Y)$ (resp.~$\dsw(Y)$ to $\dsw(X)$).
\end{prop}

\begin{proof}[Proof of Lemma~\ref{lem:weil-obj} and Proposition~\ref{prop:weil-func0}]
Let $\dsw(X)' \subset \ddel_\cS(X)$ be the full triangulated subcategory consisting of objects satisfying condition~\eqref{it:weil-std} of the lemma.  This is the category referred to as the category of \emph{constructible complexes} in~\cite[\S2.2.10]{bbd}.  The assumption~\eqref{eqn:constr-hyp} corresponds to the condition in~\cite[\S2.2.10(c)]{bbd}, and according to that statement, the analogue of the proposition holds for $\dsw(X)'$ and $\dsw(Y)'$.  As a consequence, the formalism of gluing of $t$-structures applies in $\dsw(X)'$.  In particular, $\dsw(X)'$ admits a perverse $t$-structure; cf.~\cite[\S 2.2.17]{bbd}.  Since that $t$-structure is bounded and has a finite-length heart, $\dsw(X)'$ is generated as a triangulated category by the simple perverse sheaves it contains.  But a simple perverse sheaf $j_{s!*}\cL$ clearly lies in $\dsw(X)'$ if and only if the irreducible local system $\cL$ is isomorphic to some $\uQlb[\dim X_s]\la n\ra$.  Thus, $\dsw(X)'$ contains and is generated by the $\ICm_s$, so $\dsw(X)' = \dsw(X)$.
\end{proof}

\begin{prop}\label{prop:weil-func1}
For any stratified variety $X$, the functors $\D$, $\Lotimes$, and $\cRHom$ take objects of $\dsw(X)$ to $\dsw(X)$. 
\end{prop}
\begin{proof}
The statement for $\D$ is clear, since $\D\ICm_s \cong \ICm_s$ for all $s$, and these objects generate $\dsw(X)$.  For $\Lotimes$, we first consider the special case where $X$ consists of a single stratum $X_s$.  In this case, $\ICm_s \cong \uQlb[\dim X]\la-\dim X\ra$, so
\begin{multline*}
\ICm_s \Lotimes \ICm_s \cong (\uQlb \Lotimes \uQlb)[2\dim X]\la -2\dim X\ra
\cong\\
 \uQlb[2\dim X]\la -2\dim X\ra \cong \ICm_s[\dim X]\la -\dim X\ra,
\end{multline*}
and the desired statement follows.  For general $X$, suppose $\cF, \cG \in \dsw(X)$.  Given a stratum $X_s \subset X$, we use the formula
\[
j_s^*(\cF \Lotimes \cG) \cong j_s^*\cF \Lotimes j_s^*\cG
\]
to see that each $j_s^*(\cF \Lotimes \cG)$ lies in $\dsw(X_s)$, so $\cF \Lotimes \cG \in \dsw(X)$ by Lemma~\ref{lem:weil-obj}.  Finally, the result holds for $\cRHom$ since $\cRHom(\cF,\cG) \cong \D(\cF \Lotimes \D\cG)$.
\end{proof}

\begin{defn}\label{defn:weak-strat}
Let $X$ and $Y$ be varieties endowed with stratifications $\cS$ and $\cT$, respectively.  A morphism $f: X \to Y$ is said to be \emph{weakly stratified} if for each stratum $Y_t \subset Y$, its preimage $f^{-1}(Y_t) \subset X$ is a union of strata.
\end{defn}

\begin{prop}\label{prop:weil-func2}
Let $X$ and $Y$ be varieties endowed with stratifications $\cS$ and $\cT$, respectively, and let $f: X \to Y$ be a weakly stratified morphism.  Then $f^*$ and $f^!$ take objects of $\dw_\cT(Y)$ to objects of $\dsw(X)$.
\end{prop}
\begin{proof}
For each stratum $X_s \subset X$, there is a unique stratum $Y_t \subset Y$ such that $X_s \subset f^{-1}(Y_t)$.  Let $f_s = f|_{X_s}: X_s \to Y_t$.  For $\ICm_t \in \Pervm_\cT(Y)$, we have $j_s^*f^*\ICm_t \cong f_s^* j_t^*\ICm_t$.  We know that $j_t^*\ICm_t \in \dw_\cT(Y_t)$ by Lemma~\ref{lem:weil-obj}.  We clearly have $f_s^*\uQlb \cong \uQlb$, so $f_s^*$ takes $\dw_\cT(Y_t)$ to $\dsw(X_s)$.  Therefore, the object $f_s^* j_t^*\ICm_t \cong j_s^* f^*\ICm_t$ lies in $\dsw(X_s)$ for all $s$ and $t$.  Using Lemma~\ref{lem:weil-obj} again, we see that $f^*\ICm_t \in \dsw(X)$, so the proposition holds for $f^*$.  It then follows for $f^! \cong \D \circ f^* \circ \D$.
\end{proof}

The preceding results cover most of the functors we will encounter.  (We will prove that certain push-forwards preserve the Weil category in Section~\ref{ss:stratified}.)  For the most part, we will suppress further mention of $\ddel_\cS(X)$ and silently regard sheaf operations as functors on the Weil category.  It is well known that all the usual sheaf operations enjoy the following property.

\begin{defn}\label{defn:geometric}
A functor $F: \dsw(X) \to \dw_\cT(Y)$ is said to be \emph{geometric} if it is a functor of triangulated categories that is equipped with a natural transformation
\begin{equation}\label{eqn:commute-uhom}
\uRHom(\cF,\cG) \to \uRHom(F(\cF), F(\cG))
\end{equation}
and it ``commutes with $\exs$,'' i.e., there exists a triangulated functor $\bar F: \ds(X) \to \dc_\cT(Y)$ such that
\[
\bar F \circ \exs \cong \exs \circ F.
\]
\end{defn}

For a geometric functor, the natural transformation~\eqref{eqn:commute-uhom}, combined with~\eqref{eqn:dsw-ses}, gives rise to a commutative diagram
\begin{equation}\label{eqn:dsw-ses-comm}
\vcenter{\tiny\xymatrix@C=20pt{
0 \ar[r] &
   \uHom^{i-1}(\cF,\cG)_\Fr \ar[d]\ar[r] &
   \Hom^i(\cF,\cG) \ar[d]\ar[r] &
   \uHom^i(\cF,\cG)^\Fr \ar[d]\ar[r] & 0 \\
0 \ar[r] &
   \uHom^{i-1}(F(\cF),F(\cG))_\Fr \ar[r] &
   \Hom^i(F(\cF),F(\cG)) \ar[r] &
   \uHom^i(F(\cF),F(\cG))^\Fr \ar[r] & 0}}
\end{equation}

\subsection{Weight filtrations in the Weil category}

Even though $\dsw(X)$ is not a mixed triangulated category in general, an analogue of Lemma~\ref{lem:mixed-tri-basic}\eqref{it:mixed-tri-basic-dt} still holds.

\begin{lem}\label{lem:weight-trunc}
Let $\cF$ be an object of $\dsw(X)$ with weights${}\ge a$ and${}\le b$.  For any integer $w$, there is a distinguished triangle
\begin{equation}\label{eqn:weight-trunc}
\cF_{\le w} \to \cF \to \cF_{>w} \to
\end{equation}
where $\cF_{\le w}$ has weights ${}\ge a$ and${}\le w$, and $\cF_{> w}$ has weights${}>w$ and${}\le b$.
\end{lem}
\begin{proof}
The statement holds trivially unless $a \le w \le b$, so assume that that is the case.  We proceed by induction on the ``total length'' of $\cF$, i.e., on the sum of the lengths of the perverse sheaves $\pH^i(\cF)$.  If the total length is $1$, then $\cF$ is a shift of a simple perverse sheaf, and so is pure.  The result holds trivially in this case as well.

Otherwise, let $k$ be the smallest integer such that $\pH^k(\cF)$ has composition factors of weight${}\le w+k$.  (If there is no such $k$, then $\cF$ has weights${}>w$, and the lemma holds trivially.)  Form the distinguished triangle
\[
\tau_{<k}\cF \to \cF \to \tau_{\ge k}\cF \to.
\]
Now, consider the term $\cG = W_{w+k}\pH^k(\cF)$ in the weight filtration of $\pH^k(\cF)$.  The inclusion $\cG \subset \pH^k(\cF)$ gives us a natural morphism $f: \cG[-k] \to \tau_{\ge k}\cF$.  Note that the truncation $\tau_{<k}\cF$ must have weights${}>w$.  It follows that
\[
\Hom(\cG[-k], \tau_{<k}\cF[1]) = 0
\]
by~\cite[Proposition~5.1.15(ii)]{bbd}.  Therefore, $f$ factors through $\cF$, say by $\tilde f: \cG[-k] \to \cF$.  Complete this to a distinguished triangle
\begin{equation}\label{eqn:wt-trunc1}
\cG[-k] \ovto{\tilde f} \cF \to \cF' \to.
\end{equation}
Since the induced map $\pH^k(\cG[-k]) \to \pH^k(\cF)$ is injective, we see that $\pH^i(\cF') \cong \pH^i(\cF)$ if $i \ne k$, and that $\pH^k(\cF') \cong \pH^k(\cF)/\cG$.  That is, $\cF'$ has lower total length than $\cF$, so by induction, there exists a distinguished triangle
\begin{equation}\label{eqn:wt-trunc2}
\cF'_{\le w} \to \cF' \to \cF'_{>w} \to
\end{equation}
with weights as specified in the statement of the lemma.  Using the ``$*$'' notation of~\cite[\S 1.3.9]{bbd} (cf.~Lemma~\ref{lem:kosorlov-trunc}), we see from~\eqref{eqn:wt-trunc1} and~\eqref{eqn:wt-trunc2} that
\[
\cF \in \{\cG[-k]\} * (\{\cF'_{\le w}\} * \{\cF'_{>w}\})
= (\{\cG[-k]\} * \{\cF'_{\le w}\}) * \{\cF'_{>w}\}.
\]
In particular, there is some object $\cF'' \in \{\cG[-k]\} * \{\cF'_{\le w}\}$ together with a distinguished triangle
\[
\cF'' \to \cF \to \cF'_{>w} \to.
\]
Since $\cG[-k]$ and $\cF'_{\le w}$ each have weights${}\le w$, the same holds for all objects in $\{\cG[-k]\} * \{\cF'_{\le w}\}$.  Thus, this triangle is of the desired form.
\end{proof}

\begin{cor}\label{cor:weight-filt}
Let $\Purew_\cS(X) \subset \dsw(X)$ denote the category of pure objects of weight $0$.  If $\cF \in \dsw(X)$ has weights${}\ge a$ and ${}\le b$, then
\[
\cF \in \Purew_\cS(X)[a] * \Purew_\cS(X)[a+1] * \cdots * \Purew_\cS(X)[b].
\]
\end{cor}

\subsection{Mixed perverse sheaves and the miscible category}
\label{ss:miscible}

As noted earlier, the most obvious marker of the failure of $\Pervsw(X)$ and $\dsw(X)$ to be mixed categories is the fact that pure objects need not be semisimple.  As a first step towards remedying this, we must discard some objects from our categories.  Consider first the full subcategory of $\Pervsw(X)$ given by
\[
\Pervsm(X) = \{ \cF \in \Pervsw(X) \mid \text{for all $i$, $\gr^W_i\cF$ is semisimple} \},
\]
called the \emph{category of mixed perverse sheaves}.  This is not a Serre subcategory of $\Pervsw(X)$, as it is not closed under extensions, but it is closed under subquotients.  In particular, the kernel and cokernel of any morphism in $\Pervsm(X)$ is again in $\Pervsm(X)$, so $\Pervsm(X)$ is naturally an abelian category.  It is easy to see that $\Pervsm(X)$ is, in fact, a mixed category, so the terminology is justified.

Next, we consider the full additive subcategory 
\[
\Pure_\cS(X) = \{ \text{pure semisimple objects of weight $0$ in $\dsw(X)$} \}
\]
of $\dsw(X)$.  When there is no ambiguity about the stratification, we will usually just denote this category by $\Pure(X)$.  Inspired by Corollary~\ref{cor:weight-filt}, we introduce the following notion.

\begin{defn}
An object $\cF \in \dsw(X)$ is said to be \emph{miscible} if
\[
\cF \in \Pure(X)[a] * \Pure(X)[a+1] * \cdots * \Pure(X)[b]
\]
for some integers $a \le b$.  A geometric functor $F: \dsw(X) \to \dw_\cT(Y)$ is said to be \emph{miscible} if it takes miscible objects to miscible objects.
\end{defn}

The full subcategory of $\dsw(X)$ consisting of miscible objects is denoted
\[
\dsmisc(X).
\]
Unfortunately, this is \emph{not} a triangulated subcategory of $\dsw(X)$ (except in the trivial case where $X$ is the empty variety), because the cone of a morphism between two miscible objects need not be miscible.  It is desirable to replace $\dsmisc(X)$ by a smaller category that is triangulated and that contains $\Pervsm(X)$ as the heart of $t$-structure.  The authors do not know how to do this in general, but in the next section we will describe a solution for a very special class of stratifications.

\section{Affable stratifications}
\label{sect:affable}

For the remainder of the paper, we will restrict ourselves to varieties whose stratifications are of one of the following two types:

\begin{defn}\label{defn:aff-even}
An \emph{affine even stratification} of $X$ is a stratification $\cS = \{X_s\}_{s \in S}$ satisfying the following two conditions:
\begin{enumerate}
\item Each $X_s$ is isomorphic to the affine space $\bA^{\dim X_s}$.
\item For all $s,t \in S$ and $i \in \Z$, the sheaf $H^i(\ICm_s|_{X_t})$ vanishes if $i \not\equiv \dim X_s \pmod 2$, and is isomorphic to a direct sum of copies of $\uQlb\la i\ra$ otherwise.
\end{enumerate}
\end{defn}

\begin{defn}
A stratification $\cS$ of $X$ is said to be \emph{affable} if it admits a refinement $\cS'$ that is an affine even stratification.
\end{defn}

The main examples come from representation theory: according to~\cite[Corollary~4.4.3]{bgs}, the stratification of a partial flag variety for a reductive group by orbits of a Borel subgroup is an affine even stratification.  It follows that the stratification by orbits of a parabolic subgroup is affable.  Similar statements hold for partial affine flag varieties, stratified by orbits of an Iwahori or parahoric subgroup. Note that these results are stronger than the older Kazhdan--Lusztig theorem~\cite{kl} on pointwise purity: the latter states only that each $\ICm_s|_{X_t}$ on a flag variety is a pure object of $\dsw(X_t)$, whereas~\cite[Corollary~4.4.3]{bgs} tells us in addition that each cohomology sheaf belongs to $\Pervsm(X_t)$.  

\subsection{The Weil category for an affable stratification}
\label{ss:aff-weil}

To make sure that the considerations of Section~\ref{sect:mixedweil} apply to affine even and affable stratifications, we must check that~\eqref{eqn:constr-hyp} holds.  

\begin{lem}\label{lem:aff-constr-hyp}
Condition~\eqref{eqn:constr-hyp} holds for any affable stratification $\cS$ of $X$.
\end{lem}
\begin{proof}
We proceed by induction on the number of strata in $X$.  If $X$ consists of a single stratum, there is nothing to prove.  Otherwise, let $j_s: X_s \to X$ be the inclusion of an open stratum, and let $i: Z \to X$ be the inclusion of the complementary closed subvariety.  Consider the distinguished triangle
\[
j_{s!}\uQlb[\dim X_s]\la -\dim X_s\ra \to \ICm_s \to i_*i^*\ICm_s \to.
\]
We know by induction that the induced stratification on $Z$ satisfies~\eqref{eqn:constr-hyp}.  It is clear from the definition that $i^*\ICm_s$ has the second property in Lemma~\ref{lem:weil-obj}, so $i^*\ICm_s \in \dsw(Z)$.  By Proposition~\ref{prop:weil-func1}, $i^!\ICm_s \cong \D i^*\ICm_s$ lies in $\dsw(Z)$ as well.  It follows from the distinguished triangle
\[
i^!j_{s!}\uQlb[\dim X_s]\la -\dim X_s\ra \to i^!\ICm_s \to i^*\ICm_s \to
\]
that $i^!j_{s!}\uQlb$ lies in $\dsw(Z)$, so using $\D$ again, we have $i^*j_{s*}\uQlb \in \dsw(Z)$, and this implies~\eqref{eqn:constr-hyp}.
\end{proof}

\subsection{The mixed category}
\label{ss:aff-mixed}

If $\cS$ is an affable stratification of $X$, we define the \emph{mixed category} of $X$ to be the triangulated category
\begin{equation}\label{eqn:dsm-defn}
\dsm(X) = \Kb\Pure(X).
\end{equation}
Our first goal is to describe the relationship between this category and $\dsmisc(X)$.

\begin{thm}\label{thm:misc-infext}
There is a natural equivalence of additive categories 
\begin{equation}\label{eqn:misc-infext}
I: \pse\dsm(X) \simto \dsmisc(X).
\end{equation}
\end{thm}

Once this result is proved, we will identify $\pse\dsm(X)$ with $\dsmisc(X)$.  We will explain in Section~\ref{ss:aff-infext} how to transfer various notions and results from Section~\ref{sect:infext} to the setting of $\dsmisc(X)$.  For now, note that this identification gives us a canonical functor
\[
\incl: \dsm(X) \to \dsmisc(X).
\]
By an abuse of notation, we will also write $\incl$ for the composition $\dsm(X) \to \dsmisc(X) \to \dsw(X)$.  Let $\zeta = \exs \circ \incl$, so that we have a commutative diagram
\[
\xymatrix@C=0pt{
\dsm(X) \ar[rr]^{\incl}\ar[dr]_\degr && \dsw(X) \ar[dl]^\exs \\
& \ds(X) }
\]
In the course of the proof of Theorem~\ref{thm:misc-infext}, we will simultaneously establish the following statement, which tells us in part that $\dsm(X)$ and $\Pervsm(X)$ are mixed versions of $\ds(X)$ and $\Pervs(X)$, respectively.

\begin{prop}\label{prop:mixed-version}
\begin{enumerate}
\item $\dsm(X)$ admits a natural $t$-structure whose heart can be identified with $\Pervsm(X)$, and the functor $\incl: \dsm(X) \to \dsw(X)$ is $t$-exact and restricts to the inclusion functor $\Pervsm(X) \to \Pervsw(X)$.\label{it:mixed-tstruc}
\item The functor $\degr: \dsm(X) \to \ds(X)$ is $t$-exact and induces an isomorphism\label{it:mixed-isom}
\[
\bigoplus_{n \in \Z} \Hom_{\dsm(X)}(\cF,\cG\la n\ra) \simto \Hom_{\ds(X)}(\degr\cF,\degr\cG).
\]
\end{enumerate}
\end{prop}

The proofs of Theorem~\ref{thm:misc-infext} and Proposition~\ref{prop:mixed-version} will occupy most of this section.  We begin by recalling some key results about affine even stratifications from~\cite{bgs}.  Those results are mostly stated not for $\Pervsm(X)$ but rather for the Serre subcategory
\[
\Pervsm(X)' \subset \Pervsm(X)
\qquad\text{generated by}\qquad
\{ \ICm_s\la n \ra \mid n \equiv \dim X_s \pmod 2\}.
\]
Note that $\Pervsm(X)'$ is stable under integral Tate twists $\cF \mapsto \cF\la 2n\ra$.

\begin{lem}\label{lem:even-pervm}
We have $\Pervsm(X) \cong \Pervsm(X)' \oplus \Pervsm(X)'\la 1\ra$.
\end{lem}
\begin{proof}
Let $\cF, \cG \in \Pervsm(X)'$ be two simple objects.  It is sufficient to show that $\Ext^1_{\Pervsm(X)}(\cF,\cG\la 1\ra) = 0$.  That, in turn, would follow from the vanishing of the $\Ext^1$-group in the larger category $\Pervsw(X)$.  We now proceed by induction on the number of strata in $X$.  Choose a closed stratum $j_s: X_s \to X$, and let $h: U \to X$ be the complementary open subvariety.  From the distinguished triangle $j_{s*}j_s^!\cG\la 1\ra \to \cG\la 1\ra \to h_*h^*\cG\la 1\ra \to$ in $\dsw(X)$, we obtain the long exact sequence
\[
\cdots \to \Hom(j_s^*\cF, j_s^!\cG\la 1\ra[1]) \to \Hom(\cF,\cG\la1\ra[1]) \to \Hom(h^*\cF, h^*\cG\la1\ra[1]) \to \cdots.
\]
The last term vanishes by induction.  For the first term, $j_s^*\cF$ (resp.~$j_s^!\cG\la 1\ra$) lies in the triangulated subcategory of $\dsw(X_s)$ generated by $\uQlb\la n\ra$ with $n \equiv \dim X_s \pmod 2$ (resp.~$n \not\equiv \dim X_s \pmod 2$).  It is well known that on $X \cong \bA^{\dim X_s}$, we have $\Hom_{\dsw(X_s)}(\uQlb, \uQlb[i]\la m\ra) = 0$ if $m$ is odd (cf.~Lemma~\ref{lem:affine-basic}\eqref{it:aff-cohom}), so the first term vanishes as well.  It follows that $\Ext^1_{\Pervsw(X)}(\cF,\cG\la 1\ra) = 0$.
\end{proof}

In the following theorem, parts~\eqref{it:bgs-mix-proj}--\eqref{it:bgs-mixed} are proved in~\cite{bgs} only for $\Pervsm(X)'$ (which is denoted $\tilde{\mathcal{P}}$ in {\it loc.~cit.}), but it is clear from the preceding lemma that the same statements hold for $\Pervsm(X)$ as well.

\begin{thm}\label{thm:bgs-aff-even}
Suppose $X$ has an affine even stratification.
\begin{enumerate}
\item \cite[Theorem~3.3.1 and Corollary~3.3.2]{bgs} The category $\Pervs(X)$ has enough projectives and enough injectives, and finite cohomological dimension.  In addition, the realization functor
\[
\real: D^b\Pervs(X) \to \ds(X)
\]
is an equivalence of categories.\label{it:bgs-equiv}
\item \cite[Lemma~4.4.8]{bgs} The category $\Pervsm(X)$ has enough projectives and enough injectives, and finite cohomological dimension.  An object $\cF \in \Pervsm(X)$ is projective (resp.~injective) if and only if $\exs(\cF) \in \Pervs(X)$ is projective (resp.~injective).\label{it:bgs-mix-proj}
\item \cite[Theorem~4.4.4]{bgs} The category $\Pervsm(X)$ is Koszul. \label{it:bgs-koszul}
\item \cite[Theorem~4.4.4]{bgs} The functor $\degr = \exs|_{\Pervsm(X)}: \Pervsm(X) \to \Pervs(X)$ makes $\Pervsm(X)$ into a mixed version of $\Pervs(X)$.  The composition
\[
\Db\Pervsm(X) \to \Db\Pervs(X) \xrightarrow[\real]{\sim} \ds(X)
\]
makes $\Db\Pervsm(X)$ into a mixed version of $\ds(X)$. \label{it:bgs-mixed}\qed
\end{enumerate}
\end{thm}

The proof of the next lemma depends on Proposition~\ref{prop:mixed-version}.  Due to the structure of the argument for Theorem~\ref{thm:misc-infext}, it is convenient to give a contingent proof of this lemma now, even though Proposition~\ref{prop:mixed-version} has not yet been proved.

\begin{lem}\label{lem:aff-even-semis}
For any $\cF, \cG \in \dsm(X)$, the action of $\Fr$ on $\uHom(\incl\cF,\incl\cG)$ is semisimple.  In other words, $\uHom(\incl\cF,\incl\cG) \in \Pervsm(\pt)$.  Moreover, the functor $\incl$ induces an isomorphism
\begin{equation}\label{eqn:aff-even-semis}
\Hom_{\dsm(X)}(\cF,\cG\la n\ra) \simto \left(\uHom(\incl\cF,\incl\cG)\la n\ra\right)^\Fr.
\end{equation}
In addition, there is a natural isomorphism
\begin{equation}\label{eqn:aff-even-dsw}
\Hom_{\dsw(X)}(\incl\cF,\incl\cG) \simeq \Hom_{\dsm(X)}(\cF,\cG) \oplus \Hom_{\dsm(X)}(\cF,\cG[-1]).
\end{equation}
\end{lem}
\begin{rmk}\label{rmk:aff-even-semis}
The proof given below uses only the formal properties stated in Proposition~\ref{prop:mixed-version} and general properties of $\dsw(X)$; it does not make explicit use of the definition of $\dsm(X)$.
\end{rmk}
\begin{proof}
For brevity, let us put $A = \uHom(\incl\cF,\incl\cG)$.  This is an object of $\Pervsw(\pt)$.  The weight filtration for a perverse sheaf on a point splits (cf.~Lemma~\ref{lem:affine-basic}\eqref{it:aff-hered} below), so we may write $A \simeq \bigoplus_n A^n$, where each $A^n$ is pure of weight $n$.  Note that $(A\la -n\ra)^\Fr$ is a subspace of $A^n$: indeed, is the $\Fr$-eigenspace of eigenvalue $q^{n/2}$.  Thus, to show that $\Fr$ acts semisimply on $A$, it suffices to show that $(A\la -n\ra)^\Fr = A^n$ for all $n$.  

To prove the latter assertion, consider the map
\begin{equation}\label{eqn:aff-even-semis-v}
\Hom_{\dsm(X)}(\cF,\cG) \to \Hom_{\ds(X)}(\degr\cF, \degr\cG)
\end{equation}
induced by $\degr$.  Because $\degr \cong \exs \circ \incl$, we have that $\Hom_{\ds(X)}(\degr\cF, \degr\cG)$ is canonically isomorphic to the vector space $A$ with the grading forgotten, cf.~\eqref{eqn:hom-forget}.  Furthermore, the map~\eqref{eqn:aff-even-semis-v} factors as
\begin{equation}\label{eqn:aff-even-semis-i}
\Hom_{\dsm(X)}(\cF,\cG) \to \Hom_{\dsw(X)}(\incl\cF,\incl\cG) \to 
A^\Fr \to A^0 \to A
\end{equation}
where the first map is induced by $\incl$, and the second by $\exs$, cf.~\eqref{eqn:dsw-ses}.  Taking Tate twists and summing up over $n$, we can build the diagram
\begin{equation}\label{eqn:aff-even-semis-total}
\bigoplus_n \Hom_{\dsm(X)}(\cF, \cG\la -n\ra) \to \bigoplus_n (A\la -n\ra)^\Fr \to \bigoplus_n A^n \simeq A.
\end{equation}
The second map is injective, but the composition is an isomorphism by Proposition~\ref{prop:mixed-version}.  Therefore, the second map must be an isomorphism as well, so $(A\la -n\ra)^\Fr = A^n$, as desired.

We now see that the first map in~\eqref{eqn:aff-even-semis-total} is also an isomorphism.  This establishes~\eqref{eqn:aff-even-semis}.  Since the composition
\[
\Hom_{\dsm(X)}(\cF,\cG) \to \Hom_{\dsw(X)}(\incl\cF,\incl\cG) \to \uHom(\incl\cF,\incl\cG)^\Fr
\]
is an isomorphism, the first map provides a canonical splitting of the short exact sequence~\eqref{eqn:dsw-ses}, and we have
\[
\Hom_{\dsw(X)}(\incl\cF,\incl\cG) \cong \uHom(\incl\cF,\incl\cG)^\Fr \oplus \uHom(\incl\cF, \incl\cG[-1])_\Fr.
\]
But since the action of $\Fr$ is already known to be semisimple, there is a canonical identification $\uHom(\incl\cF, \incl\cG[-1])_\Fr = \uHom(\incl\cF,\incl\cG[-1])^\Fr$, and~\eqref{eqn:aff-even-dsw} follows.
\end{proof}

\begin{proof}[Proof of Theorem~\ref{thm:misc-infext} and Proposition~\ref{prop:mixed-version} for affine even stratifications]
We begin by temporarily changing the definition of $\dsm(X)$ to $\dsm(X) = \Db\Pervsm(X)$.  We will first show that the desired results hold with this modified definition, and then we will see that it is equivalent to~\eqref{eqn:dsm-defn}.  We know from Theorem~\ref{thm:bgs-aff-even}\eqref{it:bgs-koszul} that $\Pervsm(X)$ is a Koszul category, so by Proposition~\ref{prop:kosorlov-orlov}, there is a natural equivalence of categories
\begin{equation}\label{eqn:mix-defn2}
\dsm(X) \cong \Kb(\Orl(\Pervsm(X))).
\end{equation}

The inclusion functor $\Pervsm(X) \to \Pervsw(X)$ is exact, so it gives rise to a derived functor $\Db\Pervsm(X) \to \Db\Pervsw(X)$.  Define $\jmath: \dsm(X) \to \dsw(X)$ to be the composition
\[
\dsm(X) = \Db\Pervsm(X) \to \Db\Pervsw(X) \xto{\real} \dsw(X).
\]
By construction, $\jmath$ is a functor of triangulated categories that commutes with Tate twists and restricts to the inclusion functor $\Pervsm(X) \to \Pervsw(X)$.  In particular, for any simple object $\ICm_s \in \Pervsm(X)$, we have $\jmath(\ICm_s[n]\la -n\ra) \cong \ICm_s[n]\la -n\ra$.  The full additive subcategories
\[
\Orl(\Pervsm(X)) \subset \Db\Pervsm(X)
\qquad\text{and}\qquad
\Pure(X) \subset \dsw(X)
\]
both consist of objects that are direct sums of various $\ICm_s[n]\la -n\ra$, so $\jmath$ restricts to an additive functor $\jmath|_{\Orl(\Pervsm(X))}: \Orl(\Pervsm(X)) \to \Pure(X)$.

The functor $\jmath$ has the properties attributed to $\incl$ in Proposition~\ref{prop:mixed-version}: part~\eqref{it:mixed-tstruc} is obvious, and part~\eqref{it:mixed-isom} holds by Theorem~\ref{thm:bgs-aff-even}\eqref{it:bgs-mixed}.  Therefore, according to Remark~\ref{rmk:aff-even-semis}, we may use Lemma~\ref{lem:aff-even-semis} in our setting if we replace $\incl$ by $\jmath$ in that statement.  In particular, the formula~\eqref{eqn:aff-even-dsw} shows that $\jmath: \dsm(X) \to \dsw(X)$ extends in a canonical way to an additive functor
\[
\tilde\jmath: \pse\dsm(X) \to \dsw(X).
\]
Moreover, this functor is fully faithful and makes the following diagram commute:
\begin{equation}\label{eqn:mix-defn-commute}
\vcenter{\xymatrix@R=10pt{ & \pse\dsm(X) \ar[dd]^{\tilde\jmath} \\
\dsm(X) \ar[ur]^{\incl} \ar[dr]_{\jmath} \\
& \dsw(X) }}
\end{equation}

We claim that the essential image of $\tilde\jmath$ is the full subcategory $\dsmisc(X) \subset \dsw(X)$.  Note that an object lies in this essential image if and only if it is in the essential image of $\jmath$, so we prove the claim by working with $\jmath$ instead.  First, given $\cF \in \dsmisc(X)$, say
\[
\cF \in \Pure(X)[a] * \Pure(X)[a+1] * \cdots * \Pure(X)[b],
\]
we will prove by induction on $b-a$ that there exists an object
\begin{equation}\label{eqn:mix-orl-wt}
\tilde\cF \in \Orl(\Pervsm(X))[a] * \Orl(\Pervsm(X))[a+1] * \cdots * \Orl(\Pervsm(X))[b] \subset \dsm(X)
\end{equation}
such that $\jmath(\tilde\cF) \cong \cF$.  In the case where $b-a = 0$, $\cF$ is a direct sum of objects of the form $\ICm_s[n]\la a-n\ra$, and every such object is clearly in the essential image of $\jmath|_{\Orl(\Pervsm(X))[a]}$.  In the general case, there is a distinguished triangle
\[
\cF' \to \cF \to \cF'' \ovto{\delta} \cF'[1]
\]
with $\cF' \in \Pure(X)[a]$ and
\[
\cF'' \in \Pure(X)[a+1] * \Pure(X)[a+2] * \cdots * \Pure(X)[b].
\]
By induction, we may assume that $\cF' = \jmath(\tilde \cF')$ and $\cF'' = \jmath(\tilde \cF'')$ for some objects $\tilde\cF', \tilde\cF'' \in \dsm(X)$, where $\tilde\cF' \in \Orl(\Pervsm(X))[a]$ and
\[
\tilde\cF'' \in \Orl(\Pervsm(X))[a+1] * \Orl(\Pervsm(X))[a+2] * \cdots * \Orl(\Pervsm(X))[b].
\]
In view of the equivalence~\eqref{eqn:mix-defn2} and Theorem~\ref{thm:kosorlov}, we have $\Hom_{\dsm(X)}(\tilde\cF'',\tilde\cF') = 0$, so using Lemma~\ref{lem:aff-even-semis} again, we see that $\jmath$ induces an isomorphism
\[
\Hom_{\dsw(X)}(\cF'',\cF'[1]) \cong \Hom_{\dsm(X)}(\tilde\cF'',\tilde\cF'[1]).
\]
In particular, $\delta: \cF'' \to \cF'[1]$ is equal to $\jmath(\tilde\delta)$ for some morphism $\tilde\delta: \tilde\cF'' \to \tilde\cF'[1]$ in $\dsm(X)$.  If we let $\tilde\cF$ denote the cocone of $\tilde\delta$, then we have $\cF \cong \jmath(\tilde\cF)$.  Thus, every object of $\dsmisc(X)$ lies in the essential image of $\jmath$ and $\tilde\jmath$.

Conversely, by~\eqref{eqn:mix-defn2}, every object of $\dsm(X)$ has the property~\eqref{eqn:mix-orl-wt} for some integers $a \le b$.  Since $\jmath(\Orl(\Pervsm(X))[n]) \subset \Pure(X)[n]$, it follows that $\jmath$ takes values in $\dsmisc(X)$.  Thus, the essential image of $\tilde\jmath$ is $\dsmisc(X)$, and $\tilde\jmath$ induces the desired equivalence~\eqref{eqn:misc-infext}.  We have now established both Theorem~\ref{thm:misc-infext} and Proposition~\ref{prop:mixed-version} for our modified definition of $\dsm(X)$ together with the functor $\jmath: \dsm(X) \to \dsw(X)$.  The same argument also shows that $\tilde\jmath$ restricts to an equivalence
\[
\Orl(\Pervsm(X)) \simto \Pure(X).
\]
Combining this with~\eqref{eqn:mix-defn2}, we see that our modified definition of $\dsm(X)$ is equivalent to~\eqref{eqn:dsm-defn}, as desired.
\end{proof}

The following facts emerged in the course of the preceding proof.

\begin{cor}\label{cor:aff-even-oldmix}
Suppose $X$ has an affine even stratification $\cS$.  There are natural equivalences
$\Pure(X) \cong \Orl(\Pervsm(X))$ and $\dsm(X) \cong \Db\Pervsm(X)$.
\end{cor}

The following fact brings notions from Sections~\ref{sect:hot-orlov} and~\ref{sect:kosorlov} into our setting.

\begin{prop}\label{prop:aff-even-orlov}
Suppose $X$ is endowed with an affine even stratification, and let $\Proj(X)$ and $\Inj(X)$ denote the categories of projective and injective objects, respectively, in $\Pervsm(X)$.  Then the three categories
\[
\Pure(X), \qquad \Proj(X), \qquad \Inj(X)
\]
are all Koszulescent Orlov categories.  If $\scA$ denotes any of these three categories, the inclusion functor $\scA \to \dsm(X)$ extends to an equivalence of triangulated categories $\Kb(\scA) \cong \dsm(X)$.
\end{prop}
\begin{proof}
These assertions follow from Theorem~\ref{thm:bgs-aff-even}\eqref{it:bgs-mix-proj} and Corollary~\ref{cor:aff-even-oldmix} together with Proposition~\ref{prop:kosorlov-orlov} and Theorem~\ref{thm:kosorlov-duality}. 
\end{proof} 

We now consider the case of a general affable stratification $\cS$.  Note that if $\cS_1$ is an affine even refinement of $\cS$, then a number of categories associated to $\cS$ can naturally be regarded as full subcategories of the corresponding categories associated to $\cS_1$.  Specifically, we have full subcategories
\[
\scE_\cS(X) \subset \scE_{\cS_1}(X)
\qquad\text{where $\scE$ is one of:}\quad \Pure, \Pervm, \dm, \pse\dm, \dmisc, \dw.
\]
To deduce the results for $\cS$ from what we have already proved for $\cS_1$, we need to give some of these full subcategories alternate descriptions.

Let $\cF$ be an object of $\dm_{\cS_1}(X)$, $\pse\dm_{\cS_1}(X)$, $\dmisc_{\cS_1}(X)$, or $\dw_{\cS_1}(X)$.  We say that $\cF$ is \emph{$\cS$-constructible} if we have $\pH^i(\cF) \in \Pervsw(X)$ for all $i$.  (In $\pse\dm_{\cS_1}(X)$, the notation $\pH^i(\cF)$ is an abuse that should be understood to mean $\pH^i(\pg \cF)$.)  Of course, in all but $\dw_{\cS_1}(X)$, an $\cS$-constructible object automatically satisfies the stronger condition that $\pH^i(\cF) \in \Pervsm(X)$.  

\begin{lem}\label{lem:affable-refine}
Suppose $\cS$ is an affable stratification of $X$, and let $\cS_1$ be an affine even refinement.  Then we have
\begin{equation}\label{eqn:affable-refine}
\scE_\cS(X) = \{ \cF \in \scE_{\cS_1}(X) \mid \text{$\cF$ is $\cS$-constructible} \},
\end{equation}
where $\scE$ denotes one of $\Pure$, $\Pervm$, $\dm$, $\pse\dm$, $\dmisc$, or $\dw$.
\end{lem}
\begin{proof}
For each of the categories $\Pure_\cS(X) \subset \Pure_{\cS_1}(X)$, $\Pervsm(X) \subset \Pervm_{\cS_1}(X)$, and $\dsw(X) \subset \dw_{\cS_1}(X)$, this assertion is obvious.

Next, let $\scD \subset \dm_{\cS_1}(X)$ denote the full subcategory consisting of $\cS$-constructible objects.  This is a triangulated subcategory; it is generated as a triangulated category by $\cS$-constructible objects in $\Pervm_{\cS_1}(X)$, and therefore by simple objects in $\Pervsm(X)$.  All such objects lie in $\dsm(X)$, so $\scD \subset \dsm(X)$.  On the other hand, objects of $\dsm(X)$ are obviously $\cS$-constructible, so $\dsm(X) \subset \scD$ as well.  Thus, $\dsm(X) = \scD$.

It is an immediate consequence that an object $\cF \in \pse\dm_{\cS_1}(X)$ lies in $\pse\dsm(X)$ if and only if $\cF$ is $\cS$-constructible.

It remains to consider the case of $\dsmisc(X) \subset \dmisc_{\cS_1}(X)$. Recall that the equivalence $I: \pse\dm_{\cS_1}(X) \to \dmisc_{\cS_1}(X)$ restricts to the identity functor on $\Pure_{\cS_1}(X)$, and therefore on $\Pure_{\cS}(X)$ as well.  

By Proposition~\ref{prop:mixed-version}\eqref{it:mixed-tstruc}, the functor $\incl: \dm_{\cS_1}(X) \to \dw_{\cS_1}(X)$ has the property that $\incl(\Pervsm(X)) = \Pervsm(X)$.  Since every object of $\dsm(X)$ is contained in some class of the form $\Pure_\cS(X)[a] * \Pure_\cS(X)[a+1] * \cdots * \Pure_\cS(X)[b]$, we have $\incl(\dsm(X)) \subset \dsmisc(X)$.  It follows that the equivalence $I$ of~\eqref{eqn:misc-infext} satisfies
\[
I(\pse\dsm(X)) \subset \dsmisc(X).
\]
Consider now the category $\dmisc_{\cS_1}(X) \cap \dsw(X)$, which is precisely the full subcategory of $\dmisc_{\cS_1}(X)$ consisting of $\cS$-constructible objects.  Because~\eqref{eqn:affable-refine} has already been shown for $\pse\dm_{\cS_1}(X)$, the fact that the equivalence $I: \pse\dm_{\cS_1}(X) \to \dmisc_{\cS_1}(X)$ preserves perverse cohomology means that it restricts to an equivalence
\[
I: \pse\dsm(X) \simto \dmisc_{\cS_1}(X) \cap \dsw(X).
\]
Therefore, $\dmisc_{\cS_1}(X) \cap \dsw(X) \subset \dsmisc(X)$.  But we obviously have $\dsmisc(X) \subset \dmisc_{\cS_1}(X) \cap \dsw(X)$, so $\dsmisc(X) = \dmisc_{\cS_1}(X) \cap \dsw(X)$, as desired.
\end{proof}

In the second paragraph of the preceding proof, we established the following statement, which we now record separately for future reference.

\begin{lem}\label{lem:affable-refine-dsm}
Suppose $\cS$ is an affable stratification of $X$, and let $\cS_1$ be an affine even refinement.  Then $\dsm(X)$ is the full triangulated subcategory of $\dm_{\cS_1}(X)$ generated by objects of $\Pervsm(X)$. \qed
\end{lem}

We are now almost finished with the general case.

\begin{proof}[Proof of Theorem~\ref{thm:misc-infext} and Proposition~\ref{prop:mixed-version} in general]
Given a variety $X$ with an affable stratification $\cS$, choose an affine even refinement $\cS_1$.  The known equivalence $I: \pse\dm_{\cS_1}(X) \simto \dmisc_{\cS_1}(X)$ and the functor $\incl: \dm_{\cS_1}(X) \to \dw_{\cS_1}(X)$ both preserve the property of being $\cS$-constructible.  Therefore, by Lemma~\ref{lem:affable-refine}, we obtain functors $I: \pse\dsm(X) \simto \dsmisc(X)$ and $\incl: \dsm(X) \to \dsw(X)$ with the desired properties simply by restricting the known functors defined using $\cS_1$ to suitable full subcategories.
\end{proof}

\begin{prop}\label{prop:affable-orlov}
If $\cS$ is an affable stratification, then $\Pure(X)$ is an Orlov category, and $\incl: \dsm(X) \to \dsw(X)$ induces an equivalence
\[
\Kos(\Pure_\cS(X)) \cong \Pervsm(X).
\]
\end{prop}
\begin{proof}
By Corollary~\ref{cor:aff-even-oldmix}, $\Pure_{\cS_1}(X)$ is an Orlov category, so its full subcategory $\Pure_\cS(X)$ is an Orlov category as well.  It follows from Corollary~\ref{cor:kosorlov-supp} that
\[
\Kos(\Pure_\cS(X)) = \Kos(\Pure_{\cS_1}(X)) \cap \dsm(X).
\]
By Lemma~\ref{lem:affable-refine}, this means that $\Kos(\Pure_\cS(X))$ is the full subcategory of $\cS$-constructible objects in $\Kos(\Pure_{\cS_1}(X))$.  Since $\incl$ preserves $\cS$-constructibility and induces an equivalence $\Kos(\Pure_{\cS_1}(X)) \cong \Pervm_{\cS_1}(X)$, the result follows.
\end{proof}

\begin{cor}\label{cor:aff-indep}
If $\cS$ is an affable stratification, the equivalence $I: \pse\dsm(X) \to \dsmisc(X)$ and the functor $\incl: \dsm(X) \to \dsw(X)$ are independent, up to isomorphism, of the choice of affine even refinement of $\cS$ used to define them.
\end{cor}
\begin{proof}
The two functors are related by a diagram like~\eqref{eqn:mix-defn-commute}, so it suffices to prove the statement for $\incl$.  The restriction of $\incl$ to the Orlov category $\Pure(X)$ is clearly independent of the choice of refinement.  The uniqueness of $\incl$ then follows from Theorem~\ref{thm:orlov-inf}.
\end{proof}

\subsection{$\dsmisc(X)$ as an infinitesimal extension}
\label{ss:aff-infext}

Theorem~\ref{thm:misc-infext} makes it possible to study $\dsmisc(X)$ using the machinery of Section~\ref{sect:infext}.  Note first that the isomorphism~\eqref{eqn:aff-even-dsw} can be identified with~\eqref{eqn:pseudotri-hom}.  On the other hand, the natural transformation $\iinf$ of~\eqref{eqn:iinf-defn} can be identified with the first map in~\eqref{eqn:dsw-ses}.  Thus, the following definition is consistent with Definition~\ref{defn:infext}.

\begin{defn}
A morphism $f: \cF \to \cG$ in $\dsw(X)$ is said to be \emph{infinitesimal} if $\exs(f) = 0$.
\end{defn}

There is no good notion of a \emph{genuine morphism} in $\dsmisc(X)$, however.  Recall that this is not a natural notion even in $\pse\dsm(X)$, in that it is not stable under conjugacy.  Such a notion can be transferred through an isomorphism of categories, but not through an equivalence as in Theorem~\ref{thm:misc-infext}.  As a substitute, we use the following notion.

\begin{defn}
A morphism $f: \cF \to \cG$ in $\dsmisc(X)$ is said to be \emph{miscible} if there is a commutative diagram
\[
\xymatrix{
\incl\tilde \cF \ar[r]^{\incl(\tilde f)} \ar[d]_\wr &
  \incl\tilde \cG \ar[d]^{\wr} \\
\cF \ar[r]_f & \cG}
\]
where $\tilde f: \tilde\cF \to \tilde \cG$ is some morphism in $\dsm(X)$, and the vertical maps are isomorphisms.
\end{defn}

For other terms from Section~\ref{sect:infext}, we encounter a problem: there are two {\it a priori} different notions of ``distinguished triangle'' in $\dsmisc(X)$, which we distinguish with the following terms.

\begin{defn}
A diagram $\cF' \to \cF \to \cF'' \to \cF'[1]$ in $\dsmisc(X)$ is called:
\begin{enumerate}
\item a \emph{Weil distinguished triangle} if it is a distinguished triangle in the triangulated category $\dsw(X)$;
\item a \emph{miscible distinguished triangle} if it is isomorphic to a diagram obtained by applying $\incl$ to a distinguished triangle in $\dsm(X)$.
\end{enumerate}
\end{defn}

Since $\incl: \dsm(X) \to \dsw(X)$ is a triangulated functor, every miscible distinguished triangle is a Weil distinguished triangle.  We will eventually prove the converse as well (see Theorem~\ref{thm:misc-tri}), so there is actually only a single notion of ``distinguished triangle'' in $\dsmisc(X)$.  In the meantime, the following criterion will be useful.

\begin{lem}\label{lem:misc-pt-func}
Let $F: \dsw(X) \to \dw_\cT(Y)$ be a miscible functor.  The following conditions on $F$ are equivalent:
\begin{enumerate}
\item $F$ takes every miscible morphism in $\dsmisc(X)$ to a miscible morphism in $\dmisc_\cT(Y)$.\label{it:misc-mor}
\item $F$ takes every miscible distinguished triangle in $\dsmisc(X)$ to a miscible distinguished triangle in $\dmisc_\cT(Y)$.\label{it:misc-dt}
\item $F$ restricts to a pseudotriangulated functor $F: \dsmisc(X) \to \dmisc_\cT(Y)$.\label{it:misc-pseudotri}
\end{enumerate}
\end{lem}
After Theorem~\ref{thm:misc-tri} is proved, this lemma will be superfluous, cf.~Remark~\ref{rmk:misc-pt}.
\begin{proof}
Since $F$ is a functor of triangulated categories that takes miscible objects to miscible objects, it certainly takes Weil distinguished triangles to Weil distinguished triangles.  Note that a Weil distinguished triangle is miscible if and only if at least one of its morphisms is miscible.  The equivalence of conditions~\eqref{it:misc-mor} and~\eqref{it:misc-dt} above follows.

Next, in view of Lemma~\ref{lem:aff-even-semis}, the commutative diagram~\eqref{eqn:dsw-ses-comm} shows that any miscible functor commutes with $\iinf \circ \pg$ in the sense of Definition~\ref{defn:pseudotri}.  From that definition, we see that conditions~\eqref{it:misc-dt} and~\eqref{it:misc-pseudotri} are equivalent.
\end{proof}

The last notion to translate from Section~\ref{sect:infext} is that of a ``genuine functor,'' whose definition is given below.  Table~\ref{tab:dict} summarizes the correspondence between the terminology of Section~\ref{sect:infext} and that of the present section. 

\begin{defn}
A miscible functor $F: \dsw(X) \to \dw_\cT(Y)$ is said to be \emph{genuine} if there is a functor of triangulated categories $\tilde F: \dsm(X) \to \dm_\cT(Y)$ such that $\incl \circ \tilde F \cong F \circ \incl$.  In that case, $\tilde F$ is said to be \emph{induced} by $F$.
\end{defn}

Note that a genuine functor in this sense automatically satisfies condition~\eqref{it:misc-dt} of Lemma~\ref{lem:misc-pt-func}, and therefore the other conditions as well.  In particular, a genuine functor automatically gives rise to a pseudotriangulated functor $\dsmisc(X) \to \dmisc_\cT(Y)$.  Note also that the definition of ``induced'' above is consistent with Definition~\ref{defn:infext-induc}, and recall from Lemma~\ref{lem:genuine-unique} that the induced functor $\tilde F$ of $F$, if it exists, is unique up to isomorphism.

\begin{table}[b]
\begin{center}
\begin{tabular}{|@{\quad}c@{\quad}|@{\quad}c@{\quad}|}
\hline
{\it Terminology for  $\pse\dsm(X)$} & {\it Terminology for $\dsmisc(X)$} \\
\hline
infinitesimal morphism & infinitesimal morphism \\
\hline
genuine morphism & --- \\
\hline
\begin{tabular}{c}
morphism conjugate to \\ a genuine morphism
\end{tabular} & miscible morphism \\
\hline
distinguished triangle & [miscible] distinguished triangle* \\
\hline
pseudotriangulated functor & miscible functor* \\
\hline
genuine functor & genuine functor \\
\hline
\end{tabular}
\end{center}
\caption{Dictionary for infinitesimal extensions and miscible sheaves.  For terms marked (*), see Theorem~\ref{thm:misc-tri}.}\label{tab:dict}
\end{table}

The following lemma is useful is reducing genuineness problems to the case of an affine even stratification.  We will frequently make silent use of it in the sequel, by stating results for general affable stratifications but considering only affine even ones in the proof.

\begin{lem}\label{lem:affable-genuine}
Let $\cS$ be an affable stratification of $X$ with affine even refinement $\cS_1$, and let $\cT$ be an affable stratification of $Y$ with affine even refinement $\cT_1$.  If $F: \dw_{\cS_1}(X) \to \dw_{\cT_1}(Y)$ is a genuine geometric functor that takes objects of $\dsw(X)$ to objects of $\dw_\cT(Y)$, then $F|_{\dsw(X)}: \dsw(X) \to \dw_\cT(Y)$ is genuine as well.
\end{lem}
\begin{proof}
Let $\tilde F: \dm_{\cS_1}(X) \to \dm_{\cT_1}(Y)$ be the functor induced by $F$.  Identifying $\dsm(X)$ and $\dm_\cT(Y)$ with full subcategories of $\dm_{\cS_1}(X)$ and $\dm_{\cT_1}(Y)$, respectively, we see from Lemma~\ref{lem:affable-refine} that $\tilde F$ must take objects of $\dsm(X)$ to objects of $\dm_\cT(Y)$.  The functor $\tilde F|_{\dsm(X)}: \dsm(X) \to \dm_\cT(Y)$ satisfies $\incl \circ \tilde F|_{\dsm(X)} \cong F|_{\dsw(X)} \circ \incl$, so $F|_{\dsw(X)}$ is genuine.
\end{proof}

\section{Sheaves on an affine space}
\label{sect:affine}

The easiest example of a variety with an affable stratification is, of course, an affine space $\bA^m$ endowed with the trivial stratification. In this section, we establish a large number of technical results on miscibility of objects and morphisms on an affine space.  These results lay the groundwork for the more general results to be proved in Section~\ref{sect:genuine}.  Throughout this section, $\cS$ will denote the trivial stratification on $\bA^m$.  

\begin{lem}\label{lem:affine-basic}
\begin{enumerate}
\item In $\dsw(\bA^m)$, we have\label{it:aff-cohom}
\[
\Hom^i(\uQlb, \uQlb\la n\ra) \cong
\begin{cases}
\Qlb & \text{if $i \in \{0,1\}$ and $n = 0$,} \\
0 & \text{otherwise.}
\end{cases}
\]
\item For any object $\cF \in \dsw(\bA^m)$, there is a (noncanonical) isomorphism $\cF \cong \bigoplus_i \pH^i(\cF)[-i]$. \label{it:aff-hered}
\item The weight filtration of a perverse sheaf $\cF \in \Pervsw(\bA^m)$ splits canonically.  That is, there is a canonical isomorphism $\cF \cong \bigoplus_j \gr^W_j \cF$. \label{it:wt-split}
\item An object $\cF \in \dsw(\bA^m)$ is miscible if and only if each $\pH^i(\cF)$ is semisimple. \label{it:misc-semis}
\end{enumerate}
\end{lem}
\begin{proof}
Part~\eqref{it:aff-cohom} is an immediate consequence of~\eqref{eqn:dsw-ses} and the well-known fact that $\Hom^i_{\ds(\bA^m)}(\uQlb, \uQlb) = 0$ for $i > 0$ (see~\cite[Corollary~VI.4.20]{milne}).

By induction on the length of a perverse sheaf, it follows from part~\eqref{it:aff-cohom} that for any two perverse sheaves $\cF, \cG \in \Pervsw(\bA^m)$, we have $\Hom^i(\cF,\cG) = 0$ for $i \ge 2$.  Part~\eqref{it:aff-hered} then follows by a standard argument.

For part~\eqref{it:wt-split}, let $w$ be the largest weight of any simple subquotient of $\cF$.  From the weight filtration of $\cF$, we can form a short exact sequence $0 \to \cF' \to \cF \to \cF'' \to 0$ where $\cF''$ is pure of weight $w$ and $\cF'$ has weights${}< w$.  By induction, it suffices to show that this sequence has a canonical splitting.  To show that, we must check that $\Hom(\cF'',\cF') = \Ext^1(\cF'',\cF') = 0$.  The fact that $\Hom(\cF'',\cF') = 0$ is obvious from considering the weight filtration.  From~\eqref{eqn:dsw-ses}, this implies that $\uHom(\cF'',\cF')^\Fr = 0$.  In other words, $1$ is not an eigenvalue of the action of $\Fr$ on $\uHom(\cF'', \cF')$.  It follows that $\uHom(\cF'',\cF')_\Fr = 0$ as well, so using~\eqref{eqn:dsw-ses} again, we obtain an isomorphism $\Hom^1(\cF'',\cF') \simto \uHom^1(\cF'', \cF')^\Fr$.  But $\uHom^1(\cF'', \cF') = 0$.

Finally, part~\eqref{it:misc-semis} follows from parts~\eqref{it:aff-hered} and~\eqref{it:wt-split} and the fact that a pure perverse sheaf is miscible if and only if it is semisimple.
\end{proof}

A slight modification of the notion of purity will also be useful to us.  Let us call an object $\cF \in \dsw(X)$ \emph{baric-pure} of weight $w$ if each $\pH^i(\cF)$ is pure of weight $w$.  This notion has been studied by S.~Morel~\cite{mor}; the terminology comes from~\cite{at}.  There is an analogue of Lemma~\ref{lem:weight-trunc} for baric purity; in fact, in the baric version, the triangle~\eqref{eqn:weight-trunc} is functorial (see~\cite[\S4.1]{mor} or \cite[\S 2.1]{at}).   Note that unlike ordinary purity, baric purity is stable under translation.  It follows from Lemma~\ref{lem:affine-basic} that for any $\cF \in \dsw(\bA^m)$, there is an isomorphism
\[
\cF \cong \bigoplus_j \cF^j
\]
where each $\cF^j$ is baric-pure of weight $j$.  

\begin{lem}\label{lem:aff-baric-split}
\begin{enumerate}
\item If $\cF, \cG \in \dsw(\bA^m)$ are baric-pure with distinct weights, then $\Hom(\cF,\cG) = 0$.
\item If $\cF, \cG \in \dsw(\bA^m)$ are both baric-pure of weight $j$, then the cone of any morphism is also baric-pure of weight $j$.
\end{enumerate}
\end{lem}
\begin{proof}
The first part follows from Lemma~\ref{lem:affine-basic}\eqref{it:aff-cohom}, and the second part is immediate from consideration of the long exact sequence of perverse cohomology sheaves associated to a distinguished triangle.
\end{proof}

\begin{lem}\label{lem:dotmap}
Let $\cF, \cG \in \Pervsm(\bA^m)$ be pure of weight $j$.  There is a canonical isomorphism $\phi: \Hom_{\dsw(\bA^m)}(\cF,\cG) \simto \Hom_{\dsw(\bA^m)}(\cF,\cG[1])$.  Moreover, we have $\phi(f \circ g) = f[1] \circ \phi(g) = \phi(f) \circ g$.
\end{lem}
This statement actually holds for any variety $X$ with an affable stratification, as can be seen from the proof.
\begin{proof}
For two objects $\cF, \cG \in \Pervsm(X)$ that are pure and have the same weight, we clearly have $\Hom_{\dsm(X)}(\cF,\cG[-1]) = 0$ and $\Hom_{\dsm(X)}(\cF,\cG[1]) = 0$.  The existence of $\phi$ is an immediate consequence of~\eqref{eqn:aff-even-dsw}.  Note that $\phi$ takes genuine morphisms $\cF \to \cG$ to infinitesimal morphisms $\cF \to \cG[1]$.  The composition formulas are then simply instances of~\eqref{eqn:pseudotri-comp}.
\end{proof}

For the next two lemmas, we will denote the isomorphism of Lemma~\ref{lem:dotmap} by
\[
r \in \Hom_{\dsw(\bA^m)}(\cF,\cG) \quad\mapsto\quad \dot r \in \Hom_{\dsw(\bA^m)}(\cF,\cG[1]).
\]

\begin{lem}\label{lem:aff-ext1}
Suppose $\cF, \cG \in \Pervsm(\bA^m)$ are pure objects of weight $0$.  For any morphism $r: \cF \to \cG$, the cone of $\dot r[-1]: \cF[-1] \to \cG$ is isomorphic to the object
\[
\cK = \cG \oplus \cF
\qquad\text{with $\Fr$ acting by} \qquad
\begin{bmatrix}
1 & r \\
 & 1
\end{bmatrix}.
\]
\end{lem}
\begin{proof}
This statement clearly holds when $r = 0$.  On the other hand, in the special case where $\cF \cong \cG \cong \uQlb$ and $r: \cF \to \cG$ is any nonzero map, then $\dot r \in \Ext^1(\uQlb,\uQlb)$ corresponds to a short exact sequence in $\Pervsw(\bA^m)$ whose middle term is an indecomposable pure rank-$2$ perverse sheaf on $\bA^m$.  Such an object has the form described above by~\cite[Proposition~5.3.9(i)]{bbd}.  Finally, in the general case, note that $\cF$ and $\cG$ are both direct sums of copies of $\uQlb$.  One can always choose direct-sum decompositions of these objects so that $r: \cF \to \cG$ arises as a direct sum of some number of zero maps and some number of isomorphisms $\uQlb \simto \uQlb$.  Thus, the general case follows from the special cases considered above.
\end{proof}

\begin{lem}\label{lem:aff-cone}
Suppose $\cF, \cG \in \dsw(\bA^m)$ are miscible baric-pure objects of weight $0$.  Any morphism $f: \cF \to \cG$ can be written as a sum
\begin{equation}\label{eqn:aff-cone}
f = \sum_i (p^i[-i] + \dot r^i[-i])
\end{equation}
involving morphisms $p^i: \pH^i(\cF) \to \pH^i(\cG)$ and $\dot r^i: \pH^i(\cF) \to \pH^{i-1}(\cG)[1]$.  Let $\cK$ denote the cone of $f$.  Its cohomology sheaves are described by
\begin{equation}\label{eqn:aff-cone-formula}
\pH^i(\cK) \cong \cok p^i \oplus \ker p^{i+1}
\qquad\text{with $\Fr$ acting by} \qquad
\begin{bmatrix}
1 & \bar r^{i+1} \\
 & 1
 \end{bmatrix},
\end{equation}
where $\bar r^{i+1}$ denotes the composition $\xymatrix{\ker p^{i+1} \ar[r]^{r^{i+1}} & \pH^i(\cG) \ar[r] & \cok p^i}$.  In particular, $\cK$ is miscible if and only if $r^{i+1}(\ker p^{i+1}) \subset \im p^{i}$ for all $i$.
\end{lem}
\begin{proof}
The fact that $f$ can be written as a sum~\eqref{eqn:aff-cone} follows from Lemma~\ref{lem:affine-basic}.  Note that the map $\pH^i(\cF) \to \pH^i(\cG)$ induced by $f$ is none other than $p^i$.  In particular, $f$ is an isomorphism if and only if each $p^i$ is an isomorphism.  Thus, in the special case where $f$ is an isomorphism, its cone $\cK = 0$ is indeed described by~\eqref{eqn:aff-cone-formula}.

Suppose henceforth that $\cK \ne 0$, and let $k$ be the smallest integer such that $\pH^k(\cK) \ne 0$.  Fix an isomorphism $\cF \cong \bigoplus \pH^i(\cF)[-i]$, and define two new objects as follows:
\begin{align*}
\cF' &= \bigoplus_{i \le k} \pH^i(\cF)[-i] \oplus (\ker p^{k+1})[-k-1], \\
\cF'' &= \bigoplus_{i > k+1} \pH^i(\cF)[-i] \oplus (\im p^{k+1})[-k-1].
\end{align*}
For each $i$, there is an obvious short exact sequence
\[
0 \to \pH^i(\cF') \to \pH^i(\cF) \to \pH^i(\cF'') \to 0.
\]
By putting in appropriate shifts and taking the direct sum over all $i$, we obtain a split distinguished triangle
\begin{equation}\label{eqn:aff-cone-ind}
\cF' \to \cF \to \cF'' \to.
\end{equation}
Note that $\cF'$ and $\cF''$ are both miscible by construction.  Next, we define maps $f': \cF' \to \tau_{\le k}\cG$ and $f'': \cF'' \to \tau_{> k}\cG$ by
\[
f' = \sum_{i\le k} (p^i[-i] + \dot r^i[-i]) + \dot r^{k+1},
\qquad
f'' = \sum_{i > k+1} (p^i[-i] + \dot r^i[-i]) + \bar p^{k+1},
\]
where $\bar p^{k+1}: \im p^{k+1} \to \pH^{k+1}(\cG)$ is the inclusion map.
These definitions make the two leftmost squares below commute.
\[
\xymatrix{
\cF' \ar[r] \ar[d]_{f'} & \cF\ar[r] \ar[d]_f & \cF'' \ar[r]\ar[d]_{f''} & \cF'[1] \ar[d]^{f'[1]}\\
\tau_{\le k}\cG \ar[r] & \cG \ar[r] & \tau_{> k}\cG \ar[r] & (\tau_{\le k}\cG)[1] }
\]
The rightmost square commutes as well: the map $\tau_{> k}\cG \to (\tau_{\le k}\cG)[1]$ vanishes by Lemma~\ref{lem:affine-basic}\eqref{it:aff-hered}, and~\eqref{eqn:aff-cone-ind} splits by construction.  Thus, this is a morphism of distinguished triangles.  By the $9$-lemma~\cite[Proposition~1.1.11]{bbd}, we can extend this to a diagram in which all rows and columns are distinguished triangles:
\begin{equation}\label{eqn:aff-cone-9}
\vcenter{\xymatrix{
\cF' \ar[r] \ar[d]_{f'} & \cF\ar[r] \ar[d]_f & \cF'' \ar[r]\ar[d]_{f''} & \\
\tau_{\le k}\cG \ar[r] \ar[d] & \cG \ar[r]\ar[d] & \tau_{> k}\cG \ar[r]\ar[d] & \\
\cK' \ar[r] \ar[d] & \cK\ar[r] \ar[d] & \cK'' \ar[r]\ar[d] & \\
&&&}}
\end{equation}
From the known cohomology vanishing conditions on the first two rows, it is obvious that $\pH^i(\cK') = 0$ for $i > k$, and that $\pH^i(\cK'') = 0$ for $i < k$.  In fact, we also have $\pH^k(\cK'') = 0$, since the map
\begin{equation}\label{eqn:aff-cone-pkinj}
\bar p^{k+1}: \pH^{k+1}(\cF'') \to \pH^{k+1}(\tau_{>k}\cG) \cong \pH^{k+1}(\cG)
\end{equation}
is injective.  Therefore, we have canonical isomorphisms
\[
\cK' \cong \tau_{\le k}\cK
\qquad\text{and}\qquad
\cK'' \cong \tau_{> k}\cK.
\]
If we already knew that the cohomology sheaves $\pH^i(\cK')$ and $\pH^i(\cK'')$ could be described by~\eqref{eqn:aff-cone-formula} in terms of $f'$ and $f''$, then the result would follow for $\cK$.  Note that $\cK'$ and $\cK''$ each have fewer nonzero cohomology sheaves than $\cK$.  Therefore, by induction, it suffices to prove~\eqref{eqn:aff-cone-formula} in the special case where $\cK$ has nonzero cohomology in a single degree.  We may further assume, without loss of generality, that $\cK$ is in fact concentrated in degree $0$.

With this assumption in place, we may still construct the diagram~\eqref{eqn:aff-cone-9}, taking $k = 0$.  In this case, we have $\tau_{>0}\cK = 0$, so $f''$ is an isomorphism.  That diagram tells us that to prove the result for the middle column, it suffices to prove it for the first column.  In other words, by replacing $f: \cF \to \cG$ by $f': \cF' \to \tau_{\le 0}\cG$, we may henceforth assume that
\[
\text{$\pH^i(\cF) = 0$ for $i > 1$, \qquad and \qquad $\pH^i(\cG) = 0$ for $i > 0$.}
\]
Under these conditions, let us form yet another copy of~\eqref{eqn:aff-cone-9}, this time with $k = -1$.  Now we have $\tau_{\le -1}\cK = 0$, so to prove the result, it suffices to consider the third column.  Making another replacement, we have reduced the problem to following situation:
\[
\cF \cong \pH^0(\cF) \oplus \pH^1(\cF)[-1],
\qquad
\cG \cong \pH^0(\cG),
\qquad
f = p^0 + \dot r^1.
\]
For brevity, let us put $\cF^i = \pH^i(\cF)$ for $i = 0,1$.  In particular, we have $p^1 = 0$, so $\cF^1 = \ker p^1$.  We further know that $p^0 = \bar p^0: \cF^0 \to \cG$ is injective, cf.~\eqref{eqn:aff-cone-pkinj}.  Form the octahedral diagram associated with the composition $\cF^0 \to \cF^0 \oplus \cF^1[-1] \ovto{f} \cG$:
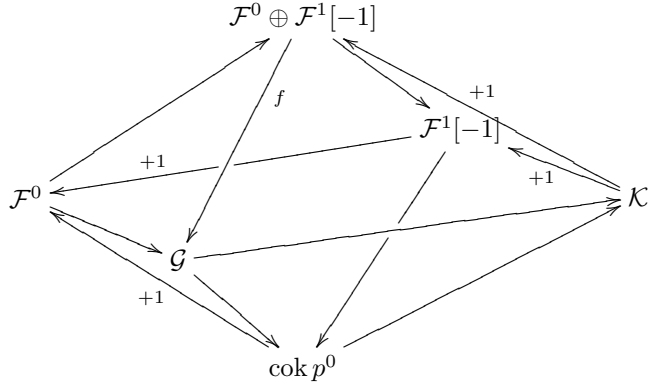
\[
\xymatrix@=10pt{
&&&& \cF^0 \oplus \cF^1[-1] \ar[ddddl]^(.3)f \ar[ddr] \\ \\
&&&&& \cF^1[-1] \ar[dlllll]|!{[uul];[ddll]}\hole_(.7){+1}
  \ar[ddddl]|!{[ddll];[drrr]}\hole \\
\cF^0 \ar[uuurrrr]\ar[drrr] &&&&&&&& \cK \ar[uuullll]_{+1}
\ar[ulll]^{+1} \\
&&& \cG \ar[urrrrr]\ar[ddr] \\ \\
&&&& \cok p^0 \ar[uuullll]^{+1}\ar[uuurrrr] }
\]
Applying Lemma~\ref{lem:aff-ext1} to the distinguished triangle $\cF^1[-1] \to \cok p^0 \to \cK \to$, we obtain the desired result.
\end{proof}

\begin{lem}\label{lem:misc-morphism}
Suppose $\cF, \cG \in \dsw(\bA^m)$ are miscible baric-pure objects of weight $0$.  If $f: \cF \to \cG$ is a morphism whose cone is miscible, then $f$ is miscible.
\end{lem}
\begin{proof}
Let $p^i$ and $r^i$ be as in~\eqref{eqn:aff-cone}, and let $f' = \sum_i p^i[-i]$.  We will construct a commutative diagram
\[
\xymatrix{
\cF \ar[r]^f \ar[d]_\phi^\wr &
  \cG \ar[d]^\psi_\wr \\
\cF \ar[r]^{f'} &
  \cG}
\]
where $\phi$ and $\psi$ are isomorphisms.  Since $f'$ is obviously miscible, $f$ will be as well.  By Lemma~\ref{lem:aff-cone}, we know that $r^i(\ker p^i) \subset \im p^{i-1}$ for all $i$.  For each $i$, choose a complement $U^i \subset \pH^i(\cF)$ to $\ker p^i$.  In other words, we have $\pH^i(\cF) \cong \ker p^i \oplus U^i$.  Then, let $u^i: \pH^i(\cF) \to \pH^{i-1}(\cF)$ be the map such that
\[
u^i(U^i) = 0,
\qquad
u^i(\ker p^i) \subset U^{i-1},
\qquad
(p^{i-1} \circ u^i)|_{\ker p^i} = r^i|_{\ker p^i}.
\]
Since $p^{i-1}$ induces an isomorphism $U^{i-1} \cong \im p^{i-1}$ and $r^i(\ker p^i) \subset \im p^{i-1}$, there is a unique map $u^i$ satisfying the conditions above.  Next, let $v^i: \pH^i(\cG) \to \pH^{i-1}(\cG)$ be a map such that
\[
(v^i \circ p^i)|_{U^i} = r^i|_{U^i}.
\]
Such a map certainly exists since $p^i$ induces an isomorphism $U^i \cong \im p^i$, although it is not uniquely determined.  Note that we have an equality
\[
r^i = p^{i-1} \circ u^i + v^i \circ p^i: \pH^i(\cF) \to \pH^{i-1}(\cG)
\]
and therefore, using the formulas in Lemma~\ref{lem:dotmap}, we have
\[
\dot r^i = p^{i-1}[1] \circ \dot u^i + \dot v^i \circ p^i: \pH^i(\cF) \to \pH^{i-1}(\cG)[1].
\]
Define $\phi: \cF \to \cF$  and $\psi: \cG \to \cG$ by
\[
\phi = \sum_i \id_{\pH^i(\cF)}[-i] + \dot u^i[-i],
\qquad
\psi = \sum_i \id_{\pH^i(\cG)}[-i] - \dot v^i[-i].
\]
It is now easy to see that
\[
f' \circ \phi = \sum_i (p^i + p^{i-1} \circ \dot u^i)[-i]
= \sum_i (p^i + \dot r^i - \dot v^i \circ p^i)[-i] = \psi \circ f.
\]
Finally, $\dot u^i$ and $\dot v^i$ are infinitesimal, so by Lemma~\ref{lem:infext-basic}, $\phi$ and $\psi$ are themselves isomorphisms, as desired.
\end{proof}

\section{Miscibility and genuineness results}
\label{sect:genuine}

We have now defined a number of properties that a functor between categories of constructible complexes on varieties over $\F_q$ may have.  In order from strongest to weakest, they are:
\[
\text{genuine} \quad\Longrightarrow\quad
\text{miscible} \quad\Longrightarrow\quad
\text{geometric} \quad\Longrightarrow\quad
\text{preserves the Weil category.}
\]
This section contains the main results of the paper, which state that various functors (including smooth pull-backs, open and closed inclusions, and smooth proper push-forwards) are genuine.  For a few more functors (including tensor products and arbitrary proper push-forwards), we prove miscibility.  Under an additional hypothesis, we will show that arbitrary proper push-forwards are genuine in Section~\ref{sect:tilting}.

We will generally not comment on the property of preserving the Weil category in the proofs below, as this has already been checked for most functors in Section~\ref{sect:mixedweil} (two exceptions occur in Proposition~\ref{prop:boxtimes} and Corollary~\ref{cor:proper-strat-weil}).  We will likewise remain silent about the property of being geometric, since this is essentially automatic for the usual sheaf operations.

Along the way, we also prove that Weil and miscible distinguished triangles coincide, as promised in Section~\ref{ss:aff-infext}, and we establish a pointwise criterion for semisimplicity (Proposition~\ref{prop:sterile-miscible}) that may be useful in other contexts as well.

\subsection{Basic results on genuineness}

The following proposition is the main tool we will use to apply results from Section~\ref{sect:hot-orlov} in the sheaf-theoretic setting.

\begin{prop}\label{prop:homog-genuine}
Let $\cS$ be an affable stratification of $X$, and $\cT$ an affable stratification of $Y$.  Let $F: \dsw(X) \to \dw_\cT(Y)$ be a geometric functor, and let $\scA \subset \dsmisc(X)$ and $\scB \subset \dmisc_\cT(Y)$ each be one of the Orlov categories of Propositions~\ref{prop:aff-even-orlov} or~\ref{prop:affable-orlov}.  
\begin{enumerate}
\item If $F(\scA) \subset \scB$, then $F$ is a miscible functor.  Furthermore, $F$ takes miscible morphisms to miscible morphisms.\label{it:homog-misc}
\item If, in addition, the restriction $F|_{\scA}: \scA \to \scB$ is homogeneous, then $F$ is genuine.\label{it:homog-gen}
\end{enumerate}
\end{prop}
\begin{proof}
By Lemma~\ref{lem:misc-pt-func}, the first assertion implies that $F|_{\dsmisc(X)}: \dsmisc(X) \to \dmisc_\cT(Y)$ is a pseudotriangulated functor, so the second assertion follows from it by Theorem~\ref{thm:orlov-genuine}, using Theorem~\ref{thm:misc-infext} and Proposition~\ref{prop:aff-even-orlov}.  Thus, it suffices to prove the first assertion. We will prove simultaneously that $F$ takes miscible objects to miscible objects and miscible morphisms to miscible morphisms.  Let $f: \cF \to \cG$ be a morphism in $\dsm(X) \cong \Kb(\scA)$.  Write these objects as chain complexes: $\cF = (\cF^\bullet, d)$ and $\cG = (\cG^\bullet, d)$.  Let $I = \{j, j+1, \ldots, k\} \subset \Z$ be the smallest interval in $\Z$ such that $\cF^i = \cG^i = 0$ for $i \notin I$.  We proceed by induction on the size of $I$.  With $k$ denoting the largest element of $I$, then there is an obvious distinguished triangle
\[
\cF^k[-k] \to \cF \to \cF' \to
\]
in $\Kb(\scA)$, where $\cF'$ is the complex obtained from $\cF$ by replacing its $k$th term by $0$.  We can form the analogous triangle $\cG^k[-k] \to \cG \to \cG' \to$ for $\cG$.  It is clear that $\Hom(\cF^k[-k],\cG') = 0$, so the composition $\cF^k[-k] \to \cF \ovto{f} \cG$ factors through $\cG^k[-k]$, and we obtain a morphism of triangles
\[
\xymatrix{
\cF^k[-k] \ar[d]_{f^k[-k]} \ar[r] &
  \cF \ar[d]_f \ar[r] & \cF' \ar[d]_{f'} \ar[r] & \cF^k[-k+1] \ar[d]^{f^k[-k+1]} \\
\cG^k[-k] \ar[r] & \cG \ar[r] & \cG' \ar[r] & \cG^k[-k+1]}
\]
for some morphisms $f^k, f'$.  All the objects in the rightmost commutative square are chain complexes whose nonzero terms appear only in degrees $i \in I \smallsetminus \{k\}$.  Therefore, by induction, all objects and morphisms in the square
\[
\xymatrix{
F(\incl \cF') \ar[d]\ar[r] & F(\incl \cF^k[-k+1]) \ar[d] \\
F(\incl \cG') \ar[r] & F(\incl \cG^k[k+1]) }
\]
are miscible.  Since the morphism $F(\incl f): F(\incl \cF) \to F(\incl \cG)$ arises by completing this square to a morphism of distinguished triangles, it follows that $F(\incl\cF)$, $F(\incl\cG)$, and $F(\incl f)$ are all miscible.
\end{proof}

The next three results are straightfoward applications of the preceding proposition.

\begin{prop}\label{prop:smooth-pullback}
Suppose $X$ and $Y$ are endowed with affable stratifications $\cS$ and $\cT$, and let $f: X \to Y$ be a weakly stratified morphism.  If $f$ is smooth, then the functors $f^*, f^!: \dw_\cT(Y) \to \dsw(X)$ are genuine.
\end{prop}
\begin{proof}
In place of $f^*$ and $f^!$, we will instead consider the functor $f^\sharp = f^*[d]\la -d\ra \cong f^![-d]\la d\ra$, where $d$ denotes the relative dimension of $f$.  This functor is $t$-exact and takes simple perverse sheaves to pure semisimple perverse sheaves of the same weight.  More specifically, if $h: f^{-1}(Y_t) \to X$ is the inclusion map, then $f^\sharp\ICm_t \cong h_{!*}\uQlb[n] \la -n\ra$, where $n = \dim f^{-1}(Y_t)$.  The constant sheaf $\uQlb[n]\la-n\ra$ on the smooth variety $f^{-1}(Y_t)$ is the direct sum of simple perverse sheaves of weight $0$ on the various connected components of $f^{-1}(Y_t)$, so $h_{!*}\uQlb[n]\la-n\ra$ is also a direct sum of simple perverse sheaves of weight $0$.  Thus, $f^\sharp$ gives rise to a homogeneous functor $\Pure(Y) \to \Pure(X)$, so it is genuine by Proposition~\ref{prop:homog-genuine}.
\end{proof}

\begin{prop}\label{prop:verdier}
Let $X$ be a variety with an affable stratification.  The Verdier duality functor $\D : \dsw(X)^\op \to \dsw(X)$ is genuine.
\end{prop}
\begin{proof}
The category $\Pure(X)^\op \subset \dsw(X)^\op$ is clearly an Orlov category with degree function given by
\[
\deg_{\Pure(X)^\op} (\ICm_s[n]\la -n\ra) = n = -\deg_{\Pure(X)} (\ICm_s[n]\la -n\ra).
\]
We clearly have $\D(\Pure(X)^\op) \subset \Pure(X)$, and $\D: \Pure(X)^\op \to \Pure(X)$ is a homogeneous functor.  The result follows by an analogue of Proposition~\ref{prop:homog-genuine}.
\end{proof}

\begin{prop}\label{prop:boxtimes}
Let $X$ and $Y$ be two varieties equipped with affable stratifications.  Then, for $\cF \in \dsw(X)$ and $\cG \in \dw_\cT(Y)$, we have $\cF \boxtimes \cG \in \dw_{\cS \times \cT}(X \times Y)$.  Moreover, the induced stratification $\cS \times \cT$ on $X \times Y$ is affable, and the functor $\boxtimes: \dsw(X) \times \dw_\cT(Y) \to \dw_{\cS \times \cT}(X \times Y)$ is genuine.
\end{prop}
See Section~\ref{ss:bifunc} for remarks on homogeneity and genuineness for bifunctors.
\begin{proof}
It suffices to treat the case where $\cS$ and $\cT$ are both affine even stratifications, so we henceforth restrict to that case.  Let $j_{s,t}: X_s \times Y_t \to X \times Y$ denote the inclusion of a stratum.  Recall that
\begin{equation}\label{eqn:boxtimes-stalk}
j_{s,t}^*(\cF \boxtimes \cG) \cong j_s^*\cF \boxtimes j_t^*\cG
\qquad\text{and}\qquad
j_{s,t}^!(\cF \boxtimes \cG) \cong j_t^!\cF \boxtimes j_t^!\cG.
\end{equation}
Using these facts, it follows by elementary dimension calculations (cf.~\cite[Proposition~4.2.8]{bbd}) that $\boxtimes$ takes perverse sheaves to perverse sheaves.  In fact, the same calculations also show that
\begin{equation}\label{eqn:boxtimes-ic}
\ICm_s \boxtimes \ICm_t \cong \ICm_{s,t}.
\end{equation}
Now, on the variety $\bA^n \times \bA^m$, we clearly have $\uQlb_{\bA^n} \boxtimes \uQlb_{\bA^m} \cong \uQlb_{\bA^n \times \bA^m}$.  Using this observation together with~\eqref{eqn:boxtimes-stalk} and~\eqref{eqn:boxtimes-ic} to compute $j_{u,v}^*\ICm_{s,t}$ and $j_{u,v}^!\ICm_{s,t}$, we see that the second condition of Definition~\ref{defn:aff-even} holds.  Thus, $\cS \times \cT$ is an affine even stratification.  Now,~\eqref{eqn:boxtimes-ic} shows that $\boxtimes$ takes pure semisimple objects in either variable to pure semisimple objects, so it clearly takes values in the Weil category when applied to objects in the Weil category.  Lastly, it is genuine by Proposition~\ref{prop:homog-genuine}.
\end{proof}

\subsection{Open and closed inclusions}
\label{ss:open-closed}

The following theorem tells us in part that mixed categories satisfy the axioms in~\cite[\S 1.4.3]{bbd} for the formalism of ``gluing,'' so all subsequent results of~\cite[\S 1.4]{bbd} apply in this setting.

\begin{thm}\label{thm:open-closed}
Let $i: Z \to X$ be the inclusion of a closed subvariety that is a union of strata, and let $j: U \to X$ be the inclusion of the complementary open subvariety.  The functors $i^*$, $i^!$, $i_*$, $j^*$, $j_*$, and $j_!$ are all genuine.  Moreover, the induced functors on the mixed categories enjoy the following properties:
\begin{enumerate}
\item The usual adjointness properties hold.
\item For $\cF \in \dsm(X)$, there are functorial distinguished triangles
\begin{equation}\label{eqn:open-closed-dt}
i_*i^!\cF \to \cF \to j_*j^*\cF \to
\qquad\text{and}\qquad
j_!j^*\cF \to \cF \to i_*i^*\cF \to
\end{equation}
in $\dsm(X)$.
\item The functors $i_*: \dsm(Z) \to \dsm(X)$ and $j_*, j_!: \dsm(U) \to \dsm(X)$ are fully faithful.
\end{enumerate}
\end{thm}

\begin{rmk}\label{rmk:open-quot}
As an example of a statement that follows purely from the formalism of gluing, we have by~\cite[\S1.4.6(b)]{bbd} that $j^*$ induces an equivalence
\[
\dsm(X)/\dsm(Z) \simto \dsm(U).
\]
Here, we have identified $\dsm(Z)$ with a full triangulated subcategory of $\dsm(X)$ using the fully faithful functor $i_*$.  This observation will be used in the proof of Theorem~\ref{thm:ajext}.
\end{rmk}

Before proving this theorem, we recall a result about the structure of projectives in $\Pervs(X)$ from~\cite{bgs}.  Let us put
\[
\Dm_s = j_{s!}\uQlb[\dim X_s]\la -\dim X_s\ra
\qquad\text{and}\qquad
\Nm_s = j_{s*}\uQlb[\dim X_s]\la -\dim X_s\ra.
\]
These objects are perverse sheaves by~\cite[Corollaire~4.1.3]{bbd}, {\it a priori} only in $\Pervsw(X)$, although it clearly follows from Theorem~\ref{thm:open-closed} that they lie {\it a posteriori} in $\Pervsm(X)$.  $\Dm_s$ is called a \emph{standard perverse sheaf}, and $\Nm_s$ is called a \emph{costandard perverse sheaf}.  The same terms are used for the objects
\[
\Delta_s = \degr(\Dm_s)
\qquad\text{and}\qquad
\nabla_s = \degr(\Nm_s)
\]
in $\Pervs(X)$.  According to~\cite[Theorem~3.3.1]{bgs}, every projective in $\Pervs(X)$ has a filtration with standard subquotients, and every injective has a filtration with costandard subquotients.  To be more specific, it follows from ``BGG reciprocity'' (see \cite[Remark~(1) following Theorem~3.2.1]{bgs}) that the standard objects $\Delta_t$ occurring as subquotients of the projective cover of $\IC_s$ all have the property that $X_s \subset \overline{X_t}$.

\begin{proof}[Proof of Theorem~\ref{thm:open-closed}]
By Lemma~\ref{lem:affable-genuine}, it suffices to treat the case where $\cS$ is an affine even stratification, and we henceforth assume this to be the case.  The proof is somewhat lengthy and proceeds in several steps.

{\it Step 1. $i_*$ and $j^*$.}  These two functors send semisimple pure objects to semisimple pure objects.  That is, they induce functors $\Pure(Z) \to \Pure(X)$ and $\Pure(X) \to \Pure(U)$.  Moreover, the latter functors are homogeneous functors of Orlov categories because $i_*$ and $j^*$ are $t$-exact.  By Proposition~\ref{prop:homog-genuine}, these functors are genuine.

{\it Step 2. $j_!$.}  Our strategy is to show that this functor induces a homogeneous functor $\Proj(U) \to \Proj(X)$.  Specifically, consider a stratum $X_s \subset U$.  Let $P_s$ denote the projective cover of $\ICm_s$ in $\Pervsm(X)$, and let $P'_s$ denote the projective cover of the simple object $\ICm_s|_U$ in $\Pervsm(U)$.  It suffices to show that
\begin{equation}\label{eqn:open-proj-adj}
j_!P'_s \cong P_s.
\end{equation}
We begin by showing that $j^*P_s \cong P'_s$.  Note first that we at least have $j^*P_s \in \Pervsm(U)$, since $j^*$ is already known to be genuine.  To prove that $j^*P_s$ is projective, it suffices, by Theorem~\ref{thm:bgs-aff-even}\eqref{it:bgs-mix-proj}, to show that $\degr(j^*P_s) \cong j^*(\degr(P_s))$ is a projective object in $\Pervs(U)$.  Making use of the equivalence in Theorem~\ref{thm:bgs-aff-even}\eqref{it:bgs-equiv} and the fact that $\degr(P_s)$ is projective, we have
\[
\Ext^1(j^*\degr(P_s),\cG) \cong \Hom(\degr(P_s), j_*\cG[1]) \cong \Hom(\degr(P_s), \pH^1(j_*\cG))
\]
for any $\cG \in \Pervs(U)$.  Since $\pH^1(j_*\cG)$ is supported on $Z$, it cannot contain $\IC_s$ as a composition factor, so the last $\Hom$-group above vanishes.  We conclude that $j^*\degr(P_s)$ and $j^*P_s$ are projective.  For similar reasons, we have
\[
\Hom(\degr(P_s), i_*i^!\degr(P_s)) = \Ext^1(\degr(P_s), i_*i^!\degr(P_s)) = 0,
\]
so we deduce from the distinguished triangle $i_*i^!\degr(P_s) \to \degr(P_s) \to j_*j^*\degr(P_s) \to$ that there are isomorphisms
\[
\Hom(\degr(P_s), \degr(P_s)) \cong \Hom(\degr(P_s), j_*j^*\degr(P_s)) \cong \Hom(j^*\degr(P_s), j^*\degr(P_s)) \cong \Bbbk.
\]
In particular, we have that $j^*\degr(P_s)$ is indecomposable, so $j^*P_s$ is as well.  Since $j^*P_s$ is an indecomposable projective with a nonzero map $j^*P_s \to \ICm_s|_U$, we must have $j^*P_s \cong P'_s$, as desired.

Consider now the distinguished triangle $\degr(j_!P'_s) \to \degr(P_s) \to \degr(i_*i^*P_s) \to$ in $\ds(X)$.  Recall that $\degr(P_s)$ has a standard filtration consisting of $\Delta_t$ with $X_s \subset \overline{X_t}$.  All such $X_t$ are contained in $U$, so $i^*\Delta_t = 0$, and therefore $\degr(i^*P_s) = 0$.  Since $\degr$ kills no nonzero object, we conclude that $i^*P_s = 0$ as well.  Thus, the natural map $j_!P'_s \to P_s$ is an isomorphism.

{\it Step 3. $i^*$.}  The strategy is similar to that in Step~2.  For any $s \in \cS$, let $P_s$ denote the projective cover of $\ICm_s$ in $\Pervsm(X)$.  For any standard object $\Delta_t \in \Pervs(X)$, we have that $i^*\Delta_t$ is either $0$ or a standard object in $\Pervs(Z)$.  In either case, it is a perverse sheaf.  Since $\degr(P_s)$ has a standard filtration, $\degr(i^*P_s)$ is a perverse sheaf as well, and hence so is $i^*P_s$.  From the distinguished triangle
\begin{equation}\label{eqn:closed-adj-dt}
j_!j^*P_s \to P_s \to i_*i^*P_s \to
\end{equation}
and the right $t$-exactness of $j_!$, we see that $i_*i^*P_s$ is a quotient of $P_s$.  Therefore, like any quotient of an indecomposable projective, it is either indecomposable or $0$. Indeed, we saw in Step~2 that it is $0$ if $X_s \subset U$.  On the other hand, if $X_s \subset Z$, there is a nonzero morphism $i^*P_s \to \ICm_s$.  Since 
\[
\Hom(i^*(\degr(P_s)), \cG[1]) \cong \Hom(\degr(P_s), i_*\cG[1]) = 0
\]
for all $\cG \in \Pervs(Z)$, we see that $\degr(i^*P_s)$ is projective, and therefore so is $i^*P_s$.  We have shown that
\[
i^*P_s  \cong
\begin{cases}
0 & \text{if $X_s \subset U$,} \\
P''_s & \text{if $X_s \subset Z$,}
\end{cases}
\]
where $P''_s$ is the projective cover of $\ICm_s$ in $\Pervsm(Z)$.  In particular, $i^*$ induces a homogeneous functor $\Proj(X) \to \Proj(Z)$, and is therefore genuine.  

{\it Step 4. $j_*$ and $i^!$.}  These follow from Steps~2 and~3 and Proposition~\ref{prop:verdier} by the formulas $j_* \cong \D \circ j_! \circ \D$ and $i^! \cong \D \circ i^* \cong \D$.

{\it Step 5. Adjointness properties.}  The fact that the induced functors on the mixed categories have the usual adjointness properties follows from Lemma~\ref{lem:infext-ind-adjoint}.

{\it Step 6. Functorial distinguished triangles.}  By Step~5, for any $\cF \in \dsm(X)$, we have an adjunction morphism $\epsilon: i_*i^!\cF \to \cF$.  Let us complete this to a distinguished triangle
\begin{equation}\label{eqn:mix-opcl-dt}
i_*i^!\cF \ovto{\epsilon} \cF \ovto{q} \cK \ovto{p} i_*i^!\cF[1].
\end{equation}
After applying $\incl$, we obtain a distinguished triangle in $\dsw(X)$ that is canonically isomorphic to the functorial distinguished triangle
\begin{equation}\label{eqn:weil-opcl-dt}
i_*i^!(\incl\cF) \ovto{\epsilon} \incl\cF \ovto{\eta} j_*j^*(\incl\cF) \ovto{\delta} i_*i^!(\incl\cF)[1]
\end{equation}
In particular, we see that~\eqref{eqn:weil-opcl-dt}, which is a priori only a Weil distinguished triangle, is actually miscible.  Recall from Lemma~\ref{lem:infext-induc} that the functor $\tilde F$ induced by a pseudotriangulated functor $F$ is given by $\tilde F \cong \pg \circ F \circ \incl$.  Therefore, applying $\pg$ to~\eqref{eqn:weil-opcl-dt} gives us a functorial distinguished triangle in $\dsm(X)$ that is isomorphic to~\eqref{eqn:mix-opcl-dt}.  The argument for the second triangle in~\eqref{eqn:open-closed-dt} is similar.

{\it Step 7. Fullness and faithfulness.}  We first note that $i_*: \dsm(Z) \to \dsm(X)$ is faithful, because the original functor $i_*: \dsw(Z) \to \dsw(X)$ is.  In addition, for $\cF \in \dsm(Z)$, the adjunction map $i^*i_*\cF \simto \cF$ is an isomorphism because the same statement holds in $\dsw(Z)$, so for any $\cG \in \dsm(Z)$, we have
\[
\Hom_{\dsm(X)}(i_*\cF,i_*\cG) \cong \Hom_{\dsm(Z)}(i^*i_*\cF,\cG) \cong \Hom_{\dsm(Z)}(\cF,\cG).
\]
Thus, $i_*: \Hom_{\dsm(Z)}(\cF,\cG) \to \Hom_{\dsm(X)}(i_*\cF, i_*\cG)$ is an injective map between vector spaces of the same dimension, so it is an isomorphism.  The arguments for $j_*$ and $j_!$ are similar.
\end{proof}

\begin{cor}\label{cor:loc-closed}
If $h: Y \to X$ is the inclusion map of a locally closed subvariety that is a union of strata, then $h^*$, $h_*$, $h^!$, and $h_!$ are all genuine.
\end{cor}
\begin{proof}
The map $h$ can be factored as an open embedding followed by a closed embedding, and the pull-back and push-forward functors for each of those maps are genuine by Theorem~\ref{thm:open-closed}.
\end{proof}

We are now able to prove the following mixed analogue of~\cite[Theorem~3.3.1]{bgs}.

\begin{prop}\label{prop:proj-std-filt}
Assume $X$ has an affine even stratification.  Every projective object in $\Pervsm(X)$ has a filtration with standard subquotients, and every injective object has a filtration with costandard subquotients.
\end{prop}
\begin{proof}
We will prove the statement for projectives; the injective case is similar.  We proceed by induction on the number of strata in $X$.  Let $j: X_t \to X$ be the inclusion of an open stratum, and let $i: Z \to X$ be the inclusion of the complementary closed subvariety.  For a projective $P \in \Pervsm(X)$, recall from the proof of Theorem~\ref{thm:open-closed} that $i^*P$ is also a perverse sheaf.  Since $j^*$ is $t$-exact and $j_!$ is right $t$-exact, we have a short exact sequence
\[
0 \to j_!j^*P \to P \to i_*i^*P \to 0.
\]
Since $X_t$ is an affine space, $j^*P$ is semisimple, and $j_!j^*P$ is a direct sum of standard objects.  On the other hand, the projective object $i^*P \in \Pervsm(Z)$ has a standard filtration by induction.  Thus, $P$ has a standard filtration.
\end{proof}

\begin{rmk}
It can also be deduced using the methods of~\cite[Lemma~4.4.8]{bgs} that if $P \in \Pervsm(X)$ is projective, then the standard filtration of $\degr(P)$ lifts to some filtration of $P$ with subquotients $\cF_s$ satisfying $\degr(\cF_s) \cong \degr(\Dm_s)$.  But the stronger statement in Proposition~\ref{prop:proj-std-filt} requires knowing that $j_{s!}$ and $j_{s*}$ are miscible functors; it does not directly follow from the results of~\cite{bgs}, as far as we understand.
\end{rmk}

\begin{rmk}
Now that we know that the objects $\Dm_s$ belong to $\Pervsm(X)$, it is easy to check, by a further use of Theorem~\ref{thm:open-closed}, that $\Ext^k(\Dm_s, \ICm_t\la n\ra) = 0$ unless $n = -k$.  In other words, the $\Dm_s$ are ``Koszul objects'' of $\Pervsm(X)$ in the sense of~\cite[Definition~2.14.1]{bgs}, cf.~the remark following~\cite[Theorem~3.11.4]{bgs}.
\end{rmk}

\subsection{Weil and miscible distinguished triangles}

We can now supply a foundational fact about $\dsmisc(X)$ that was promised in Section~\ref{ss:aff-infext}.  The proof relies on Theorem~\ref{thm:open-closed}.

\begin{thm}\label{thm:misc-tri}
Assume that $X$ has an affable stratification.  A diagram
\[
\cF' \to \cF \to \cF'' \to \cF'[1]
\]
in $\dsmisc(X)$ is a miscible distinguished triangle if and only if it is a Weil distinguished triangle.
\end{thm}
\begin{rmk}\label{rmk:misc-pt}
We now see that all three conditions of Lemma~\ref{lem:misc-pt-func} hold for all miscible functors.  In particular, every miscible functor $F: \dsw(X) \to \dw_\cT(Y)$ gives rise to a pseudotriangulated functor $F|_{\dsmisc(X)}: \dsmisc(X) \to \dmisc_\cT(Y)$, as indicated in Table~\ref{tab:dict}.
\end{rmk}
\begin{proof}
Recall that every miscible distinguished triangle is a Weil distinguished triangle; we need only prove the opposite implication.  It suffices to treat the case where $\cS$ is an affine even stratification, and we henceforth restrict to this case.

We begin by proving the statement in the special case where $X = \bA^m$.  It follows from Lemmas~\ref{lem:affine-basic} and~\ref{lem:aff-baric-split} that any Weil distinguished triangle $\cF' \to \cF \to \cF'' \to$ in $\dsmisc(\bA^m)$ can be written as the direct sum over $j \in \Z$ of triangles
\[
(\cF')^j \to \cF^j \to (\cF'')^j \to
\]
in which all three terms are baric-pure of weight $j$.  In particular, to prove the proposition, it suffices to consider the case where $\cF'$, $\cF$, and $\cF''$ are all baric-pure of weight $j$.  Since all three objects are miscible, all three morphisms in the triangle are miscible by Lemma~\ref{lem:misc-morphism}, so the distinguished triangle is miscible, as desired.

For the case of a general variety $X$ with an affine even stratification, we proceed by induction on the number of strata in $X$.  Choose a closed stratum $j_t: X_t \to X$, and let $h: U \to X$ be the inclusion of the open complement to $X_t$.  From the given distinguished triangle, form the following commutative diagram:
\[
\xymatrix{
h_!h^*\cF' \ar[r]\ar[d] &
  h_!h^*\cF \ar[r]\ar[d] &
    h_!h^*\cF'' \ar[r]\ar[d] & \\
\cF' \ar[r]\ar[d] &
  \cF \ar[r]\ar[d] &
    \cF'' \ar[r]\ar[d] & \\
j_{t*}j_t^*\cF' \ar[r]\ar[d] &
  j_{t*}j_t^*\cF \ar[r]\ar[d] &
    j_{t*}j_t^*\cF'' \ar[r]\ar[d] & \\
    &&&}
\]
The columns of this diagram are miscible triangles by Theorem~\ref{thm:open-closed}.  Furthermore, the triangles $h^*\cF' \to h^*\cF \to h^*\cF'' \to $ and $j_t^*\cF' \to j_t^*\cF \to j_t^*\cF'' \to $ are miscible by induction, as $U$ and $X_t$ each consist of fewer strata than $X$.  Since $h_!$ and $j_{t*}$ are miscible functors, the top and bottom rows of this diagram are miscible triangles.  Thus, the given distinguished triangle is obtained by taking the ``cone'' of the miscible commutative diagram
\[
\xymatrix{
j_{t*}j_t^*\cF'[-1] \ar[r]\ar[d] &
  j_{t*}j_t^*\cF[-1] \ar[r]\ar[d] &
    j_{t*}j_t^*\cF''[-1] \ar[r]\ar[d] & \\
h_!h^*\cF' \ar[r] &
  h_!h^*\cF \ar[r] &
    h_!h^*\cF'' \ar[r] & }
\]
and is therefore miscible itself, as desired.
\end{proof}

\subsection{Proper stratified morphisms}
\label{ss:stratified}

Let $X$ and $Y$ be two varieties equipped with affable stratifications, denoted $\cS$ and $\cT$, respectively.  In this section, we will study functors arising from morphisms $f: X \to Y$ that respect the stratifications, in the following sense.

\begin{defn}\label{defn:stratified}
Assume $X$ and $Y$ have affine even stratifications.  A morphism $f: X \to Y$ is called a \emph{stratified morphism} if the following two conditions hold:
\begin{enumerate}
\item For each stratum $Y_t \subset Y$, its preimage $f^{-1}(Y_t) \subset X$ is a union of strata.\label{it:strat-preim}
\item For each point $y \in Y_t$, the collection of spaces\label{it:strat-fiber}
\[
\cS_y = \{ X_s \cap f^{-1}(y) \mid X_s \subset f^{-1}(Y_t) \}
\]
constitutes an affine even stratification of $f^{-1}(y)$.  Moreover, there is an isomorphism
\[
f^{-1}(y) \times Y_t \simto f^{-1}(Y_t)
\]
that restricts to an isomorphism $(X_s \cap f^{-1}(y)) \times Y_t \simto X_s$ for each $X_s$, and such that the composition
\[
f^{-1}(y) \times Y_t \simto f^{-1}(Y_t) \ovto{f} Y_t
\]
is just projection onto the second factor.
\end{enumerate}
If $X$ and $Y$ have only affable stratifications, then $f: X \to Y$ is called a \emph{stratified morphism} if both stratifications admit simultaneous affine even refinements that make $f$ stratified in the above sense.
\end{defn}

This definition is very close to the one originally introduced by Goresky and MacPherson~\cite[Definition~1.2]{gm}.  Note that part~\eqref{it:strat-preim} is simply the definition of a weakly stratified morphism.

The remainder of the section is devoted to studying proper stratified morphisms.  We begin by giving a useful alternate characterization of pure miscible objects.

\begin{defn}\label{defn:sterile}
An object $\cF \in \dsw(X)$ is said to be \emph{sterile} of weight $w$ if for all $s \in \cS$, the objects $j_s^*\cF$ and $j_s^!\cF$ are pure and semisimple of weight $w$.
\end{defn}

It is immediate from the definition that a sterile object of weight $w$ is pure of weight $w$.

\begin{prop}\label{prop:sterile-miscible}
Every sterile object is miscible, and therefore semisimple.
\end{prop}
\begin{proof}
Let $\cF \in \dsw(X)$ be sterile of weight $w$.  We proceed by induction on the number of strata in $X$.  If $X$ consists of a single stratum, then $\cF$ is miscible because it is pure and semisimple by definition.  Otherwise, choose a closed stratum $j_t: X_t \to X$, and let $h: U \to X$ be the inclusion of the complementary open subset.  Then $h^*\cF \in \dsw(U)$ is sterile, and therefore miscible by induction.  It also follows that $h_!h^*\cF$ is miscible.  On the other hand, $j_t^*\cF$ is pure and semisimple, and so miscible, by definition.  Consider the distinguished triangle
\begin{equation}\label{eqn:sterile-miscible}
h_!h^*\cF \to \cF \to j_{t*}j_t^*\cF \ovto{\delta}.
\end{equation}
To show that $\cF$ is miscible, it suffices to show that $\delta[-1]: j_{t*}j_t^*\cF[-1] \to h_!h^*\cF$ is a miscible morphism.  Consider the distinguished triangle
\[
j_{t*}j_t^!h_!h^*\cF \to h_!h^*\cF \to h_*h^*\cF \to,
\]
which is miscible by Theorem~\ref{thm:open-closed}.  There is no nonzero morphism $j_{t*}j_t^*\cF[-1] \to h_*h^*\cF$, so $\delta[-1]$ factors through $j_{t*}j_t^!h_!h^*\cF \to h_!h^*\cF$.  We are therefore reduced to showing that the map $j_t^*\cF[-1] \to j_t^!h_!h^*\cF$ is miscible.  If we complete this map to a distinguished triangle, we simply obtain the diagram
\[
j_t^*\cF[-1] \to j_t^!h_!h^*\cF \to j_t^!\cF \to
\]
given by applying $j_t^!$ to~\eqref{eqn:sterile-miscible}.  Here, the first and last terms are miscible by the definition of a sterile object, and the middle term is miscible by Theorem~\ref{thm:open-closed}.  Therefore, the whole triangle is miscible by Theorem~\ref{thm:misc-tri}.
\end{proof}

The following lemma is a special case of Theorem~\ref{thm:smooth-proper} below.

\begin{lem}\label{lem:proper-pushforward}
Let $f: X \to \bA^m$ be a proper stratified morphism, where $\bA^m$ is endowed with the trivial stratification, denoted $\cT$.  Then, for $\cF \in \dsw(X)$, we have $f_*\cF \in \dw_\cT(\bA^m)$.  Moreover, the functor $f_*: \dsw(X) \to \dw_\cT(\bA^m)$ is miscible, and for a simple perverse sheaf $\ICm_s \in \Pervsm(X)$, we have that $\pH^k(f_*\ICm_s) = 0$ if $|k| > \dim X_s - m$.
\end{lem}
\begin{proof}
Assume that $\cS$ is an affine even stratification, and consider a simple object $\ICm_s \in \Pervsm(X)$.  In addition to the vanishing condition stated at the end of the lemma, we will prove that all nonzero $\pH^k(f_*\ICm_s)$ are in fact direct sums of copies of $\uQlb[m]\la k - m\ra$.  That assertion implies that $f_*\ICm_s$ lies in $\dw_\cT(\bA^m)$ and is miscible (by Lemma~\ref{lem:affine-basic}\eqref{it:misc-semis}).  In particular, it follows that $f_*$ preserves the Weil category and, by Proposition~\ref{prop:homog-genuine}, that $f_*: \dsw(X) \to \dw_\cT(\bA^m)$ is miscible.

We proceed by induction on the number of strata in $X$.  Let $j_t: X_t \to X$ be the inclusion of an open stratum, and let $i: Z \to X$ be the inclusion of the complementary closed subvariety.  If $X_s \subset Z$, then, by a slight abuse of notation, we may write  $f_*\ICm_s \cong (f \circ i)_*\ICm_s$.  Note that $f \circ i: Z \to \bA^m$ is also proper and stratified, so the lemma holds for $(f \circ i)_*$ by assumption.  Thus, $f_*\ICm_s$ has the required properties.

If $s = t$, on the other hand, we may form the distinguished triangle
\[
j_{t!}j_t^*\ICm_t \to \ICm_t \to i_*i^*\ICm_t \to.
\]
Let $n = \dim X_t$.  Applying $f_* \cong f_!$, we obtain
\begin{equation}\label{eqn:proper-pushf-genuine}
(f \circ j_t)_! \uQlb[n]\la -n\ra \to f_*\ICm_t \to (f \circ i)_*i^*\ICm_t \to.
\end{equation}
By Theorem~\ref{thm:open-closed}, $i^*\ICm_t$ is a miscible object of $\dsw(Z)$, so by induction, $(f \circ i)_*i^*\ICm_t$ is miscible.  We also know that $\pH^k(i^*\ICm_t) = 0$ for $k \ge 0$.  Moreover, for $k \le -1$, any composition factor of $\pH^k(i^*\ICm_t)$ is a simple perverse sheaf $\ICm_u$ with $\dim X_u < n$.  The cohomology vanishing for $(f \circ i)_*$ implies that
\[
\pH^k((f \circ i)_*i^*\ICm_t) = 0 \qquad\text{if $k \ge n - m - 1$.}
\]
Recall from Definition~\ref{defn:stratified} that there is an isomorphism $X_t \cong \bA^{n-m} \times \bA^m$ such that $(f \circ j_t): \bA^n \to \bA^m$ can be identified with projection onto the second factor.  It follows that
\[
(f \circ j_t)_! \uQlb[n]\la -n\ra \cong \uQlb[2m-n]\la n-2m\ra 
\cong (\uQlb[m]\la n - 2m\ra) [m-n].
\]
Thus, $\pH^k((f \circ j_t)_! \uQlb[n]\la -n\ra)$ vanishes except when $k = n - m$.  Now, forming the long exact sequence in perverse cohomology associated to~\eqref{eqn:proper-pushf-genuine}, we see that
\[
\pH^k(f_*\ICm_t) \cong
\begin{cases}
\pH^k((f \circ i)_*i^*\ICm_t) & \text{if $k < n -m -1$,} \\
\uQlb[m]\la n - 2m\ra & \text{if $k = n - m$,} \\
0 & \text{otherwise.}
\end{cases}
\]
By induction, each $\pH^k((f \circ i)_*i^*\ICm_t)$ is miscible and therefore semisimple.  But we already know that $f_*\ICm_t$ is pure of weight $0$, so in fact, each $\pH^k(f_*\ICm_t)$ must be a direct sum of copies of $\uQlb[m]\la k-m\ra$, as desired.  We have just seen that $\pH^k(f_*\ICm_t) = 0$ for $k > \dim X_t - m$, and the vanishing for $k < m - \dim X_t$ follows by Verdier duality.
\end{proof}

\begin{cor}\label{cor:proper-strat-weil}
Let $f: X \to Y$ be a proper stratified morphism.  If $\cF \in \dsw(X)$, then $f_*\cF \in \dw_\cT(Y)$.
\end{cor}
\begin{proof}
For each stratum $j_t: Y_t \to Y$, we know that $\cF|_{f^{-1}(Y_t)} \in \dsw(f^{-1}(Y_t))$, so by the previous lemma, the object
\[
j_t^*f_*\cF \cong (f|_{f^{-1}(Y_t)})_*(\cF|_{f^{-1}(Y_t)})
\]
lies in $\dsw(Y_t)$.  By Lemma~\ref{lem:weil-obj}, $f_*\cF \in \dw_\cT(Y)$.
\end{proof}

\begin{thm}\label{thm:smooth-proper}
If $f: X \to Y$ is a proper stratified morphism, then the functor $f_*: \dsw(X) \to \dw_\cT(Y)$ is miscible.  If $f$ is also smooth, then $f_*$ is genuine.
\end{thm}
\begin{proof}
We will show that $f_*$ takes any simple perverse sheaf on $X$ to a pure miscible object of the same weight on $Y$.  It will then follow by Proposition~\ref{prop:homog-genuine} that $f_*$ is miscible.  In general, the induced functor $f_*: \Pure(X) \to \Pure(Y)$ will \emph{not} be homogeneous, so we cannot use that same proposition to prove genuineness.  However, in the case where $f$ is also smooth, it has a right adjoint $f^!$ that is genuine by Proposition~\ref{prop:smooth-pullback}, so $f_*$ is genuine by Theorem~\ref{thm:adjoint-genuine}.

Let $j_t: Y_t \to Y$ denote the inclusion of a stratum in $Y$, and let $h: f^{-1}(Y_t) \to X$ denote the inclusion of its preimage in $X$.  In addition, let $f_0 = f|_{f^{-1}(Y_t)}: f^{-1}(Y_t) \to Y_t$.  Then $h^*$ and $h^!$ are miscible by Corollary~\ref{cor:loc-closed}, and $f_{0*}$ is as well, by Lemma~\ref{lem:proper-pushforward}.  It follows that the objects
\[
j_t^*f_*\ICm_s \cong f_{0*}h^*\ICm_s
\qquad\text{and}\qquad
j_t^!f_*\ICm_s \cong f_{0*}h^!\ICm_s
\]
are miscible.  We know that $f_*\ICm_s$ is pure of weight $0$.  Since $Y$ is endowed with an affine even stratification, it follows that all objects $j_t^*f_*\ICm_s$ and $j_t^!f_*\ICm_s$ are pure.  Since they are pure and miscible, they are semisimple, and so $f_*\ICm_s$ is sterile.  By Proposition~\ref{prop:sterile-miscible}, $f_*\ICm_s$ is miscible, as desired.
\end{proof}

\subsection{Other miscible functors}

It is reasonable to expect that $f_*$ is genuine for any proper stratified morphism $f$, regardless of whether it is smooth, but unfortunately, the authors do not know how to prove this statement.  Similar remarks apply to the following statement.

\begin{prop}\label{prop:tensorhom}
Suppose $X$ has an affable stratification.  Then the functors
\begin{align*}
\Lotimes&: \dsw(X) \times \dsw(X) \to \dsw(X), \\
\cRHom&: \dsw(X)^\op \times \dsw(X) \to \dsw(X)
\end{align*}
are miscible.
\end{prop}
\begin{proof}
For $\Lotimes$, let us assume that $\cS$ is an affine even stratification.  If $\cF, \cG \in \Pure(X)$, then for any stratum $j_s: X_s \to X$, we have
\begin{equation}\label{eqn:otimes-misc}
j_s^*(\cF \Lotimes \cG) \cong j_s^*\cF \Lotimes j_s^*\cG
\qquad\text{and}\qquad
j_s^!(\cF \Lotimes \cG) \cong j_s^*\cF \Lotimes j_s^!\cG.
\end{equation}
On a single stratum $X_s \cong \bA^{\dim X_s}$, we clearly have $\uQlb \Lotimes \uQlb \cong \uQlb$.  It follows that the tensor product of semisimple pure objects on $\bA^n$ is semisimple.  Since $j_s^*\cF$, $j_s^*\cG$, and $j_s^!\cG$ are pure and semisimple, \eqref{eqn:otimes-misc}~shows that $j_s^*(\cF \Lotimes \cG)$ and $j_s^!(\cF \Lotimes \cG)$ are pure and semisimple.  Thus, $\cF \Lotimes \cG$ is sterile, so it is miscible by Proposition~\ref{prop:sterile-miscible}.  Since the bifunctor $\Lotimes$ takes $\Pure(X) \times \Pure(X)$ to $\Pure(X)$, it is miscible by Proposition~\ref{prop:homog-genuine}\eqref{it:homog-misc}.

Finally, since we have a natural isomorphism $\cRHom(\cF,\cG) \cong \D(\cF \Lotimes \D\cG)$, the miscibility of $\cRHom$ follows from Proposition~\ref{prop:verdier} and the statement for $\Lotimes$.
\end{proof}

\subsection{Ind-varieties}
\label{ss:ind-var}

We conclude Part~\ref{part:sheaf} of the paper by explaining how to extend the above results to certain ind-varieties.  Let $X$ be an ind-variety over $\F_q$.  Let $\cS = \{ X_s \}_{s \in S}$ be a collection of disjoint locally closed ordinary (finite-dimensional) subvarieties of $X$ whose union is $X$.  Assume that the closure of each $X_s$ is the union of $X_s$ and finitely many other $X_t$'s.  In particular, each $\overline{X_s}$ is an ordinary, finite-dimensional variety, so it makes sense to form the perverse sheaf $\ICm_s \in \Pervsm(\overline{X_s})$.  We call $\cS$ an \emph{affine even stratification} or an \emph{affable stratification} if it restricts to such a stratification on each variety $\overline{X_s}$.

The index set $S$ is partially ordered by containment of closures: we say that $s \le t$ if $\overline{X_s} \subset \overline{X_t}$.  Whenever $s \le t$, we have an inclusion map $i_{s,t}: \overline{X_s} \hookrightarrow \overline{X_t}$, and if $s \le t \le u$, we clearly have
\[
i_{s,u} = i_{t,u} \circ i_{s,t}.
\]
These closed inclusion maps give rise to fully faithful push-forward functors
\[
i_{s,t*}: \scE(\overline{X_s}) \to \scE(\overline{X_t})
\]
where $\scE$ stands for one of the following eight categories:
\begin{equation}\label{eqn:cat-list}
\Pervsm,\quad
\Pure,\quad
\Pervsw,\quad
\Pervs,\qquad
\dsm,\quad
\dsmisc,\quad
\dsw,\quad
\ds.
\end{equation}
We define the corresponding categories on $X$ by taking inductive limits over $S$:
\[
\scE(X) = \tlimind_{\substack{\longrightarrow \\ S}} \scE(\overline{X_s})
\qquad\text{where $\scE$ comes from the list~\eqref{eqn:cat-list}.}
\]
Every object and every morphism in one of these inductive limit categories is ``supported'' on some finite-dimensional variety $\overline{X_s}$, and as a result, many results about the categories attached to $\overline{X_s}$ generalize to $X$ without any difficulty.  The following basic facts are straightforward to verify; we omit the proofs.

\begin{prop}
Let $X$ be an ind-variety with an affable stratification.
\begin{enumerate}
\item $\dsm(X)$, $\dsw(X)$, and $\ds(X)$ are triangulated categories. We also have $\dsmisc(X) \cong \pse\dsm(X)$.
\item $\Pervsm(X)$, $\Pervsw(X)$, and $\Pervs(X)$ are  the hearts of $t$-structures on $\dsm(X)$, $\dsw(X)$, and $\ds(X)$, respectively.  They are all finite-length categories.
\item $\Pure(X)$ is an Orlov category, and $\dsm(X) \cong \Kb\Pure(X)$.
\item $\Pervsm(X)$ is a mixed abelian category, and its mixed structure makes $\dsm(X)$ into a mixed version of $\ds(X)$.
\item If $\cS$ is an affine even stratification, then $\Pure(X)$ is Koszulescent.  As a consequence, $\dsm(X) \cong \Db\Pervsm(X)$, and $\Pervsm(X)$ is a Koszul abelian category. \qed
\end{enumerate}
\end{prop}

As an immediate consequence, all the miscibility and genuineness results proved in Section~\ref{sect:genuine} apply in the ind-variety setting.

\begin{rmk}
Note that in the ind-variety setting, when $\cS$ is an affine even stratification, $\Pervsm(X)$ need not have enough projectives.  (The proof of Theorem~\ref{thm:bgs-aff-even}\eqref{it:bgs-mix-proj} does not go through, as it involves induction on the number of strata.) 
\end{rmk}

\section{Mixed tilting sheaves}
\label{sect:tilting}

In this section, we consider only varieties with an affine even stratification.  Tilting perverse sheaves (whose definition is recalled below) on such a variety are certain objects that enjoy both the ``local'' nature of $\IC$ objects and the good $\Ext$-vanishing properties of projectives and injectives.  For basic properties and applications to flag varieties in the setting of $\Pervs(X)$, see~\cite{bbm}.  Similar statements in the setting of $\Pervsw(X)$ can be found in~\cite{yun}.

Here, we classify the indecomposable tilting perverse sheaves in $\Pervsm(X)$.  Under an additional assumption on the variety $X$, we prove that tilting perverse sheaves form a Koszulescent Orlov category.  As an application of the latter, we show how to strengthen Theorem~\ref{thm:smooth-proper}.

\subsection{Classification of tilting perverse sheaves}
\label{ss:tilt-pervsm}

We begin with the definition.  For the equivalence of the two conditions below, see~\cite[Proposition~1.3]{bbm}.

\begin{defn}
Let $X$ be a variety with an affine even stratification $\cS$.  A perverse sheaf $\cF$ (in any of $\Pervs(X)$, $\Pervsw(X)$, or $\Pervsm(X)$) is said to be \emph{tilting} if either of the following equivalent conditions holds:
\begin{enumerate}
\item For each stratum $j_s: X_s \to X$, both $j_s^*\cF$ and $j_s^!\cF$ are perverse sheaves.
\item $\cF$ admits both a standard filtration and a costandard filtration.
\end{enumerate}
\end{defn}

The next two statements are adapted from results in~\cite{bbm,yun}; we include proofs because $\Pervsm(X)$ differs in some details from $\Pervs(X)$ and $\Pervsw(X)$.

\begin{lem}\label{lem:tilt-indecomp}
Let $j_t: X_t \to X$ be the inclusion of a closed stratum, and let $h: U \to X$ be the inclusion of the complementary open subset.  Let $M \in \Pervsm(X)$ be a tilting perverse sheaf, and assume that the canonical morphism $j_t^!M \to j_t^*M$ vanishes.  Then $M$ is indecomposable if and only if $h^*M$ is indecomposable.
\end{lem}
\begin{proof}
Let $M_U = h^*M \in \Pervsm(U)$.  We can form two short exact sequences
\[
0 \to j_{t*}j_t^!M \ovto{p} M \to h_*M_U \to 0
\qquad\text{and}\qquad
0 \to h_!M_U \to M \ovto{q} j_{t*}j_t^*M \to 0.
\]
By assumption, we have $q \circ p = 0$.  Applying $\Hom(M,\cdot)$ to the first of these, we obtain an exact sequence
\[
0 \to \Hom(M,j_{t*}j_t^!M) \to \End(M) \to \Hom(M,h_*M_U) \to \Ext^1(M,j_{t*}j_t^!M).
\]
Note that $\Ext^1(M,j_{t*}j_t^!M) \cong \Ext^1(j_t^*M,j_t^!M) = 0$, since all $\Ext^1$-groups in the category $\Pervsm(X_t)$ vanish.  So we actually have a short exact sequence; rewriting it using the usual adjointness properties, we obtain
\[
0 \to \Hom(j_t^*M,j_t^!M) \ovto{\theta} \End(M) \to \End(M_U) \to 0.
\]
For $f: j_t^*M \to j_t^!M$, we have $\theta(f) = p \circ f \circ q$.  Since $q \circ p = 0$, the image of $\theta$ is a nil ideal in $\End(M)$.  It follows that $\End(M)$ is a local ring if and only if $\End(M_U)$ is a local ring.  In other words, $M$ is indecomposable if and only if $M_U$ is indecomposable.
\end{proof}

\begin{prop}\label{prop:tilt-exist}
Let $X$ be a variety with an affine even stratification $\cS$.  For each stratum $X_s$, there exists a unique (up to isomorphism) indecomposable tilting perverse sheaf $\Tm_s \in \Pervsm(X)$ whose support is $\overline{X_s}$ and whose restriction to $X_s$ is given by $\Tm_s|_{X_s} \cong \uQlb[\dim X_s]\la-\dim X_s\ra$.  Moreover, every indecomposable tilting perverse sheaf is isomorphic to some $\Tm_s\la n\ra$.
\end{prop}

\begin{rmk}
In general, the uniqueness statement for $\Tm_s$ does not hold in $\Pervsw(X)$, essentially because there may be a nonvanishing $\Ext^1$-group between perverse sheaves supported on a closed stratum.  See~\cite[Remark~2.2]{yun}.
\end{rmk}

\begin{proof}
We proceed by induction on the number of strata.  Let $j_t: X_t \to X$ be a closed stratum, and let $h: U \to X$ be the complementary open subvariety.  In the case where $s = t$, it is clear that $\Tm_t = \ICm_t$ is the unique indecomposable tilting perverse sheaf supported on $X_t$ up to Tate twist.

Suppose now that $X_s \subset U$.  By induction, there is a unique indecomposable tilting perverse sheaf $T \in \Pervsm(U)$ supported on $\overline{X_s} \cap U$ and satisfying $T|_{X_s} \cong \uQlb[\dim X_s]\la -\dim X_s\ra$.  Since $T$ admits both a standard filtration and a costandard filtration, both $h_!T$ and $h_*T$ are perverse sheaves.  Let $A$ and $B$ denote the kernel and cokernel, respectively, of the natural map $h_!T \to h_*T$, so that we have an exact sequence
\begin{equation}\label{eqn:tilt-ext2}
0 \to A \to h_!T \to h_*T \to B \to 0.
\end{equation}
Both $A$ and $B$ are supported on $X_t$.  By Theorem~\ref{thm:open-closed}, the group $\Ext^2_{\Pervsm(X)}(B,A)$ can be computed instead in $\Pervsm(X_t)$.  The latter is a semisimple category, so $\Ext^2(B,A) = 0$.  Therefore, there exists an object $\Tm_s \in \Pervsm(X)$ that fits into two short exact sequences
\begin{equation}\label{eqn:tms-tilt}
0 \to A \to \Tm_s \to h_*T \to 0
\qquad\text{and}\qquad
0 \to h_!T \to \Tm_s \to B \to 0.
\end{equation}
These sequences show that $j_{t*}j_t^!\Tm_s \cong A$ and $j_{t*}j_t^*\Tm_s \cong B$, so $\Tm_s$ is certainly tilting.  It is obvious from~\eqref{eqn:tilt-ext2} that the canonical map $A \to B$ vanishes, so by Lemma~\ref{lem:tilt-indecomp}, $\Tm_s$ is indecomposable.

Now, let $M \in \Pervsm(X)$ be an indecomposable tilting perverse sheaf that is not supported on $X_t$.  Then, by Lemma~\ref{lem:tilt-indecomp}, $M_U = h^*M$ is indecomposable, and by induction, there is some stratum $X_s \subset U$ and some $n \in \Z$ such that $M_U \cong h^*\Tm_s\la n \ra$.  Assume for now that $n = 0$.  Let us apply the functors $\Hom(M,\cdot)$ and $\Hom(\cdot,M)$, respectively, to the two short exact sequences~\eqref{eqn:tms-tilt}.  By the reasoning in the proof of Lemma~\ref{lem:tilt-indecomp}, we obtain two new short exact sequences:
\begin{gather*}
0 \to \Hom(j_t^*M, B) \to \Hom(M,\Tm_s) \to \Hom(M_U,h^*\Tm_s) \to 0,\\
0 \to \Hom(A, j_t^!M) \to \Hom(\Tm_s,M) \to \Hom(h^*\Tm_s,M_U) \to 0.
\end{gather*}
Fix an isomorphism $f_U: M_U \simto h^*\Tm_s$, and let $g_U = f_U^{-1}: h^*\Tm_s \to M_U$.  We can lift these to maps $\tilde f_U: M \to \Tm_s$ and $\tilde g_U: \Tm_s \to M$.  Note that $\tilde g_U \circ \tilde f_U \in \End(M)$ is a unit, since $h^*(\tilde g_U \circ \tilde f_U) = \id_{M_U}$.  Similarly, $\tilde f_U \circ \tilde g_U \in \End(\Tm_s)$ is a unit.  We conclude that $\tilde f_U$ and $\tilde g_U$ are isomorphisms.

We have shown that every indecomposable tilting perverse sheaf in $\Pervsm(X)$ is isomorphic to some $\Tm_s\la n\ra$.  The uniqueness of $\Tm_s$ follows.
\end{proof}

Consider the unmixed tilting perverse sheaves $\Tw_s = \exs(\Tm_s) \in \Pervs(X)$.  It is well known that
\[
\Ext^k(\Tw_s, \Tw_t) = 0\qquad\text{for all $k > 0$.}
\]
This is essentially a consequence of the fact that $\Ext^k(\Delta_s, \nabla_t) = 0$ for $k > 0$.  The analogous fact for tilting perverse sheaves in $\Pervsm(X)$ follows by Proposition~\ref{prop:mixed-version}.  This observation can be used to establish the following fact; see~\cite[Proposition~1.5]{bbm} for a proof.

\begin{prop}\label{prop:ktilt-dsm}
The natural functor $\Kb\Tilt(X) \to \Db\Pervsm(X) \cong \dsm(X)$ is an equivalence of categories.\qed
\end{prop}

\subsection{Tilting objects as an Orlov category}
\label{ss:tilt-orlov}

We now consider varieties on which tilting objects obey the constraint described below:

\begin{defn}
A variety $X$ with an affine even stratification is said to satisfy \emph{condition~(W)} if for any two strata $X_s, X_t \subset X$ with $X_t \subset \overline{X_s}$ and $t \ne s$, we have that $j_t^*\Tm_s$ has weights${}\ge 1$ and $j_t^!\Tm_s$ has weights${}\le -1$.
\end{defn}

The terminology is taken from~\cite{yun}, where it is shown that flag varieties and affine flag varieties satisfy condition~(W)~\cite[Theorem~5.3.1]{yun}.  The following is the main result of this section.

\begin{thm}\label{thm:tilt-kosorlov}
Let $X$ be a variety with an affine even stratification $\cS$ satisfying condition~(W), and let $\Tilt(X)$ be the additive category of tilting perverse sheaves in $\Pervsm(X)$.  For an indecomposable tilting perverse sheaf $\Tm_s\la n\ra$, let us put
\[
\deg \Tm_s\la n\ra = -n.
\]
With respect to this degree function, $\Tilt(X)$ is a Koszulescent Orlov category.
\end{thm}

We first require the following lemma, suggested by the remarks in~\cite[\S1.3]{yun}.

\begin{lem}\label{lem:tilt-tstruc}
Consider the following two subcategories of $\dsm(X)$:
\begin{align*}
\tD^{\le 0} &= \{ \cF \in \dsm(X) \mid \text{$j_s^*\cF$ has weights${}\ge 0$ for all strata $X_s$} \}, \\
\tD^{\ge 0} &= \{ \cF \in \dsm(X) \mid \text{$j_s^!\cF$ has weights${}\le 0$ for all strata $X_s$} \}.
\end{align*}
Then $(\tD^{\le 0}, \tD^{\ge 0})$ is a bounded $t$-structure on $\dsm(X)$.  Its heart is a finite-length category, and a set of representatives for the isomorphism classes of simple objects is
\[
\{ \Tm_s[n]\la -n\ra \mid \text{$s \in \cS$, $n \in \Z$} \}.
\]
\end{lem}
\begin{proof}
If $X$ consists of a single stratum, then it is clear that for $\cF \in \tD^{\le 0}$ and $\cG \in \tD^{\ge 0}$, we have $\Hom(\cF,\cG[-1]) = 0$.  Moreover, it follows from parts~\eqref{it:aff-hered} and~\eqref{it:wt-split} of Lemma~\ref{lem:affine-basic} that every $\cF \in \dsm(X)$ fits into a split distinguished triangle $\cF' \to \cF \to \cF'' \to $ with $\cF' \in \tD^{\le 0}$ and $\cF''[1] \in \tD^{\ge 0}$.  Thus, $(\tD^{\le 0}, \tD^{\ge 0})$ is indeed a $t$-structure.  Its heart consists precisely of pure objects of weight $0$ in $\dsm(X)$.  This is evidently a semisimple abelian category whose simple objects are precisely those of the form $\uQlb[n]\la-n\ra$.

In the general case, one sees that $(\tD^{\le 0}, \tD^{\ge 0})$ is a $t$-structure by induction on the number of strata and the formalism of gluing, made available by Theorem~\ref{thm:open-closed}.  It is clear from condition~(W) that the $\Tm_s[n]\la-n\ra$ lie in the heart of this $t$-structure.  In fact, they satisfy the stronger condition from~\cite[Corollaire~1.4.24]{bbd} characterizing objects arising from the ``intermediate-extension'' functor, so by~\cite[Proposition~1.4.26]{bbd}, these are precisely the simple objects in the heart.
\end{proof}

\begin{proof}[Proof of Theorem~\ref{thm:tilt-kosorlov}]
To prove that $\Tilt(X)$ is an Orlov category, we proceed by induction on the number of strata in $X$.  Choose a closed stratum $j_u: X_u \to X$, and let $h: U \to X$ be the inclusion of the complementary open subvariety.  Now, consider two indecomposable tilting perverse sheaves $\Tm_s\la n\ra$ and $\Tm_t\la m\ra$.  Recall that $h_*h^*\Tm_t$ is a perverse sheaf, since $h^*\Tm_t$ has a costandard filtration.  We therefore have a short exact sequence
\[
0 \to j_{u*}j_u^!\Tm_t\la m\ra \to \Tm_t\la m\ra \to h_*h^*\Tm_t\la m \ra \to 0.
\]
This gives rise to a short exact sequence
\begin{multline*}
0 \to \Hom(j_u^*\Tm_s\la n\ra, j_u^!\Tm_t\la m\ra) \to \Hom(\Tm_s\la n\ra, \Tm_t\la m\ra) \to\\
 \Hom(h^*\Tm_s\la n\ra, h^*\Tm_t\la m\ra) \to 0.
\end{multline*}
(The sequence is exact because $\Ext^1(j_u^*\Tm_s\la n\ra, j_u^!\Tm_t\la m\ra) = 0$.)  Suppose that $n \ge m$.  Consider first the case where $s = t = u$.  Then the last term vanishes, and the first term vanishes if $n > m$.  Now, consider the case where at least one of $s$ and $t$ is distinct from $u$.  If $s \ne t$ or if $n > m$, then the last term vanishes by induction.  For the first term, note that $j_u^*\Tm_s\la n\ra$ has weights${}\ge n$, and that $j_u^!\Tm_t\la m\ra$ has weights${}\le m$.  Moreover, at least one of these inequalities must be strict (since at least one of $s$ and $t$ is distinct from $u$). Since $n \ge m$, the first term above vanishes as well.

We conclude in all cases that $\Hom(\Tm_s\la n\ra, \Tm_t\la m\ra) = 0$ if $n > m$ or if $n = m$ and $s \ne t$, so $\Tilt(X)$ is an Orlov category.

Using Proposition~\ref{prop:ktilt-dsm}, we henceforth identify $\Kb\Tilt(X)$ with $\dsm(X)$.  To prove that $\Tilt(X)$ is Koszulescent, consider the abelian category $\Kos(\Tilt(X))$ of Proposition~\ref{prop:kosorlov-tstruc}.  According to that proposition, the simple objects in that category are of the form $\Tm_s[n]\la -n\ra$.  But these objects also lie in the heart of the $t$-structure of Lemma~\ref{lem:tilt-tstruc}, so we conclude that the two $t$-structures coincide:
\[
\Kos(\Tilt(X)) = \tD^{\le 0} \cap \tD^{\ge 0}.
\]
From the description in Lemma~\ref{lem:tilt-tstruc}, it is easy to see that $\Kos(\Tilt(X))$ contains the objects
\[
\Dm_s[n]\la-n\ra
\qquad\text{and}\qquad
\Nm_s[n]\la-n\ra,
\]
and that these objects satisfy graded versions of axioms~(1)--(6) of~\cite[\S3.2]{bgs}.  Then, the argument of~\cite[Theorem~3.2.1]{bgs} shows that $\Kos(\Tilt(X))$ has enough projectives (resp.~injectives), and that these objects admit standard (resp.~costandard) filtrations.  Finally, the argument of~\cite[Corollary~3.3.2]{bgs} shows that the realization functor $\real: \Db\Kos(\Tilt(X)) \to \dsm(X)$ is an equivalence of categories.  Thus, $\Tilt(X)$ is Koszulescent.
\end{proof}

\begin{prop}
Let $X$ and $Y$ be two varieties with affine even stratifications, denoted $\cS$ and $\cT$, and assume that both satisfy condition~(W).  If $f: X \to Y$ is a proper stratified morphism, then the functor $f_*: \dsw(X) \to \dsw(Y)$ is genuine.
\end{prop}
\begin{proof}
According to~\cite[Proposition~3.4.1]{yun}, $f_*$ takes each indecomposable tilting perverse sheaf in $\Pervsw(X)$ either to $0$ or to an indecomposable tilting perverse sheaf in $\Pervw_\cT(Y)$ of the same degree.  Since $f_*$ is already known to be miscible, the same statement holds with respect to $\Pervsm(X)$ and $\Pervm_\cT(Y)$.  In particular, $f_*$ restricts to a homogeneous functor of Orlov categories $\Tilt(X) \to \Tilt(Y)$.  In view of Proposition~\ref{prop:ktilt-dsm}, Proposition~\ref{prop:homog-genuine} applies to these categories as well, and we conclude that $f_*$ is genuine.
\end{proof}

\part{Applications to representation theory}
\label{part:applications}

\section{Ext-algebras of Andersen--Jantzen sheaves}
\label{sect:ajext}

Let $G$ be a semisimple algebraic group of adjoint type over $\C$.  Fix a Borel subgroup $B \subset G$ and a maximal torus $T \subset B$.  Let $X^*(T)$ denote the weight lattice of $T$, and let $X^*_+(T) \subset X^*(T)$ be the set of dominant weights with respect to $B$.  Any $\lambda \in X^*(T)$ determines a line bundle $\cL_\lambda$ on the flag variety $G/B$.  Let $\tcN$ denote the cotangent bundle of $G/B$, with projection map $\pi: \tcN \to G/B$.  We also have the \emph{Springer resolution} $\mu: \tcN \to \cN$, where $\cN$ is the nilpotent cone in the Lie algebra of $G$.  The \emph{Andersen--Jantzen sheaf} of weight $\lambda \in X^*(T)$ is the object
\[
A_\lambda = R\mu_*\pi^*\cL_\lambda
\]
in the bounded derived category $\Db\Coh^G(\tcN)$.  When $\lambda$ is dominant, the higher direct images $R^i\mu_*\pi^*\cL_\lambda$ vanish for $i > 0$ (see~\cite[Theorem~3.6]{aj:cir} for the strictly dominant case and~\cite[Theorem~2.4]{bro:lbcb} for the general case), so $A_\lambda$ is in fact a coherent sheaf.  For their role in the cohomology of quantized tilting modules, see~\cite{o:ekt,bez:ctm}.  The aim of this section is to calculate the $\Ext$-algebra of $A_\lambda$, using the mixed derived category of the affine Grassmannian for the Langlands dual group $\Gv$.  This section has benefitted from conversations with Victor Ostrik and David Treumann.

\subsection{Coherent sheaves on the nilpotent cone}
\label{ss:coh-nilp}

Let the multiplicative group $\Gm$ act on $\cN$ and on fibers of $\tcN$ by $(t,x) \mapsto t^{-2}x$.  This action commutes with the natural action of $G$ on both of these varieties, so we have an action of $G \times \Gm$.  Let $\Cohgm(\cN)$ and $\Cohgm(\tcN)$ denote the abelian categories of $(G \times \Gm)$-equivariant coherent sheaves on these two varieties.  For an object $\cF$ in one of the bounded derived categories $\Db\Cohgm(\cN)$ or $\Db\Cohgm(\tcN)$, let $\cF \la n \ra$ denote the object obtained from $\cF$ by twisting the $\Gm$-action by $z \mapsto z^n$.  We define graded $\Hom$-spaces by
\begin{equation}\label{eqn:homc}
\uHom(\cF,\cG) = \bigoplus_{n \in \Z} \Hom(\cF, \cG\la -n\ra).
\end{equation}
Graded $\Ext$-groups are defined analogously.  By endowing the line bundle $\cL_\lambda$ on $G/B$ with trivial $\Gm$-action, we may naturally regard the Andersen--Jantzen sheaves $A_\lambda$ as objects of $\Cohgm(\cN)$.

For $\lambda, \mu \in X^*_+(T)$, we write $\mu \le \lambda$ if $\lambda - \mu$ is a sum of positive roots, as usual.  For $\lambda \in X^*_+(T)$, let $\cD_{\le \lambda}$ (resp.~$\cD_{< \lambda}$, $\cD_\lambda$) denote the full triangulated subcategory of $\Db\Cohgm(\cN)$ generated by the objects $A_\mu\la n\ra$ with $\mu \le \lambda$ (resp.~$\mu < \lambda$, $\mu = \lambda$) and $n \in \Z$.  It follows from~\cite[Proposition~4(a)]{bez:qes} that for a fixed $\lambda$, the full additive subcategory consisting of direct sums of objects of the form $A_\lambda\la n \ra$ (for $n \in \Z$) is a semisimple abelian category.  Then, by~\cite[Lemma~3]{bez:qes}, $\cD_\lambda$ admits a unique $t$-structure whose heart $\cA_\lambda$ contains the $A_\lambda\la n \ra$.  $\cA_\lambda$ is a finite-length category, and the simple objects (up to isomorphism) are precisely the $A_\lambda\la n \ra$.

For $\lambda \in X^*_+(T)$, let $V_\lambda$ denote the irreducible $G$-representation of highest weight $\lambda$.  Regard it as a $(G \times \Gm)$-equivariant coherent sheaf on a point, with trivial $\Gm$-action.  Its pullback to $\tcN$ (resp.~$\cN$) is denoted $\cO_\tcN \otimes V_\lambda$ (resp.~$\cO_\cN \otimes V_\lambda$).  Let $\Dfr\Cohgm(\cN)$ denote the full subcategory of $\Db\Cohgm(\cN)$ generated by the objects $(\cO_\cN \otimes V_\lambda)\la n\ra$, known as the category of \emph{perfect complexes}.

\begin{prop}\label{prop:qes}
\begin{enumerate}
\item We have $\cO_\cN \otimes V_\lambda \in \cD_{\le \lambda}$.\label{it:free}
\item The projection functor $\Pi: \cD_{\le \lambda} \to \cD_{\le \lambda}/\cD_{< \lambda}$ induces an equivalence of categories $\cD_\lambda \overset{\sim}{\to} \cD_{\le \lambda}/\cD_{< \lambda}$.  Moreover, $\Pi(\cO_\cN \otimes V_\lambda)$ lies in $\Pi(\cA_\lambda)$ and is a projective cover of $\Pi(A_\lambda)$.\label{it:proj}
\item The realization functor $\real: D^b\cA_\lambda \to \cD_\lambda$
is also an equivalence of categories.  In particular, we have\label{it:ext}
\[
\uExt^k_{\cA_\lambda}(A_\lambda,A_\lambda) \simeq \uHom_{\scE}(A_\lambda, A_\lambda[k]),
\]
where $\scE$ is any of:
$\cD_\lambda$, $\cD_{\le \lambda}$, $\cD_{\le \lambda}/\cD_{< \lambda}$, $\Db\Cohgm(\cN)$.
\end{enumerate}
\end{prop}
\begin{proof}
This proposition is mostly a restatement of results of~\cite{bez:qes}.  It follows from the proof of~\cite[Proposition~4(a)]{bez:qes} that the object $\cO_\cN \otimes V_\lambda$ can be obtained by repeatedly taking extensions among various $A_\nu\la n \ra$ with $\nu$ a weight of $V_\lambda$.  By~\cite[Proposition~3]{bez:qes}, we may assume that all the required $\nu$'s are dominant weights of $V_\lambda$, so part~\eqref{it:free} of the proposition follows, as does the fact that $\Pi(\cO_\cN \otimes V_\lambda) \in \Pi(\cA_\lambda)$.  Next, the fact that $\Pi$ induces an equivalence as in part~\eqref{it:proj} is simply a $\Gm$-equivariant analogue of~\cite[Lemma~4(d)]{bez:qes}.  That lemma also says that the inverse equivalence factors through a right adjoint $\Pi^r: \cD_{\le \lambda}/\cD_{< \lambda} \to \cD_{\le \lambda}$ to $\Pi$.  Therefore,
\begin{equation}\label{eqn:proj}
\uHom^i(\Pi(\cO_\cN \otimes V_\lambda), \Pi(A_\lambda)) \simeq \uHom^i(\cO_\cN \otimes V_\lambda, \Pi^r(\Pi(A_\lambda)) \simeq
\uHom^i(\cO_\cN \otimes V_\lambda, A_\lambda).
\end{equation}
Since the last term vanishes when $i = 1$, we see that $\Pi(\cO_\cN \otimes V_\lambda)$ is a projective object in $\Pi(\cA_\lambda)$.  Moreover, it follows from~\cite[Fact~1(a)]{bez:qes} (see also~\cite[Equation~(27)]{bez:qes}) that $\dim \uHom(\Pi(\cO_\cN \otimes V_\lambda), \Pi(A_\lambda)) = 1$.  Since $\Pi(A_\lambda)$ is, up to Tate twist, the unique simple object of $\Pi(\cA_\lambda)$, it follows that $\Pi(\cO_\cN \otimes V_\lambda)$ is the projective cover of $\Pi(A_\lambda)$.

Finally, we see from~\eqref{eqn:proj} that the functor $\uHom^i(\Pi(\cO_\cN \otimes V_\lambda), \cdot)$ vanishes on $\Pi(\cA_\lambda)$ for all $i > 0$.  It follows that $\uHom^i(\cdot,\cdot)$ can be computed on $\cA_\lambda$ by taking projective resolutions in the first variable.  The equivalence in part~\eqref{it:ext} follows.
\end{proof}

\begin{cor}\label{cor:qes}
For $X \in \cD_{\le \lambda}$, we have $X \in \cD_\lambda$ if and only if $\Hom^i(\cO_\cN \otimes V_\mu\la n \ra, X) = 0$ for all $\mu < \lambda$ and all $i, n \in \Z$.
\end{cor}
\begin{proof}
Suppose $X \notin \cD_\lambda$.  By~\cite[Lemma~4(b)]{bez:qes}, we know that there is some object $Y \in \cD_{<\lambda}$ such that $\Hom(Y,X) \ne 0$.  Since $\cD_{<\lambda}$ is generated by the $A_\mu\la n\ra$ with $\mu < \lambda$, there is some $\mu < \lambda$ and some $i, n \in \Z$ such that $\Hom^i(A_\mu\la n\ra, X) \ne 0$.  Suppose that $\mu$ is chosen to be minimal with this property, i.e.,
\begin{equation}\label{eqn:qes-perf}
\Hom^j(A_\nu\la m\ra, X) = 0 \qquad\text{if $\nu < \mu$.}
\end{equation}
By a repeated use of~\cite[Lemma~4(e)]{bez:qes}, the inclusion functor $\iota_\mu: \cD_{\le \mu} \to \cD_{\le \lambda}$ has a right adjoint $\iota_\mu^r$, so $\Hom^i(A_\mu\la n\ra, \iota_\mu^r X) \ne 0$.  It then follows from~\eqref{eqn:qes-perf} and~\cite[Lemma~4(b)]{bez:qes} that $\iota_\mu^r X \in \cD_\mu \cong \Db\cA_\mu$.  From the proof of the preceding proposition, we know that $\Pi^r\Pi(\cO_\cN \otimes V_\mu)$ is the projective cover of the unique (up to Tate twist) simple object in $\cA_\mu$.  Since $\iota_\mu^r X \ne 0$, there certainly exist $i, n \in \Z$ such that $\Hom^i(\Pi^r\Pi(\cO_\cN \otimes V_\mu)\la n\ra, \iota_\mu^r X) \ne 0$.  Next, there is a distinguished triangle $Y \to \cO_\cN \otimes V_\mu \to \Pi^r\Pi(\cO_\cN \otimes V_\mu) \to $ with $Y \in \cD_{<\mu}$.  Using~\eqref{eqn:qes-perf} once again, we see that $\uHom^\bullet(Y, \iota_\mu^r X) = 0$, so we then have
\[
\Hom^i(\cO_\cN \otimes V_\mu\la n\ra, X) \cong
\Hom^i(\cO_\cN \otimes V_\mu\la n\ra, \iota_\mu^r X) \ne 0.
\]
Finally, the opposite implication is clear: if $X \in \cD_\lambda$, then $\Hom^i(\cO_\cN \otimes V_\mu\la n \ra, X) = 0$ for $\mu < \lambda$ by~\cite[Lemma~4(b)]{bez:qes} and Proposition~\ref{prop:qes}\eqref{it:free}.
\end{proof}

Let $H'_\lambda$ denote the graded ring $\uEnd(\Pi(\cO_\cN \otimes V_\lambda))$.  The category $H'_\lambda\Mod$ of graded $H'_\lambda$-modules is endowed with a shift-of-grading functor, also denoted $X \mapsto X\la 1 \ra$.  A standard argument (see, for example,~\cite[Proposition~II.2.5]{ars}) yields the following result.

\begin{prop}\label{prop:morita}
There is an equivalence of categories $\cA_\lambda \simeq H'_\lambda\Mod$ that commutes with $X \mapsto X\la 1 \ra$ and that sends $A_\lambda$ to the trivial $H'_\lambda$-module.\qed
\end{prop}

\subsection{Mixed perverse sheaves on the affine Grassmannian}
\label{ss:mixed-gr}

Fix a prime $p$, and consider the field $\fK = \F_p((t))$ and its subring $\fO = \F_p[[t]]$.  By \emph{affine Grassmannian}, we mean the ind-variety $\Gr = \Gv(\fK)/\GvO$.  Recall that the choice of $B$ determines the Iwahori subgroup $\Iv \subset \GvO$.  The stratification of $\Gr$ by orbits of $\Iv$ is an affine even stratification, and the stratification by $\GvO$-orbits is affable.  When naming categories of constructible sheaves, the stratification will be indicated by a group as a subscript: for instance, $\Pervm_\GvO(\Gr)$ or $\dw_\GvO(\Gr)$.  It is well-known that the $\Iv$-orbits (resp.~$\GvO$-orbits) on $\Gr$ are parametrized by $X^*(T)$ (resp.~$X^*_+(T)$).  For $\lambda \in X^*_+(T)$, let $\Gr_\lambda$ denote the corresponding $\GvO$-orbit in $\Gr$. We also have the corresponding simple perverse sheaf $\ICm_\lambda \in \Pervm_\GvO(\Gr)$.

An important result due to Arkhipov--Bezrukavnikov--Ginzburg (see~\cite[Theorem~9.4.3]{abg}) is the construction of an equivalence of triangulated categories 
\begin{equation}\label{eqn:abg-equiv}
P: \dm_\Iv(\Gr) \to \Db\Cohgm(\tcN).
\end{equation}
This equivalence does not commute with Tate twist; instead, we have $P(\cF\la 1\ra) \cong (P\cF)\la 1\ra [1]$.  When $\lambda \in X^*_+(T)$, we have $P(\ICm_\lambda) \simeq \cO_\tcN \otimes V_\lambda$.  Define
\[
\Phi: \dm_\GvO(\Gr) \to \Db\Cohgm(\cN)
\qquad\text{by}\qquad
\Phi = R\mu_* \circ P|_{\dm_\GvO(\Gr)}.
\]
We claim that $\Phi$ induces an equivalence
\begin{equation}\label{eqn:perf-equiv}
\Phi: \dm_\GvO(\Gr) \simto \Dfr\Cohgm(\cN).
\end{equation}
Indeed, $P$ induces an equivalence between $\dm_\GvO(\Gr)$ and the full triangulated subcategory of $\Db\Cohgm(\tcN)$ generated by the objects $(\cO_\tcN \otimes V_\lambda)\la n\ra$.  Using the fact that $R\mu_*\cO_\tcN \cong \cO_\cN$ (see~\cite[Lemma~3.9]{aj:cir}), one easily checks that the functors $R\mu_*$ and $L\mu^*$ induce quasi-inverse equivalences between this category and $\Dfr\Cohgm(\cN)$.

Since $\cO_\cN \otimes V_\mu \in \cD_{\le \mu}$ for any $\mu$, $\Phi$ clearly restricts to a fully faithful functor $\dm_\GvO(\overline{\Gr_\lambda}) \to \cD_{\le \lambda}$.  This functor takes objects supported on the closed subvariety $\overline{\Gr_\lambda} \smallsetminus \Gr_\lambda$ to $\cD_{< \lambda}$.  As a consequence of Theorem~\ref{thm:open-closed} (see Remark~\ref{rmk:open-quot}), there is a natural equivalence
\[
\dm_\GvO(\overline{\Gr_\lambda})/\dm_\GvO(\overline{\Gr_\lambda} \smallsetminus \Gr_\lambda) \simeq \dm_\GvO(\Gr_\lambda)
\]
induced by restriction.  Thus, $\Phi$ gives rise to a functor $\Phi_\lambda: \dm_\GvO(\Gr_\lambda) \to \cD_{\le \lambda}/\cD_{<\lambda}$ that takes $\ICm_\lambda|_{\Gr_\lambda} \cong \uQlb[\dim \Gr_\lambda]\la -\dim \Gr_\lambda\ra$ to $\Pi(\cO_\cN \otimes V_\lambda)$. 

\begin{lem}\label{lem:ffaith}
The functor $\Phi_\lambda: \dm_\GvO(\Gr_\lambda) \to \cD_{\le \lambda}/\cD_{<\lambda}$ induced by $\Phi$ is fully faithful.
\end{lem}
\begin{proof}
Consider the essential image of the functor $j_*: \dm_\GvO(\Gr_\lambda) \to \dm_\GvO(\overline{\Gr_\lambda})$, where $j: \Gr_\lambda \to \overline{\Gr_\lambda}$ is the inclusion map.  The quotient functor $\dm_\GvO(\overline{\Gr_\lambda}) \to \dm_\GvO(\Gr_\lambda)$ induces an equivalence
\[
j_*(\dm_\GvO(\Gr_\lambda)) \simto \dm_\GvO(\Gr_\lambda).
\]
This statement is analogous to Proposition~\ref{prop:qes}\eqref{it:proj}.  In view of that, it suffices to show that $\Phi$ takes objects in $j_*(\dm_\GvO(\Gr_\lambda))$ to objects in $\cD_\lambda \subset \Db\Cohgm(\cN)$.  Consider an object $j_*\cG$, where $\cG \in \dm_\GvO(\Gr_\lambda)$.  Clearly, $\Hom^i(\ICm_\mu\la n\ra, j_*\cG) = 0$ if $\mu < \lambda$.  Because $\Phi$ is fully faithful, it follows that $\Hom^i(\cO_\cN \otimes V_\mu\la n\ra, \Phi(j_*\cG)) = 0$, so by Corollary~\ref{cor:qes}, we have $\Phi(j_*\cG) \in \cD_\lambda$.
\end{proof}

We are now ready for the main result of this section.  For $\lambda \in X^*_+(T)$, let $P_\lambda \subset G$ be the standard parabolic subgroup whose simple roots are orthogonal to $\lambda$, and consider the cohomology ring $H_\lambda = H^\bullet(G/P_\lambda)$.  This is a graded ring, so we can define graded $\Hom$- and $\Ext$-groups over it as in~\eqref{eqn:homc}.

\begin{thm}\label{thm:ajext}
There is an isomorphism of bigraded algebras
$\uExt^\bullet(A_\lambda, A_\lambda) \simeq \uExt^\bullet_{H_\lambda}(\C,\C)$.
\end{thm}
$H_\lambda$ can be described in terms of the coinvariant ring of the Weyl group~\cite[Theorem~5.5]{bgg:sccs}, so the result above can be used to carry out explicit calculations.  
\begin{proof}
In view of Propositions~\ref{prop:qes}\eqref{it:ext} and~\ref{prop:morita}, the proof of this statement reduces to showing that $H_\lambda \cong H'_\lambda$.  The fully faithful functor $\Phi_\lambda$ of Lemma~\ref{lem:ffaith} has the property that $\Phi_\lambda(\uQlb[\dim \Gr_\lambda]\la-\dim \Gr_\lambda\ra) \cong \Pi(\cO_\cN \otimes V_\lambda)$, so we have
\begin{multline*}
H'_\lambda = \uEnd(\Pi(\cO_\cN \otimes V_\lambda)) \cong \bigoplus_{n \in \Z} \Hom^n_{\dm_\GvO(\Gr_\lambda)}(\uQlb, \uQlb\la-n\ra) \\
\cong \bigoplus_{n \in \Z} H^n(\Gr_\lambda) \cong \bigoplus_{n\in \Z} H^n(G/P_\lambda) = H_\lambda,
\end{multline*}
where the third isomorphism follows from purity of the cohomology of the smooth variety $\Gr_\lambda$, and the fourth one from the fact that $\Gr_\lambda$ is naturally a vector bundle over $G/P_\lambda$.
\end{proof}

\section{Wakimoto sheaves}
\label{sect:wakimoto}

We retain the notation and conventions of the previous section, with the exception that we now allow $G$ to be an arbitrary connected reductive group.  Line bundles $\cL_\lambda$ on $\tcN$ (where now we allow any $\lambda \in X^*(T)$, not just dominant weights) form a particularly important class of objects in $\Db\Cohgm(\tcN)$, and it is natural to ask what objects in $\dm_\Iv(\Gr)$ they correspond to under the equivalence~\eqref{eqn:abg-equiv}.  In~\cite[Remark~9.4.4]{abg}, it was conjectured that line bundles should correspond to \emph{Wakimoto sheaves}, whose definition we will review below.  In fact, this was proved for $\lambda$ dominant or antidominant, and the analogous statement for the unmixed version of~\eqref{eqn:abg-equiv} (involving $\dc_\Iv(\Gr)$) was proved in general.  But in the mixed case, it was not known in~\cite{abg} whether Wakimoto sheaves for general $\lambda$ are miscible.  In this section, we provide a positive answer to this question.

\subsection{Twisted external tensor products and convolution products}
\label{ss:twisted}

Let $\Fl = \Gv(\fK)/\Iv$ be the \emph{affine flag variety} of $\Gv$.  As with $\Gr$, this is an ind-variety equipped with an affine even stratification given by orbits of $\Iv$, but these orbits are now indexed by the extended affine Weyl group $W = W_0 \ltimes X^*(T)$, where $W_0(T) = N_G(T)/T$ is the ordinary Weyl group.  To explain the construction of the convolution product, we require the \emph{equivariant derived category} of $\Fl$ in the sense of Bernstein--Lunts~\cite{bl}.  This category, denoted $\dw_{\eq\Iv}(\Fl)$, is a triangulated category equipped with a forgetful functor $\dw_{\eq\Iv}(\Fl) \to \dw_\Iv(\Fl)$, as well as with a $t$-structure whose heart $\Pervw_{\eq\Iv}(\Fl)$ is known as the category of \emph{equivariant perverse sheaves}.  When restricted to this abelian category, the forgetful functor $\Pervw_{\eq\Iv}(\Fl) \to \Pervw_\Iv(\Fl)$ is full and faithful.  

Now, consider the diagram
\[
\Fl \times \Fl \overset{p}{\longleftarrow} \Gv(\fK) \times \Fl \ovto{q} \Gv(\fK) \times_\Iv \Fl \ovto{m} \Fl
\]
where $p$ and $q$ are the obvious projection maps, and $m$ is the map induced by the action of $\Gv(\fK)$ on $\Fl$.  Suppose $\cF \in \dw_\Iv(\Fl)$ and $\cG \in \dw_{\eq\Iv}(\Fl)$. The \emph{twisted external tensor product} of $\cF$ and $\cG$, denoted $\cF \tboxtimes \cG$, is the unique object of $\dsw(\Gv(\fK) \times_\Iv \Fl)$ characterized by the property that
\[
q^*(\cF \tboxtimes \cG) \cong p^*(\cF \boxtimes \cG).
\]
Here, $\cS$ denotes the stratification whose strata are subvarieties of the form
\[
\Fl_w \ttimes \Fl_v = q(p^{-1}(\Fl_w \times \Fl_v))\qquad\text{where $w,v \in W$.}
\]
This construction actually gives us a bifunctor of triangulated categories
\[
\tboxtimes: \dw_\Iv(\Fl) \times \dw_{\eq\Iv}(\Fl) \to \dsw(\Gv(\fK) \times_\Iv \Fl).
\]
Finally, the \emph{convolution product} is the bifunctor
\[
\star: \dw_\Iv(\Fl) \times \dw_{\eq\Iv}(\Fl) \to \dw_\Iv(\Fl)
\qquad\text{given by}\qquad
\cF \star \cG = m_!(\cF \tboxtimes \cG).
\]
The $\Iv$-equivariance of $\cG$ is an essential ingredient in this construction; there is no way to make $\tboxtimes$ or $\star$ a bifunctor on $\dw_\Iv(\Fl) \times \dw_\Iv(\Fl)$ instead.  (It is possible to avoid equivariant derived categories at the expense of replacing one copy of $\Fl$ by the ``extended affine flag manifold''; see~\cite[\S 8.9]{abg}.)

\begin{lem}\label{lem:conv-aff-even}
The stratification $\cS$ of $\Gv(\fK) \times_\Iv \Fl$ is an affine even stratification.
\end{lem}
\begin{proof}
Given two strata $\Fl_v, \Fl_w \subset \Fl$, we have a functor $\tboxtimes : \dm_\Iv(\Fl_w) \times \dm_{\eq\Iv}(\Fl_v) \to \dw_\Iv(\Fl_w \ttimes \Fl_v)$ defined as above using the diagram
\[
\Fl_w \times \Fl_v \overset{p}{\longleftarrow} p^{-1}(\Fl_w \times \Fl_v) \ovto{q} \Fl_w \ttimes \Fl_v.
\]
Observe that
\begin{equation}\label{eqn:tbox-const}
\uQlb_{\Fl_w} \tboxtimes \uQlb_{\Fl_v} \cong \uQlb_{\Fl_w \ttimes \Fl_v}.
\end{equation}
If $\tilde\jmath_{w,v}: \Fl_w \ttimes \Fl_v \to \Gv(\fK) \times_\Iv \Fl$ denotes the inclusion map, it is straightforward to check that
\begin{equation}\label{eqn:tboxtimes-stalk}
\begin{aligned}
\tilde\jmath_{x,y}^*(\ICm_w \tboxtimes \ICm_v) &\cong j_x^*\ICm_w \tboxtimes j_y^*\ICm_v, \\
\tilde\jmath_{x,y}^!(\ICm_w \tboxtimes \ICm_v) &\cong j_x^!\ICm_w \tboxtimes j_y^!\ICm_v.
\end{aligned}
\end{equation}
In particular, it follows by a dimension calculation that
\begin{equation}\label{eqn:tboxtimes-ic}
\ICm_w \tboxtimes \ICm_v \cong \ICm_{w,v}.
\end{equation}
In view of~\eqref{eqn:tbox-const}, it follows now from~\eqref{eqn:tboxtimes-stalk} that $\cS$ is an affine even stratification of $\Gv(\fK) \times_\Iv \Fl$.  
\end{proof}

Since the inclusion map $j_w: \Fl_w \to \Fl$ is $\Iv$-equivariant, the object $\Dm_w = j_{w!}\uQlb[\dim \Fl_w]\la-\dim \Fl_w\ra$ can naturally be regarded as an object of $\dw_{\eq\Iv}(\Fl)$, so convolution products of the form $\cF \star \Dm_w$ are defined.

\begin{prop}\label{prop:conv-delta-misc}
If $s \in W$ is a simple reflection, the functor $({-}) \star \Dm_s: \dw_\Iv(\Fl) \to \dw_\Iv(\Fl)$ is miscible.
\end{prop}
\begin{proof}
Let $\Jv \subset \Gv(\fK)$ denote the standard parahoric subgroup corresponding to the simple reflection $s$.  Let $\Fl^s = \Gv(\fK)/\Jv$ be the associated partial affine flag variety, and let $\pi_s: \Fl \to \Fl^s$ denote the natural projection map.

If $\cG \in \dw_{\eq\Jv}(\Fl)$, then, by a construction using the diagram
\[
\Fl^s \times \Fl \longleftarrow \Gv(\fK) \times \Fl \longrightarrow \Gv(\fK) \times_\Jv \Fl \longrightarrow \Fl,
\]
one has a convolution product functor
\[
({-}) \star^s \cG: \dw_\Iv(\Fl^s) \to \dw_\Iv(\Fl).
\]
Similarly, if $\cG \in \dw_{\eq\Iv}(\Fl^s)$, then there is a convolution product functor
\[
({-}) \star_s \cG: \dw_\Iv(\Fl) \to \dw_\Iv(\Fl^s).
\]
It is straightforward to check that these new convolution products are associative in the appropriate sense.

For instance, consider the orbit closure $\overline{\Fl_s} \subset \Fl$, which is a single $\Jv$-orbit.  The object $\ICm_s[-1]\la 1\ra \cong \uQlb$ can be regarded as an object of $\dw_{\eq\Jv}(\Fl)$.  We claim that there is an isomorphism of functors
\begin{equation}\label{eqn:conv-pullback}
({-})\star^s \ICm_s[-1]\la 1\ra \cong \pi_s^*: \dw_\Iv(\Fl^s) \to \dw_\Iv(\Fl).
\end{equation}
To see this, we first note that the map $m: \Gv(\fK) \times_\Jv \overline{\Fl_s} \to \Fl$ is an isomorphism, since $\overline{\Fl_s}$ can be identified with $\Jv/\Iv$.  Let $r = \pi_s \circ m : \Gv(\fK) \times_\Jv \overline{\Fl_s} \to \Fl^s$.  To establish~\eqref{eqn:conv-pullback}, it suffices to show that $\cF \tboxtimes \uQlb \cong r^*\cF$.  But this follows from the observation that $q^*r^*\cF \cong p^*(\cF \boxtimes \uQlb)$, where $p$ and $q$ are the maps in the diagram
\[
\Fl^s \times \overline{\Fl_s} \overset{p}{\longleftarrow}
  \Gv(\fK) \times \overline{\Fl_s} \ovto{q} \Gv(\fK) \times_\Jv \overline{\Fl_s}.
\]

Next, consider the object $\pi_{s*}\ICm_e \in \dw_\Iv(\Fl^s)$, where $e \in W$ is the identity element.  This is a skyscraper sheaf on $\Fl^s$; it can certainly be regarded as an object of $\dw_{\eq\Iv}(\Fl^s)$.  An argument similar to (but easier than) the one above shows that there is an isomorphism of functors
\begin{equation}\label{eqn:conv-pushforward}
({-})\star_s \pi_{s*}\ICm_e \cong \pi_{s*}: \dw_\Iv(\Fl) \to \dw_\Iv(\Fl^s).
\end{equation}

Note that $(\pi_{s*}\ICm_e) \star^s \ICm_s[-1]\la 1\ra \cong \pi_s^*\pi_{s*}\ICm_e \cong \ICm_s[-1]\la 1\ra$.  By associativity of convolution products, we have
\[
({-}) \star \ICm_s[-1]\la 1\ra \cong ({-}) \star_s \pi_{s*}\ICm_e \star^s \ICm_s[-1]\la 1\ra \cong \pi_s^* \circ \pi_{s*}.
\]
The functors $\pi_s^*$ and $\pi_{s*}$ are genuine by Proposition~\ref{prop:smooth-pullback} and Theorem~\ref{thm:smooth-proper}, respectively, so the functor $({-}) \star \ICm_s[-1]\la 1\ra$ is as well.

The functor $({-}) \star \ICm_e \cong \id$ is obviously genuine as well.  Consider now the distinguished triangle $\ICm_s[-1]\la 1\ra \ovto{\theta} \ICm_e \to \Dm_s\la 1\ra \to$.  By a routine calculation involving the convolution products in~\eqref{eqn:conv-pullback} and~\eqref{eqn:conv-pushforward}, one can check that the morphism of functors
\[
({-}) \star \theta: ({-}) \star \ICm_s[-1]\la 1\ra \to ({-}) \star \ICm_e
\]
can be identified with the adjunction morphism $\pi_s^* \circ \pi_{s*} \to \id$.  

Suppose now that $\cF \in \dmisc_\Iv(\Fl)$.  We then have a distinguished triangle
\[
\cF \star \ICm_s[-1]\la 1\ra \ovto{\cF \star \theta} \cF \star \ICm_e \to \cF \star \Dm_s\la 1\ra \to.
\]
It follows from Lemma~\ref{lem:infext-ind-adjoint} that the adjunction morphism $\cF \star \theta$ is miscible, so its cone $\cF \star \Dm_s\la 1\ra$ is miscible as well, as desired.
\end{proof}

\begin{rmk}
In the course of the preceding proof, we saw that the functor $\cF \mapsto \cF \star \ICm_s$ is genuine.  By an induction argument on lengths of elements in $W$, one can deduce that the convolution product of any two simple perverse sheaves is a pure semisimple object of $\dw_\Iv(\Fl)$.  For another proof of this fact, see~\cite[Proposition~3.2.5]{by}.
\end{rmk}

\subsection{Wakimoto sheaves}
\label{ss:convolve}

Given a weight $\lambda \in X^*(T)$, choose two dominant weights $\mu, \nu \in X^*_+(T)$ such that $\lambda = \mu - \nu$.  All these weights can be regarded as elements of the affine Weyl group $W$, so they determine strata in $\Fl$.  The \emph{Wakimoto sheaf} of weight $\lambda$ is defined to be
\[
\cW_\lambda = \Nm_\mu \star \Dm_{-\nu}.
\]
This object is independent of the choice of $\mu$ and $\nu$; see~\cite[\S 8.3]{abg}.  Sometimes, the term \emph{Wakimoto sheaf} is instead used for the object
\[
\bar\cW_\lambda = \pi_*\cW_\lambda,
\]
where $\pi: \Fl \to \Gr$ is the natural projection map.  The following result answers a question posed in~\cite[Remark~9.4.4]{abg}.

\begin{prop}
The Wakimoto sheaves $\cW_\lambda \in \dw_\Iv(\Fl)$ and $\bar\cW_\lambda \in \dw_\Iv(\Gr)$ are miscible for all $\lambda \in X^*(T)$.
\end{prop}
\begin{proof}
Choose $\mu, \nu \in X^*_+(T)$ such that $\lambda = \mu - \nu$, and then choose a reduced expression $-\nu = ts_1s_2\cdots s_k$ in $W$, where $t$ is an element of length $0$, and the $s_i$ are simple reflections.  The stratum $\Fl_t$ is closed in $\Fl$, so $\Dm_t \cong \Nm_t$.  Therefore, $\Nm_\mu \star \Dm_t \cong \Nm_\mu \star \Nm_t \cong \Nm_{\mu t}$ (for the last step, see, e.g.,~\cite[Equation~(8.2.3)]{abg}).  The object $\Nm_{\mu t}$ is, of course, miscible, and then it follows from Proposition~\ref{prop:conv-delta-misc} and induction on $i$ that
\[
(\Nm_\mu \star \Dm_t) \star \Dm_{s_1} \star \Dm_{s_2} \star \cdots \star \Dm_{s_i}
\]
is miscible for each $i \in \{1,2,\ldots, k\}$.  Since $\Dm_t \star \Dm_{s_1} \star \cdots \star \Dm_{s_k} \cong \Dm_{-\nu}$, we conclude that $\cW_\lambda$ is miscible.  Lastly, since $\pi: \Fl \to \Gr$ is a smooth, proper stratified morphism, we have from Theorem~\ref{thm:smooth-proper} that $\bar\cW_\lambda$ is miscible as well.
\end{proof}

\end{document}